\documentclass{amsart} 

\newif\ifshort
\shortfalse

\ifshort
\usepackage[top=1.1in, bottom=1.1in, left=1.3in, right=1.3in]{geometry}
\fi

\usepackage[table]{xcolor}

\usepackage{sty2el}
\usepackage{graphics}

\usepackage[pdftex]{hyperref}
\hypersetup{colorlinks,citecolor=blue,linktocpage,hyperindex=true,backref=true}

\usepackage{array,longtable,float}

\usepackage{calc,tikz}
\usepackage{xcolor}
\usepackage{array,booktabs,colortbl}

\usepackage{todonotes}

\date{May 11, 2022} \title[Compactifications of moduli spaces of K3 surfaces]%
{Compactifications of moduli spaces of K3 surfaces
  with a nonsymplectic involution}

\author{Valery Alexeev}
\email{valery@uga.edu}
\address{Department of Mathematics, University of Georgia, Athens GA
  30602, USA}
\author{Philip Engel}
\email{philip.engel@uga.edu}
\address{Department of Mathematics, University of Georgia, Athens GA
  30602, USA}

\begin{document} 
\begin{abstract}
  There are $75$ moduli spaces $F_S$ of K3 surfaces with a
  nonsymplectic involution. We give detailed descriptions of Kulikov
  models for one-parameter degenerations in $F_S$.
  In the $50$ cases where the fixed locus of
  the involution has a component $C_g$ of genus $g\ge2$, we identify
  normalizations of the KSBA compactifications of $F_S$ via
  stable pairs $(X,\epsilon C_g)$, with explicit semitoroidal
  compactifications of $F_S$.
\end{abstract}

\maketitle 

\ifshort 
\setcounter{tocdepth}{1}
\else
\setcounter{tocdepth}{2}  
\fi

\tableofcontents

\section{Introduction}
\label{sec:intro}

Let $X$ be a smooth projective K3 surface. An involution $\iota$ of
$X$ is called \emph{nonsymplectic} if it acts as
$\iota^*\omega_X = -\omega_X$ on a generator of
$H^{2,0}(X)$.  The $(+1)$-eigenspace of the induced involution
on $H^2(X,\bZ)$ is a hyperbolic lattice~$S$. All the possibilities for
$S$ were found in a classical work
\cite{nikulin1979integer-symmetric} of Nikulin, who proved that there
are $75$ cases, given in \fign. The lattices $S$ are uniquely
determined by a triple of invariants $(r,a,\delta)$, or an equivalent
set of invariants, $(g,k,\delta)$.

For a given lattice $S$, there is a moduli space $F_S$ of K3 surfaces with
an involution and generic Picard lattice $S$. It is an open subset of
$\bD_S/\Gamma$, the quotient of a type IV domain of dimension $20-\rank S$
by an arithmetic group.

The K3 surfaces that appear include many interesting ones, for example
the double covers of: Enriques surfaces, smooth del Pezzo surfaces, log
del Pezzo surfaces of index~$2$, index~$2$ Halphen pencils, and rational
elliptic surfaces. They have been the subject of a
great deal of research. Here, we are interested in geometric
compactifications of the moduli spaces $F_S$. 

In $50$ of the $75$ cases, the fixed locus $R$ of the involution
contains a smooth curve $C_g$ of genus $g\ge2$. The divisor $C_g$ is
semiample and defines a contraction $X\to \oX$ to a K3 surface with
$ADE$ singularities and an ample Cartier divisor $\oC_g$.

It follows that for $0<\epsilon\ll1$ the pair $(\oX, \epsilon\oC_g)$
is a KSBA stable pair, see \cite{kollar2023families-of-varieties} for
their general theory. Stable pairs have complete, projective moduli
spaces. One thus obtains a geometrically meaningful KSBA
compactification $\oF_S$.

\begin{problem}\label{prob:describe-compactification}
  Describe the compactification $\oF_S$ explicitly.
\end{problem}

In previous collaborations, we solved this problem in two cases:
for the degree~2 K3 surfaces \cite{alexeev2019stable-pair} and for
the elliptic degree~2 K3 surfaces
\cite{alexeev2020compactifications-moduli}, which is a Heegner divisor
in the previous case. In this paper, we solve it for all the remaining
cases:

\begin{theorem}\label{thm:main-compactifications}
  The normalization of $\oF_S$ is a semitoroidal compactification of
  $\bD_S/\Gamma$ for an explicit semifan $\fF\dram$. 
  In $48$ of the $50$ cases it is dominated by a toroidal
  compactification for a Coxeter reflection fan.
\end{theorem}

The precise statement is given in Theorem~\ref{thm:compact-moduli}. 
We reduce Problem~\ref{prob:describe-compactification} to the
following problem which we solve \emph{for all $75$ cases:}

\begin{problem}\label{prob:kulikov-models}
  For each cusp of the Baily-Borel compactification $\oF_S\ubb$ 
  and each one-parameter degeneration $C\setminus 0\to F_S$
  approaching that cusp, describe explicitly a Kulikov
  model $X\to (C,0)$ adapted to the ramification divisor~$R$.
\end{problem}

A {\it Kulikov model} is a $K$-trivial SNC model $X\to (C,0)$ with smooth total
space, and it is {\it adapted to $R$} if $R_t\subset X_t$ for $t\in C\setminus 0$ extends to
$X_0$ as a divisor not containing singular strata, and the limit of any component
of positive genus is nef.

The answer to the last problem can be read off
directly from the Coxeter diagrams of the reflection groups of the
hyperbolic lattices $\oT$ appearing at the $0$-cusps of $F_S$. The
reason for this is quite simple. The main tool  we use is the
mirror symmetry between degenerations in the $S$-family and nef line
bundles on mirror K3 surfaces $\hX$ with Picard lattice
$\hS=\oT$. The nef cone of a K3 surface depends on, and is described
by the reflection group of its Picard lattice.

\medskip

The structure of the paper is as follows.  In Section~\ref{sec:2-el}
we recall the theory of K3 surfaces with a nonsymplectic involution,
of $2$-elementary lattices, and the general facts about the moduli
spaces $F_S$ of K3 surfaces with group action. We
also recall basic facts about the combinatorial (Baily-Borel,
toroidal, semitoroidal) and functorial (KSBA) compactifications of these
moduli spaces.

In Section~\ref{sec:reflection-groups} we recall Vinberg's theory of
reflection groups in hyperbolic spaces and the Coxeter-Vinberg
diagrams. We don't need the Coxeter groups for all the $75$ of the
$2$-elementary lattices but only for those that appear at the
$0$-cusps of $F_S$ for some $S$. These are the lattices with $g\ge1$,
excluding $(r,a,\delta)=(14,6,0)$. For the lattices with $g\ge2$ the Coxeter
diagrams were computed by Nikulin in \cite{alexeev2006del-pezzo}. A
few cases on the $g=1$ line were previously known. We complete the job
for the remaining lattices, the answer is given in
Figs.~\ref{fig:coxeter1} and \ref{fig:coxeter2}.

In Section~\ref{sec:Y-nef-cone} we describe K3 surfaces appearing in
the $75$ families, their quotients by the involution, and their nef
cones. The $75$ families are woven together in a web by certain
``Heegner divisor moves,'' corresponding to when one moduli
space $F_{S'}$ is a Heegner divisor in or at the boundary of $F_S$.
 We describe these moves and their
properties.

In Section~\ref{sec:cusps} we completely describe the $0$- and
$1$-cusps of $F_S$ together with the incidence relations between them.
In particular, the $0$-cusps of $F_S$ are described by three kinds of
``mirror moves'' on the nodes of \fign, making it into a directed graph
in which every vertex has in- and out-degrees equal to $0$, $1$, $2$, or
$3$.

In Section~\ref{sec:kulikov+ias} we discuss the theory of integral-affine
spheres ($\ias$) in relation to Kulikov models.
It is well known that the dual graph $\Gamma(X_0)$ of a Type III
Kulikov central fiber $X_0=\cup V_i$ is a triangulation of $S^2$.
In simple terms, the integral affine sphere $B=\Gamma(X_0)$ is an
economical description for $X_0$. The singularities of the $\ias$
describe the nontoric components $V_i$. The same integral-affine
structures describe a Lagrangian torus fibration $\mu\colon \hX\to B$ on a
mirror K3 surface $\hX$ with a symplectic form, e.g. given by an ample
line bundle.

In Section~\ref{sec:mirror-with-involution} we study this mirror
correspondence specifically for K3 surfaces with a nonsymplectic
involution. To understand the K\"ahler geometry of $X$, encoded by
the divisor $R$, we must understand the complex geometry of $\hX$.
The key is a special degeneration of $\hX$ into two copies $\hX_0 = \hY\cup \hY$
of the surface $\hY=\hX/\hiota$, the quotient of $\hX$ by the mirror
involution. This applies to all the cases except for the Enriques
case, where the answer is even more interesting: $\hY$ is an Halphen
pencil, and the gluing is by a $2$-torsion
twist on the multiple fiber.

The resulting $\ias$ is of a particularly simple kind:
$B=P\cup P\uopp$, the union of two isomorphic ``hemispheres'', 
Symington polytopes for $\hY$ glued along a circular equator representing
an anticanonical boundary of $\hY$.
 We prove that the mirror correspondence exchanges the
$(\pm1)$-eigenspaces on the lattices $H^2(X,\bZ)$, $H^2(\hX,\bZ)$
modulo the vanishing cycle, resp. the fiber class of the Lagrangian
torus fibration.

In Section~\ref{sec:kulikov-models-all}, for each lattice $\oT$
appearing at a $0$-cusp of $F_S$, and each monodromy invariant
$\lambda\in\oT$ encoding the Picard-Lefschetz transform of a one-parameter
degeneration, we construct explicitly the families of Kulikov surfaces with involution
that appear, up to taking some multiple of $\lambda$.

In Section~\ref{sec:compact-moduli} we first define the semifans appearing in
Theorem~\ref{thm:main-compactifications}. Next, we compute the
stable models for the Kulikov surfaces of
Section~\ref{sec:kulikov-models-all}. Finally, we prove 
Theorem \ref{thm:main-compactifications}
by applying the general theory of \cite{alexeev2021compact}
and \cite{alexeev2021nonsymplectic}.

\begin{acknowledgements}
  The authors were partially supported by the NSF grants DMS-2201222
  and DMS-2201221 respectively.
\end{acknowledgements}

\section{K3 surfaces with involution and $2$-elementary lattices} 
\label{sec:2-el}

\subsection{K3 surfaces with a nonsymplectic involution}
\label{sec:k3-with-invo}

Let $X$ be a smooth projective complex K3 surface. An involution $\iota$ of
$X$ is called \emph{nonsymplectic} if it acts as
$\iota^*\omega_X = -\omega_X$ on a non-vanishing holomorphic two-form
$\omega_X\in H^{2,0}(X)$. It is well known that the quotient $Y=X/\iota$ is
either a rational or Enriques surface and that $X$ is algebraic.

The main invariant of the involution is the $(+1)$ eigenspace
$S=H^2(X,\bZ)^+$, a hyperbolic lattice of some rank $r$. Its
orthogonal complement $T=S^\perp=H^2(X,\bZ)^-$ in $H^2(X,\bZ)$ is a
lattice of signature $(2,20-r)$.
There is a canonical isomorphism $A_S=A_T$ between the discriminant
lattices $A_S=S^*/S$ and $A_T=T^*/T$.  The involution acts by
multiplication by $\pm1$ on $A_S$ and $A_T$ respectively. This implies
that $A_S=\bZ_2^a$ for some $a\ge0$.
Such lattices are called {\it $2$-elementary}.

Conversely, if $S\subset\lk={\rm II}_{3,19}=U^{\oplus 3}\oplus E_8^{\oplus 2}$
is a primitive $2$-elementary lattice and
$T=S^\perp$ 
then the involution $\rho$ of $\lk\otimes\bQ$ acting as $\pm1$ on
$S$ and $T=S^\perp$ respectively is an involution of $\lk$. If
$X$ is a K3 surface whose Picard lattice $S_X$ equals $S$ via some
marking $H^2(X,\bZ)\to\lk$ then by the Torelli theorem,
there exists a unique involution $\iota$ of $X$ such that $\iota^*=\rho$. 

An indefinite even
$2$-elementary lattice is uniquely determined by its signature and the
triple $(r,a,\delta)=({\rm rk}_{\bZ} S,\,{\rm rk}_{\bZ_2}A_S, \,\delta)$,
where $\delta\in\{0,1\}$ is an additional
invariant explained in the next section.
The $2$-elementary hyperbolic lattices admitting a primitive embedding
into $\lk$ were classified by Nikulin in
\cite[3.6.2]{nikulin1979integer-symmetric}. There are 75 lattices and
for each of them, an embedding into $\lk$ is unique up to $O(\lk)$.
The result is given in Fig.~\ref{fig:discforms}.

\begin{figure}[ht]
  \centering
  \includegraphics[width=5 in]{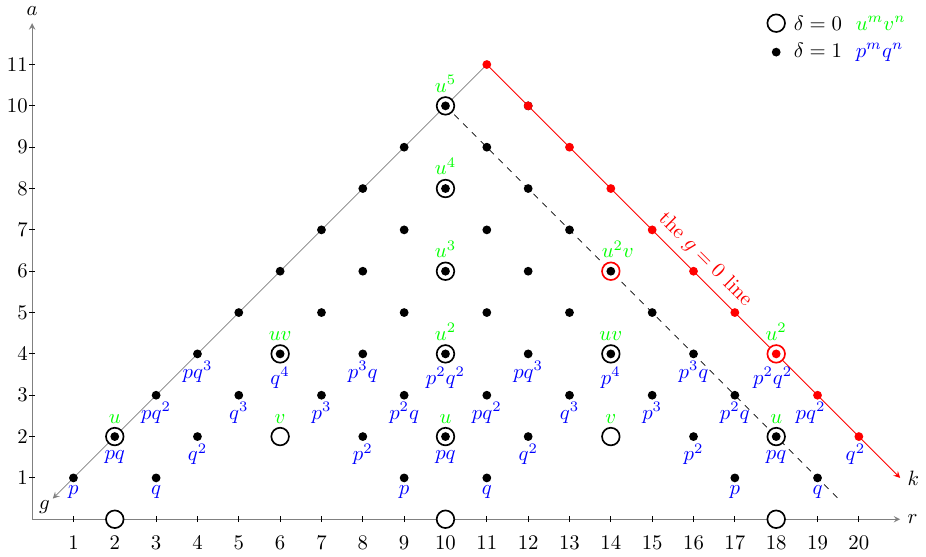}
  \caption{Hyperbolic $2$-elementary K3 lattices $(r,a,\delta)$}
  \label{fig:discforms}
\end{figure}

The fixed locus $R=X^\iota$ of the involution and the quotient
surface $Y=X/\iota$ are smooth. Denote by $k+1$ the number of the
irreducible components of $R$ and by $g$ the sum of their genera
(excluding the special Enriques case). Then
\begin{equation}\label{eq:g-k}
  g = 11 - \tfrac12 (r+a) \text{ and } k = \tfrac12 (r-a)
\end{equation}
and the triple $(g,k,\delta)$ is an alternative set of invariants of
$S$. There are three cases:
\begin{enumerate}
\item $R$ is a union of a curve $C_0$ of genus $g\ge0$ and $k$
  additional curves $C_1,\dotsc,C_k$ each isomorphic to $\bP^1$.
  The surface $Y$ is rational. \vspace{1pt}
\item $R=C_0\sqcup C_1$ is the union of two elliptic curves. Then
  $(r,a,\delta)=(10,8,0)$ and $(g,k,\delta) = (2,1,0)$.
  The surface $Y$ is rational elliptic. \vspace{1pt}
\item $R=\emptyset$. Then $(r,a,\delta)=(10,10,0)$ and
  $(g,k,\delta)=(1,0,0)$. The surface $Y$ is Enriques.
\end{enumerate}

\begin{remark}\label{rem:10-8-0}
  The case $(10,8,0)$ in many ways is more similar to those on the
  $g=1$ line than to those with $g\ge2$. For example, the automorphism
  group of a K3 surface with this Picard lattice is infinite.  Its
  Coxeter diagram, given in Fig.~\ref{fig:coxeter1}, is the same as
  for $(10,10,1)$.  See also Theorem~\ref{thm:nikulin-elliptic}.
\end{remark}

In the next sections, we briefly recall the theory of $2$-elementary
lattices and two ways of constructing the moduli spaces of K3 surfaces
with a nonsymplectic involution.

\subsection{$2$-elementary lattices}
\label{sec:2-elem-lattices}

 A \emph{lattice} $H$ is a free finite
rank $\bZ$-module together with a nondegenerate $\bZ$-valued bilinear
form. It is called \emph{even} if $x^2\equiv 0\pmod {2\bZ}$ for all $x\in H$ and
\emph{odd} otherwise. The \emph{discriminant lattice} is $A_H = H^*/H$, where
$H^*\subset H_\bQ$ is the dual lattice. It comes with the \emph{discriminant
  form} $q_H\colon A_H\to \bQ/\bZ$, $q(x) = x^2$. Moreover, if $H$ is
even then $q_H$ takes well-defined values in $\bQ/2\bZ$. One also has the
associated bilinear form $b_H\colon A_H\times A_H\to \bQ/\bZ$.

If $L$ is a unimodular lattice, $S\subset L$ a primitive
(i.e. saturated) sublattice and $T=S^\perp$ then $(A_S, q_S) = (A_T,
-q_T)$ in a canonical way.

\begin{num}\label{{num:invo-2el}}

If $S$ and $T$ are, respectively, the $(\pm 1)$-eigenspaces
of an involution $\iota$ on $L$ then $\iota$ acts as identity on $A_S$
and as $(-1)$ on $A_T$. Therefore, $A_S\simeq\bZ_2^a$ for some
$a$. Lattices $H$ with $A_H\simeq\bZ_2^a$ are called
\emph{$2$-elementary}. Thus, $S$ and $T$ are $2$-elementary in this
case.
\end{num}

\begin{definition}
  We define an additional invariant, \emph{coparity} $\delta_H$
  as follows: $\delta=0$ if for all $x\in A_H$ one has
  $q_H(x) \equiv 0\pmod\bZ$ and $\delta=1$ otherwise. We will call
  lattices with $\delta_H=0$ \emph{co-even} and lattices with
  $\delta_H=1$ \emph{co-odd}.
\end{definition}

This notation is explained by the following. Recall that for any
lattice $K$, $K(n)$ denotes the lattice with the bilinear product
scaled by $n$, i.e. $(x,y)_{K(n)} = n\cdot (x,y)_K$.

\begin{definition}\label{def:doubled-dual}
  For a $2$-elementary (not necessarily even) lattice $H$, the {\it doubled
  dual} is $H^\dag = H^*(2)$. The assignment $H\to H^\dag$ is an involution since
  $$(H^\dag)^\dag = \left(H^*(2) \right)^* (2) = H(\tfrac12)(2) = H.$$
\end{definition}

The doubled dual operation interchanges the parity and co-parity:

\begin{lemma}\label{lem:doubled-dual} 
  Let $H$ be a $2$-elementary lattice with invariants
  $(r,a,\delta)$. Then $H^\dag$ is a $2$-elementary lattice of the
  same signature with 
  invariants $(r,r-a,\delta^\dag)$ and the discriminant group is
  \begin{displaymath}\label{eq:discr-doubled-dual}
    A_{H^\dag} = (H^\dag)^* / H^\dag = H(\tfrac12) / H^*(2).
  \end{displaymath}
  Moreover, $\delta=0$ (resp.~$\delta=1$) iff $H^\dag$ is even (resp. odd),
  and $H$ is even (resp. odd) iff $\delta^\dag=0$ (resp.~$\delta^\dag=1$).
\end{lemma}
\begin{proof}
  For $x,y\in H^*$ one has $2x\in H$ so
  $(x,y)_{H^\dag} = 2(x,y) = (2x,y)\in\bZ$. So $H^\dag$ is indeed a
  $\bZ$-lattice. The equation for the discriminant group is immediate.
  Since $H$ is $2$-elementary, one has
  \begin{displaymath}
    H\subset 2H^*\subset \tfrac12 H    \implies
    2H^\dag = 2H^*(2) \subset \tfrac12 H(2) = H(\tfrac12),
  \end{displaymath}
  so $H^\dag$ is $2$-elementary of the same rank $r$, and the
  $\bZ_2$-rank of the discriminant group is $2r-(r+a) = r-a$. 
  For $x\in H^*$ one has $(x,x)_{H^*} = \frac12 (x,x)_{H^\dag}$, so
  $\delta_H=0$ iff $H^\dag$ is even. The last part holds by symmetry.
\end{proof}

We recall the following facts about $2$-elementary lattices proved in
\cite{nikulin1979integer-symmetric}. Any \emph{indefinite} even
$2$-elementary lattice $H$ is uniquely defined by its signature
$(n_+,n_-)$ and the invariants $(r,a,\delta)$, where $r=n_++n_-$ is
its rank, $a = \rk_{\bZ_2} A_H$ is the $\bZ_2$-rank of the discriminant
lattice, and $\delta$ is the coparity.  Moreover, the homomorphism
$O(H) \to O(q_H)$ from the isometry group to the isometry group of
$(A_H,q_H)$ is surjective. For definite $2$-elementary lattices,
the genus of the lattice is uniquely defined but there may be several
isomorphism
classes, cf. Section~\ref{sec:1-cusps}.

\begin{notation}
  Instead of writing ``a lattice $H$ of signature $(n_+,n_-)$ with
  invariants $(r,a,\delta)$'' we will simply write
  $H=(r,a,\delta)_{n_+}$. Moreover, for hyperbolic lattices, which are
  the majority of lattices in this paper, we will frequently omit the
  subscript $n_+=1$ and write simply $(r,a,\delta)$.
\end{notation}

The discriminant forms of even lattices were classified in
\cite{nikulin1979integer-symmetric}. For the even $2$-elementary
lattices they are direct sums of $p:=q_1(2)$, $q:=q_{-1}(2)$,
$u:=u(2)$ and $v:=v(2)$, which are the discriminant forms of the
lattices $\la 2\ra$, $\la -2\ra$, $U(2)$, $V(2)$:
\begin{displaymath}
  (2), \quad (-2), \quad
  \begin{pmatrix}
    0 & 2 \\ 2 & 0
  \end{pmatrix},
  \quad
  \begin{pmatrix}
    4 & 2 \\ 2 & 4
  \end{pmatrix}
\end{displaymath}
considered as lattices over the $2$-adic numbers.
Among them $u$
and $v$ are co-even, and $p$ and $q$ are co-odd. The values of $q$ in
$\bQ/2\bZ$, on $\bZ_2e^*$ and $\bZ_2e^*\oplus \bZ_2 f^*$, are \begin{align*}
&p(e^*)=\tfrac12,  && q(e^*)=-\tfrac12,   \\ 
&u(e^*)=u(f^*)=0, \,u(e^*+f^*)=1 &&v(e^*)=v(f^*)=v(e^*+f^*)=1. \end{align*}

We write the discriminant form for a direct sum of lattices multiplicatively.
The relations between the generators $p,q,u,v$ are generated by the identities
\begin{displaymath}
  u^2=v^2,\quad p^4=q^4,\quad up = (pq)p,\quad uq = (pq)q,\quad vp=q^3,\quad vq=p^3.
\end{displaymath}

The signature of a discriminant form is well defined mod $8$. For $p$,
$q$, $u$, $v$ it is $1$, $-1$, $0$, $4$ respectively.  This
makes it easy to compute the discriminant forms for all the cases. We
show some of them in \fign, enough to see the pattern. 

All of the lattices in \fign\ can be written as direct sums of 
the negative definite lattices $A_1$, $D_4$, $D_6$, $D_8$, $E_7$, $E_8$, $E_8(2)$, 
and hyperbolic lattices $\la 2\ra$, 
$U={\rm II}_{1,1}$, $U(2)$.
Their discriminant forms are as follows. For the co-even ones
$q(U)= q(E_8)=1$, $q(U(2))=q(D_8)=u$, $q(D_4)=v$,
$q(E_8(2))=u^4$; for the co-odd ones $q(A_1)=q$, $q(E_7)=q(\la 2\ra)=p$,
$q(D_6)=p^2$.

\subsection{Moduli of K3 surfaces with an involution}
\label{sec:moduli-with-involution}
The K3 surfaces with a nonsymplectic involution corresponding to a
given $2$-elementary lattice $S$ come with a natural moduli space. One
way to approach it is using the moduli of $S$-polarized K3 surfaces
following \cite{dolgachev1996mirror-symmetry}, as in
\cite{dolgachev2007moduli-of-k3}. The construction is a little delicate. 
Another, more direct approach applies to K3 surfaces with any finite
automorphism group that is not totally symplectic, see
\cite[Sec.~2A]{alexeev2021nonsymplectic}.

Fix an involution $\rho$ of $\lk$ with the $(\pm 1)$ eigenspaces $S$ and
$T$. A {\it $\rho$-marking} of a K3 surface with an involution $\iota$ is an
isometry $\phi\colon H^2(X,\bZ)\to\lk$ such that $\iota^* = \phi\inv
\circ \rho \circ \phi$.  Let  $\bD_S := \bP\{x\in T\otimes\bC \mid
 x\cdot x=0, \ x\cdot\bar x > 0 \}$ be the {\it period domain}.
 Then $\rho$-marked K3 surfaces with involution have a period
 $\phi(\bC\omega_X)\in \bD_S$.

One defines the discriminant
locus $\Delta = (\cup_\delta \delta^\perp)\cap\bD_S$, with $\delta$
ranging over the $(-2)$-vectors in $S$. The $\rho$-markings of
a K3 surface with involution are a torsor over
\begin{displaymath}
  \Gamma_\rho := \{ \gamma\in O(\lk) \mid \gamma\circ\rho = \rho\circ\gamma \}.
\end{displaymath}
Then $F_S=(\bD_S\setminus\Delta) / \Gamma_\rho$ is the coarse space
of K3 surfaces that admit a $\rho$-marking.

Recall that for the $2$-elementary indefinite lattices $S$ and $T$ the
homomorphisms $O(S)\to O(A_S,q_S)$ and $O(T)\to O(A_T,q_T)$ are
surjective and one of course has $O(A_S,q_S) = O(A_T,q_T)$. Thus,
$\bD_S/\Gamma_\rho=\bD_S/O(T)$ is the quotient
by the full group of isometries of $T$, and $F_S$ is the
complement of finitely many divisors in it.

Note that for the surfaces parameterized by $F_S$ the Picard
lattice $S_X$ could be bigger than $S$ but the $(+1)$-eigenspace
$S_X^+$ can be identified with $S$.

\subsection{Baily-Borel, toroidal, and semitoroidal compactifications}
\label{sec:bb-etc-compactifications}
This material is well known, so we refer to \cite{ash1975smooth-compactifications,
  looijenga2003compactifications-defined2} for details.
Let $\bD=\bD_S$ as above, and let
$\Gamma\subset O(T)$ be a finite index subgroup.
The {\it Baily-Borel compactification}
$$\bD/\Gamma\hookrightarrow \overline{\bD/\Gamma}\ubb$$
is a projective variety whose boundary
consists of finitely many points ($0$-cusps) and modular curves
($1$-cusps). The ``Type III" $0$-cusps (resp. ``Type II" $1$-cusps) are in a bijection with
$\Gamma$-orbits of primitive isotropic lines $I\subset T$ (resp. 
primitive isotropic planes $J\subset T$). The ``Type" terminology arises
from the Kulikov-Persson-Pinkham classification of K3 degenerations
\cite{kulikov1977degenerations-of-k3-surfaces,
  persson1981degeneration-of-surfaces}.

A {\it toroidal compactification} $\overline{\bD/\Gamma}^\fF$ is a
combinatorially defined normal variety specified by the data
$\fF=\{\fF_I\}$ of a compatible system of admissible fans for each
cusp. For a Type IV domain, the data for the $1$-cusps is trivial, so
the only important fans are for the $0$-cusps and they are always
compatible. The fan $\mathfrak{F}_I$ is a rational polyhedral
decomposition of the rational closure $\ocC_{I,\bQ}$ of the positive
cone $\cC_I\subset I^\perp/I\otimes \bR$. It is required to satisfy
the usual fan axioms, and additionally be $\Gamma$-invariant
with only finitely many orbits of cones.

A semitoroidal (or semitoric) compactification
$\overline{\bD/\Gamma}^\fF$ of Looijenga 
is a generalization in which the cones of $\fF_I$ are locally polyhedral,
but not necessarily finitely generated. The data for the $1$-cusps
and the compatibility condition may be nontrivial.
By \cite[Thm.~5.14]{alexeev2021compact},
a semitoroidal compactification is the same as a normal
compactification which may be sandwiched between the Baily-Borel and a
toroidal compactification.

\subsection{Stable pair compactifications}
\label{sec:ksba-compactifications}

We refer the reader to \cite{kollar2023families-of-varieties} for the
definition of slc singularities and the existence of the KSBA
compactifications of moduli spaces via KSBA stable pairs.  

In the case at hand, a {\it stable pair} $(X,\epsilon R)$ consists of a
seminormal surface $X$ with only slc singularities (in particular,
double normal crossing in codimension~$1$) with a trivial dualizing
sheaf and an ample Cartier divisor $R$ which does not
contain any log canonical centers of $X$. For $0<\epsilon\ll1$ this
pair is a KSBA stable pair, for all small enough $\epsilon$ bounded 
in terms $R^2$. For fixed $R^2$ there exists a projective moduli space
for such pairs.  For full details, see
\cite{alexeev2019stable-pair} and
\cite{alexeev2020compactifications-moduli}.

When $g\geq 2$, we denote by $\oF_S$ the closure of the pairs
$F_S=\{(\oX,\epsilon \oC_g)\}$ in the space of KSBA stable pairs.
One of the main goals of this paper is to
prove that $$(\oF_S)^\nu = \oF_S^\fF$$ 
for a particular semitoroidal compactification
for an explicit semifan $\fF=\{\fF_I\}$. Here $\nu$
denotes the normalization.

\section{Reflection groups}
\label{sec:reflection-groups}

One of the main tools in the study of K3 surfaces is reflection
groups. In this paper we apply it in two ways: to determine the nef
cones in Section~\ref{sec:nef-XY} and to describe certain toroidal
compactifications of Section~\ref{sec:bb-etc-compactifications}.

\subsection{Vinberg's theory}
\label{sec:vinberg-theory}

We refer to \cite{vinberg1973some-arithmetic,
  vinberg1972-groups-of-units} for Vinberg's theory of reflection
groups of hyperbolic lattices. We briefly describe it below.

\smallskip

Let $H$ be a hyperbolic lattice. Let $\cC$ the component of the set 
$\{ v\in H_\bR \mid v^2>0\}$, containing a fixed class $h$ with
$h^2>0$. Let $\cH=\bP\cC$ be the corresponding hyperbolic space.
A vector $v\ne0$ with $v^2=0$ in the closure of $\cC$ is a point on
the sphere at infinity of $\cH$. 

There are two kinds of closures of $\cC$, and it is always clear from the
context which one we have in mind. When $H=\Pic X$ for some surface
$X$, the nef cone $\Nef(X)$ is naturally a subset of the round cone
$\ocC = \{ v\in H_\bR \mid v^2\ge0\}$, so here we add all infinite
points of $\cH$.  When $\cC$ is used to define a (semi)fan $\fF$ for
some (semi)toroidal compactification, one considers the \emph{rational
  closure} $\ocC_{\bQ}$ instead, with only the rays $\bR_{\ge0}v$ spanned by
rational vectors $v$ added.

A \emph{reflection} in a root $\alpha\in H$ is the isometry
$w_\alpha(v) = v - \frac{2(\alpha,v)}{(\alpha,\alpha)} \alpha$.  One
must have $2\di(\alpha)\in(\alpha,\alpha) \bZ$ for it to be
well defined.
Let $W\subset O(H)$ be a group generated by some subset of
reflections. The most interesting cases are the groups
\begin{enumerate}
\item $W\dref$ generated by all reflections, and \vspace{1pt}
\item $W_2$ generated by the $(-2)$-reflections, in roots with
  $\alpha^2=-2$. 
\end{enumerate}

\begin{definition}\label{def:chambers}
  We denote by $\ch$ the fundamental chamber for $W$. Equivalently,
  one can treat it as the (possibly infinite) polyhedron $\bP\ch \subset
  \cH$. One has 
\begin{displaymath}
  \ch = \left\{ v\in\ocC \text{ or } \ocC_\bQ \mid
    (\alpha_i, v) \ge 0 
    \text{ for simple roots } \alpha_i \right\},
  \quad O(H) = W \ltimes \Sym(\ch).
\end{displaymath}
\end{definition}

The fundamental chamber is encoded in a Coxeter-Vinberg diagram
$\Gamma$. The vertices correspond to the simple roots $\alpha_i$ and the
edges show the angles between them as follows. Let
$g_{ij} = (\alpha_i,\alpha_j) / \sqrt{(\alpha_i,\alpha_i)
  (\alpha_j,\alpha_j)}$. One connects $i$ and $j$ by
\begin{itemize}
\item an $m$-tuple line if $g_{ij} = \cos\frac{\pi}{m+2}$. The
  hyperplanes $\alpha_i^\perp$, $\alpha_j^\perp$ intersect in $\cH$.
\item a thick line if $g_{ij}=1$. $\alpha_i^\perp$,
  $\alpha_j^\perp$ are parallel, meet at an infinite point of $\cH$.
\item a dotted line if $g_{ij} > 1$. $\alpha_i^\perp$,
  $\alpha_j^\perp$ do not meet in $\cH$ or its closure.
\end{itemize}

We identify a subset $V'\subset V(\Gamma)$ of vertices of $\Gamma$
with the induced subgraph $\Gamma'$. 
The faces of $\ch$ are of the form 
\begin{equation}\label{eq:ch-faces}
  F = \cap_{i\in \Gamma'} \alpha_i^\perp \cap\ch
\end{equation}
for the $\Gamma'$ which are elliptic or parabolic, i.e. corresponding
to a negative definite or negative semi-definite matrix.
This correspondence is bijective for elliptic
subdiagrams. But disjoint parabolic subdiagrams define the same ray
of~$\ch$.

The subgroup $W\subset O(H)$ has finite index iff the polyhedron
$\bP\ch$ has finite volume. One says $W$ has finite covolume.
In that case, rational vectors at infinity
correspond to maximal parabolic subdiagrams, of rank $\dim
H-1$. Otherwise, there may exist some $v$ for which the maximal parabolic
subdiagram has lower rank; for example it could be empty. 

\medskip

In a $2$-elementary even hyperbolic lattice, the roots are the
$(-2)$-vectors and the $(-4)$-vectors of divisibility $2$. In the
Coxeter diagram we denote the $(-2)$-vectors by transparent, white
vertices and the $(-4)$-vectors by filled, black vertices. 
In addition, when the hyperbolic lattice is interpreted as the Picard
lattice $S$ of a K3 surface with an involution, and the white vertices
as $(-2)$-curves on it, the double-circled vertices denote the
$(-2)$-curves which are fixed pointwise by the involution.
See Fig. \ref{fig:coxeter1} for some examples.

\medskip

For a K3 surface $X$ with $H=\Pic X$, its nef cone $\Nef(X)$ is
identified with $\fC_2\subset\ocC$,  described by the Coxeter
diagram $\Gamma_2$. This is the main object of our interest because it
appears in the Mirror Theorem~\ref{thm:EF-mirror}.

But in the most important case, when $H=S$ is a $2$-elementary lattice
lying on the $g=1$ line, the group $W_2\subset O(S)$ has infinite
index, unless $S=(19,1,1)$. Indeed, this is equivalent to
$|\Aut X|=\infty$, which holds by
\cite{nikulin1979quotient-groups}. The Coxeter diagram $\Gamma_2$ in
these cases is infinite. Working with the smaller, usually finite,
diagram $\Gamma_r$ instead is much more convenient.

For the $50$ lattices $S\ne(10,8,0)$ with $g\ge2$, $\Aut X$ is finite, and
$W_2$ has finite covolume, see Section~\ref{sec:vinberg-g>1} below.
But usually $\Gamma_2$ is enormous and $\gref$ is relatively small.
In Section~\ref{sec:vinberg-g=1} we compute Coxeter diagrams for the
lattices $S$ on the $g=1$ line and prove that for most of them $\wref$
has finite covolume and $\gref$ is finite.

\smallskip

In fact, there are many intermediate reflection groups between
$W_2$ and $\wref$:

\begin{definition}\label{def:refl-chambers}
  Let $V(\gref) = V_2 \sqcup V_4$ be the decomposition
  of the vertices of $\gref$ into the $(-2)$-roots (white) and the
  $(-4)$-roots (black).
  Consider a subset $\bB\subset V_4$ and let $\bB^c$ be the
  complement in $V(\gref)$, which therefore includes all of
  $V_2$.  We define two reflection subgroups of $\wref$:
  \begin{displaymath}
  W(\bB) = \la w_\alpha \mid \alpha\in\bB \ra
  \qquad
  W\dnor(\bB^c) = \la w_{g(\alpha)} \mid g\in \wref,\ \alpha\in\bB^c  \ra
  \end{displaymath}
  The latter is the minimal \emph{normal} subgroup of $\wref$
  generated by $W(\bB^c)$.
  Let $\fC_{W\dnor(\bB^c)}$ be the fundamental chamber for the action of
  $W\dnor(\bB^c)$ on $\ocC$.
\end{definition}

Two special cases are: $W(\emptyset) = \wref$ and $W(V_4) =
W_2$. 

\begin{lemma}\label{lem:refl-chambers}
  One has $\wref = W(\bB) \ltimes W\dnor(\bB^c)$.
\end{lemma}
\begin{proof}
  This follows by Proposition on p.2 of
  \cite{vinberg1983two-most}, which applies because roots in $V_4$
  have divisibility~$2$. 
  See also
  \cite[Prop.~2.4.1]{alexeev2006del-pezzo} for the case
  $\bB=V_4$. 
\end{proof}

\begin{corollary}\label{cor:refl-chambers}
 $W(\bB)$ acts on $\fC_{W\dnor(\bB^c)}$ with
  the fundamental chamber~$\fcref$. In particular,
  $W(V_4)$ acts on $\fC_2$ with the fundamental chamber $\fcref$.
\end{corollary}

Any $\lambda\in\fcref$ is also a vector in $\fC_2$, so
any elliptic subdiagram of $\gref$ can be translated into an elliptic
subdiagram of $\Gamma_2$.
We note the following
useful conversion rules, see Fig.~\ref{fig:conversion}.
\begin{equation}\label{eq:conversion}
  B_n(2)\to A_1^n \quad C_3 \to A_3 \quad C_n \to D_n \quad F_4 \to D_4
\end{equation}
\begin{figure}[htbp]
  \centering
  \includegraphics{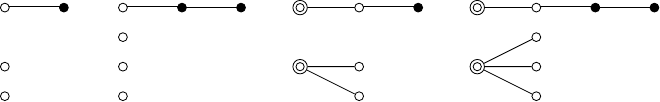}
   \caption{Conversion: $B_n(2)\to A_1^n$ ($n=2,3$), $C_3\to A_3$, $F_4\to D_4$}
  \label{fig:conversion} 
\end{figure}
The reverse direction is possible up to ``irrelevant'' walls formed by
connected diagrams consisting entirely of the $(-4)$-roots.

\subsection{Coxeter, or reflection semifan}
\label{sec:coxeter-semifan}

With the above notations:

\begin{definition}\label{def:coxeter-semifan}
  A Coxeter, or reflection semifan $\fF$ is a semifan with support on
  $\ocC_\bQ$ with the following cones: the fundamental chamber $\fC$,
  its faces, and their $W$-translates. It is a fan iff $W$ has finite
  covolume.
\end{definition}

In particular, we have the semifan $\fF_2$ for the Weyl group $W_2$
and its refinement, the semifan $\fref$ for the full reflection group
$\wref$.

For a semitoroidal compactification defined by $\fF$, the Type III
cones correspond to elliptic subdiagrams of the Coxeter diagram. If
$\fF$ is a fan then the Type II cones are the rays on the boundary,
corresponding to maximal parabolic subdiagrams of rank $\dim H-1$.

\subsection{Coxeter diagrams for lattices with $g\ge2$, excluding
  $(10,8,0)$}
\label{sec:vinberg-g>1}

For $50$ of the $75$ Picard lattices of \fign, namely those with
$g\ge2$ excluding $(10,8,0)$, the Coxeter diagrams $\gref$ were computed by
Nikulin in \cite[Table 1]{alexeev2006del-pezzo}. We recomputed and
confirmed them for this paper.

These are exactly the cases when the fixed locus $R$ of the involution
contains a curve of genus $g\ge2$. Another interpretation is that
these are the $2$-elementary Picard lattices $S_X$ for which a K3
surface has finite automorphism group, excluding the lattice
$(19,1,1)$ with $g=1$ for which the automorphism group is also finite.

\subsection{Coxeter diagrams for lattices with $g=1$ and for $(10,8,0)$}
\label{sec:vinberg-g=1}

These are the $2$-elementary lattices corresponding to K3 surfaces
with an elliptic pencil that is preserved by an infinite automorphism
group $\Aut X$, see Section~\ref{sec:k3-elliptic-pencil}.  For several
of them the Coxeter diagrams are known, e.g.  $(10,10,0)$ in
\cite{vinberg1973some-arithmetic}, $(18,2,0)$ in
\cite{vinberg1978the-groups}, $(19,1,1)$ in
\cite{kondo1989algebraic-k3}. We complete the computation in the
remaining cases.

\begin{theorem}\label{thm:finite-covolume}
  $\wref$ has finite covolume for all the lattices on the $g=1$ line
  except for $(18,2,1)$ and $(17,3,1)$. The finite Coxeter diagrams are as given in
  Figures~\ref{fig:coxeter1}, \ref{fig:coxeter2}.
  For $(18,2,1)$ and $(17,3,1)$, the Coxeter diagrams are infinite
  and are described in Section~\ref{sec:vinberg-exceptional}.
\end{theorem}

Note that in the lattice $(10,8,0)$ the roots generate the index~$2$
sublattice equal $(10,10,1)$, so the two lattices have the same
Coxeter diagrams. In all other cases, the roots generate the full lattice.

\begin{proof}
  The proof is a direct computation using Vinberg's algorithm
  \cite{vinberg1972-groups-of-units, vinberg1973some-arithmetic} which
  is computationally involved but straightforward. In all the cases
  except for $(18,2,1)$ and $(17,3,1)$ the algorithm satisfies
  Vinberg's stopping condition after finitely many steps.

  In the $(18,2,1)$
  case, there are no $(-4)$-vectors of divisibility $2$, so
  $\wref=W_2$.
  By \cite{nikulin1979quotient-groups,
    nikulin1981quotient-groups}   
  a K3 surface
  with this Picard lattice $S_X$ has an infinite automorphism
  group. By the Torelli theorem, this is equivalent to
  $W_2(S_X) \subset O(S_X)$ being of infinite index. Another proof, which
  also works for $(17,3,1)$ is that in both cases there exists a
  negative definite lattice $K$ such that $S\simeq U\oplus K$ and the
  root sublattice $R\subset K$ is of infinite index. These are the
  lattices $A_{15}\!*\!*$ and $A_{13}A_1(2)\!*\!*$ respectively
  of Theorem~\ref{thm:2-elem-negdef}.
  As explained in Section~\ref{sec:coxeter-semifan},
  if $\wref\subset O(S)$ is of finite index then the rays
  on the boundary of $\cH$, giving the $1$-cusps, correspond to
  maximal parabolic subdiagrams and their root sublattices have finite
  index.  
\end{proof}

\begin{figure}[hbp]
  \centering
  \includegraphics[width=5 in]{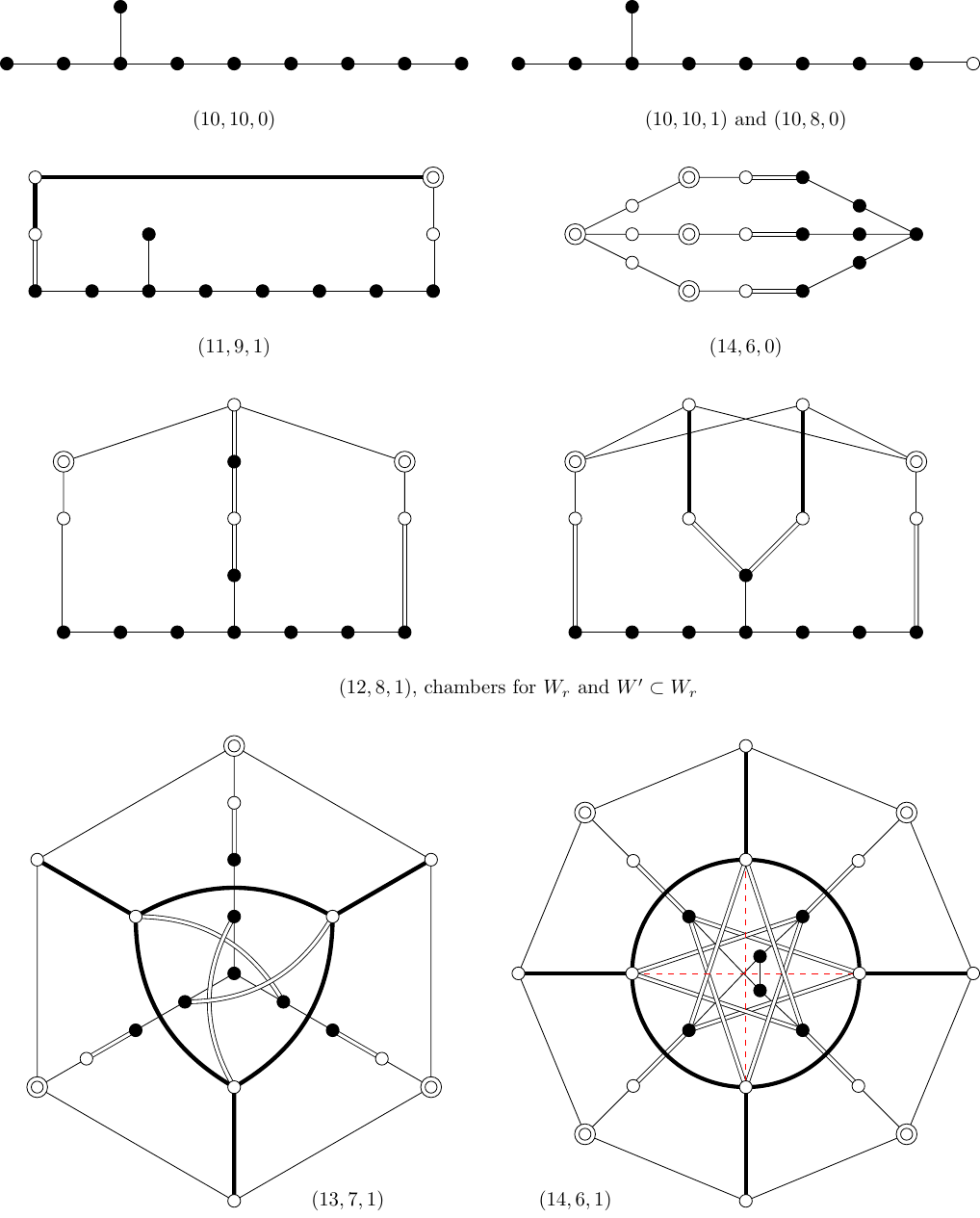}
   \caption{Coxeter diagrams for lattices on the $g=1$ line, part 1}
   \label{fig:coxeter1}
\end{figure}

\begin{figure}[htp]
  \centering
  \includegraphics[width=5 in]{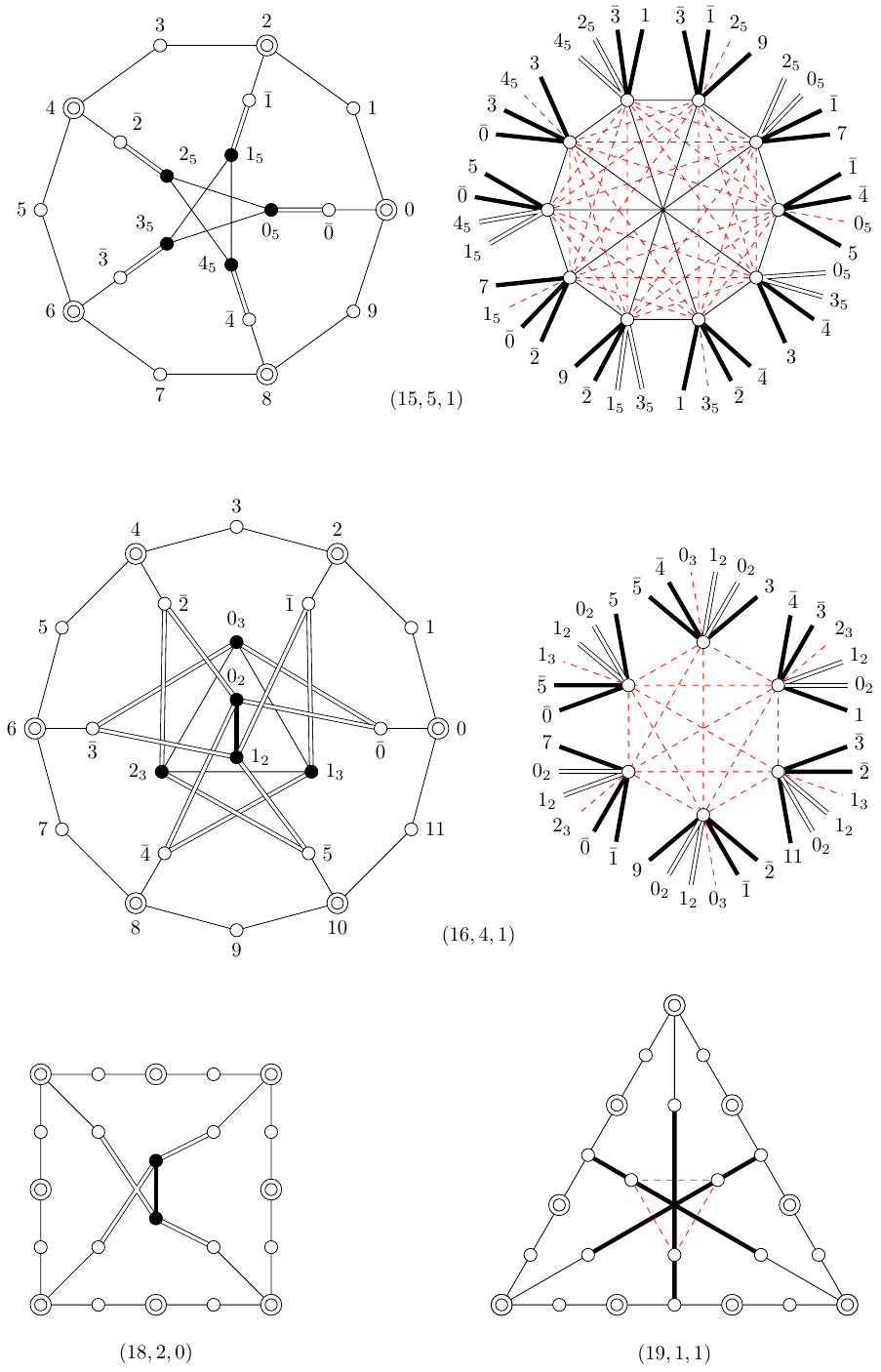}
   \caption{Coxeter diagrams for lattices on the $g=1$ line, part 2}
   \label{fig:coxeter2}
\end{figure}

\begin{remark}
  The $(12,8,1)$ diagram contains the subdiagram $\wB_3$ which
  generates the same lattice as $\wA_3$.  As in
  Lemma~\ref{lem:refl-chambers}, let $\bB$ consist of the isolated
  $(-4)$-vector, so that $W(\bB)=S_2$.  Reflecting the attached
  $(-2)$-vectors gives two other $(-2)$-vectors. One of them forms the
  $\wA_3$ diagram with the others. This gives the Coxeter diagram for
  an index~2 reflection subgroup $W'\subset \wref$, also shown in
  Fig.~\ref{fig:coxeter1}. The righthand diagram is of greater
  relevance for later constructions.
\end{remark}

The diagrams for $(15,5,1)$ and $(16,4,1)$ are quite large so drawn in
two parts. 

\begin{remark}
  The overarching reason for the two exceptional cases is that the
  lattices on the $g=1$ line are mirrors
  of the double covers
  of del Pezzo surfaces. The del Pezzo surfaces of Picard rank $n+1$
  correspond to the $E_n$ lattices. For $n\ge3$ they are root
  lattices: $A_2A_1$, $A_4$, $D_5$, $E_6$, $E_7$, $E_8$. But the $E_2$
  lattice of degree $7$ is not a root lattice, its root sublattice has
  corank~$1$. And for $n=1$ there are two types $E_1$, $E_1'$
  corresponding to $\bF_1$ and $\bF_0$. The lattice $E_1'$
  corresponding to $\bF_0$ has a finite index root sublattice
  $A_1$. But the lattice $E_1$ corresponding to $\bF_1$ has an empty
  root lattice. See more about these in
  \cite{alexeev2020compactifications-moduli}.
\end{remark}

\subsection{Coxeter diagrams for $(18,2,1)$ and $(17,3,1)$}
\label{sec:vinberg-exceptional}

Here we treat the two exceptional cases which do not have finite
covolume. The following is the result of applying Vinberg's algorithm.

\smallskip\underline{$(18,2,1)$}.
There are no 
$(-4)$-vectors of divisibility~$2$ in this lattice $S$, and since the automorphism
group of a K3 surface with this Picard lattice is infinite, there are
infinitely many $(-2)$-vectors, so the Coxeter diagram is
infinite. Rather than attempting to draw it, we describe it in words.

There is a wheel $\wA_{15}$ of $16$ $(-2)$-curves forming a singular
fiber of an elliptic fibration. In addition to them there are $8$
sections $a_0, a_2, \dotsc, a_{14}$ attached to $e_0,e_2, \dotsc a_{14}$
singly, i.e. with $(a_i,e_j) =\delta_{ij}$; and bisections $b_1, b_3,
\dotsc, b_{15}$ with $(b_i,e_j) = 2\delta_{ij}$. The bisections
have divisibility~$2$. 

Let $\wA_{15}^\perp = \la c, v \ra$ with $c$ the class of the
fiber. The Eichler transformation $E = E_{c,v}$
\cite[3.7]{scattone1987on-the-compactification-of-moduli}
is an isometry
$S$ which fixes the vertices of $\wA_{15}$. The orbits
$a_i^{n} := E^n(a_i)$ and $b_i^{n} := E^n(b_i)$ are the remaining
$(-2)$-vectors. Each orbit is isomorphic to $\bZ$.

For the typical representatives $a_0=a_0^{0}$ and $b_1=b_1^{0}$ we list
the vectors with intersections $0,2$. All the other $a_i^{n}$ and
$b_i^{n}$ have intersections $>2$ and give dashed edges in the
Coxeter diagram. 

For $a_0$ the $(-2)$-vectors outside of the $\wA_{15}$ cycle with
intersections $0,2$ are: 

\begin{enumerate}
  \setcounter{enumi}{-1}
\item $a_2^{0}$, $a_4^{1}$, $a_6^{1}$, $a_{10}^{2}$,  $a_{12}^{2}$, $a_{14}^{3}$
  and $b_1^{0}$, $b_3^{0}$, $b_{13}^{2}$, $b_{15}^{2}$. 
  \setcounter{enumi}{1}
\item $a_2^{1}$, $a_8^{1}$, $a_8^{2}$, $a_{14}^{2}$ and 
  $b_7^{1}$, $b_9^{1}$.  
\end{enumerate}

For $b_1$ the $(-2)$-vectors outside of the $\wA_{15}$ cycle with
intersections $0,2$ are: 

\begin{enumerate}
  \setcounter{enumi}{-1}
\item $a_0^0$, $a_2^1$, $a_4^1$, $a_{14}^3$.
  \setcounter{enumi}{1}
\item $a_8^2$, $a_{10}^2$.
\end{enumerate}

The other intersection numbers are recovered by the fact that $E$ is
an isometry, the cyclic symmetry $(u_i, v_j) = (u_{i-k}, v_{j-k})$
($u,v\in \{a,b\}$) for as long as all the indices stay in
$\{0,\dotsc 15\}$. Going around the circle the shift is
$u_{i+16}^n = u_i^{n-3}$. For example,
\begin{eqnarray*}
  && (b_{15}, a_{6}^{-1}) = (b_9, a_0^{-1}) = (b_9^1, a_0^0) = 2,\\
  && (a_{14}^3, a_2^1) = (a_{12}^3, a_0^1) = (a_{12}^2, a_0^0)= 0, \\
  && (b_{15}, a_{16}^1) = (b_{15}, a_0^{-2}) = (b_{15}^2, a_0^0) = 0.     
\end{eqnarray*}

\begin{example}
  $b_1$ is a $(-2)$-vector of divisibility $2$ which does not lie on
  the $\wA_{15}$-cycle. 
  The vectors which have intersection $0$ with $b_1$ are $e_i$ for
  $i\ne1$ and $a_{14}^3$, $a_0^0$, $a_2^1$, $a_4^1$. The last four
  vectors are mutually orthogonal with the exception of
  $(a_0^0, a_2^1)=2$. We conclude that these $15+4$ vectors form the
  Coxeter diagram for the lattice $(17,1,1)$ that was given in
  \cite[Table~1]{alexeev2006del-pezzo}.
  This agrees with Lemma~\ref{lem:heegner-diagrams}.
\end{example}

\smallskip\underline{$(17,3,1)$}. 
Both the diagrams for the full reflection
group $\wref$ and the reflection group $W_2$ in $(-2)$-vectors are
infinite. We compute the latter one.

There is a wheel $\wA_{13}$ of $14$ $(-2)$-curves forming a singular
fiber of an elliptic fibration. In addition to them there are $7$
sections $a_0, a_2, \dotsc, a_{12}$ attached to $e_0,e_2, \dotsc a_{12}$
singly, i.e. with $(a_i,e_j) =\delta_{ij}$; and bisections $b_1, b_3,
\dotsc, b_{13}$ with $(b_i,e_j) = 2\delta_{ij}$. The bisections
have divisibility~$2$. 

Let $\wA_{13}^\perp = \la c, v_1, v_2 \ra$ with $c$ the class of the
fiber. For $v\in\wA_{13}^\perp$ the Eichler transformation
$E = E_{c,v}$ is an isometry $S$ which fixes the vertices of
$\wA_{15}$. We pick $v_1,v_2$ so that their intersection form is
$\begin{pmatrix}
  -4 & 2\\ 2 & -8
\end{pmatrix}.$  The orbits
\begin{displaymath}
a_i^{n_1,n_2} := E_{c,n_1v_1+n_2v_2}(a_i), \quad  
b_i^{n_1,n_2} := E_{c,n_1v_1+n_2v_2}(b_i)
\end{displaymath}
are the remaining $(-2)$-vectors. Each orbit is isomorphic
to~$\bZ^2$. Note that $E_{c,n_1v_1}$ is contained in the Weyl group
$W(\wA_1)$ generated by the reflections in the $(-4)$-vectors $v_1$
and $c-v_1$.

For the typical representatives $a_0=a_0^{0,0}$ and $b_1=b_1^{0,0}$ we list
the vectors with intersections $0,2$. All the other $a_i^{n_1,n_2}$ and
$b_i^{n_1,n_2}$ have intersections $>2$ and give dashed edges in the
Coxeter diagram. 

For $a_0$ the $(-2)$-vectors outside of the $\wA_{13}$ cycle with
intersections $0,2$ are: 

\begin{enumerate}
  \setcounter{enumi}{-1}
\item   $a_0^{-1,0}$, $a_0^{1,0}$, \ 
  $a_2^{0,0}$, $a_2^{0,1}$, $a_2^{1,1}$, \ 
  $a_4^{0,1}$, $a_4^{1,1}$, \ 
  $a_6^{1,2}$, \ 
  $a_8^{1,2}$, \ 
  $a_{10}^{1,3}$, $a_{10}^{2,3}$, \newline
  $a_{12}^{1,3}$, $a_{12}^{2,3}$, $a_{12}^{2,4}$;  \qquad
  $b_1^{0,0}$, $b_1^{1,0}$, \ 
  $b_3^{1,1}$, \ 
  $b_{11}^{2,3}$, \ 
  $b_{13}^{2,4}$, $b_{13}^{3,4}$.
  \setcounter{enumi}{1}
\item $a_0^{-1,-1}$, $a_0^{0,-1}$, $a_0^{0,1}$, $a_0^{1,1}$, \ 
  $a_2^{-1,0}$, $a_2^{1,0}$, \ 
  $a_4^{1,2}$, \ 
  $a_6^{0,1}$, $a_6^{0,2}$, $a_6^{1,1}$, $a_6^{2,2}$, \newline
  $a_8^{0,2}$, $a_8^{1,3}$, $a_8^{2,2}$, $a_8^{2,3}$, \ 
  $a_{10}^{1,2}$, \ 
  $a_{12}^{1,4}$, $a_{12}^{3,4}$; \qquad
  $b_1^{1,1}$,  \ 
  $b_5^{1,1}$, \ 
  $b_7^{1,2}$, $b_7^{2,2}$, \ 
  $b_9^{2,3}$, \ 
  $b_{13}^{2,3}$.
\end{enumerate}

For $b_1$ the $(-2)$-vectors outside of the $\wA_{13}$ cycle with
intersections $0,2$ are: 

\begin{enumerate}
  \setcounter{enumi}{-1}
\item $a_0^{-1,0}$, $a_0^{0,0}$, \ 
  $a_2^{-1,0}$, $a_2^{0,0}$, \ 
  $a_4^{0,1}$, \ 
  $a_{12}^{1,3}$.  
  \setcounter{enumi}{1}
\item $a_0^{-1,-1}$, \ 
  $a_2^{0,1}$, \ 
  $a_6^{0,1}$, \ 
  $a_8^{0,2}$, $a_8^{1,2}$, \ 
  $a_{10}^{1,3}$.  
\end{enumerate}

The other intersection numbers are recovered by the fact that Eichler
transformations are isometries, the symmetries
$(u_i, v_j) = (u_{i-k}, v_{j-k})$ ($u,v\in \{a,b\}$) for as long as
all the indices stay in $\{0,\dotsc 13\}$. Going around the circle the
shift is $u_{i+14}^{n_1,n_2} = u_i^{n_1-2,n_2-4}$.

\begin{example}
  $b_1$ is a $(-2)$-vector of divisibility~$2$ which does not lie on
  the $\wA_{13}$-cycle.
  The vectors which have intersection $0$ with $b_1$ are $e_i$ for
  $i\ne1$ and $a_0^{-1,0}$, $a_0^{0,0}$, $a_2^{-1,0}$, $a_2^{0,0}$,
  $a_4^{0,1}$, $a_{12}^{1,3}$. Using the rules above, we obtain that
  $(a_0^{0,0}, a_2^{-1,0}) = (a_0^{-1,0}, a_2^{0,0}) =2$ and all the
  other intersection numbers among these six vectors are zero. This is
  the same diagram of $(-2)$-curves as the one for the $(16,2,1)$
  lattice obtained from the Coxeter diagram given in
  \cite[Table~1]{alexeev2006del-pezzo} by applying the Weyl group
  $W(A_1)=S_2$ for the unique black vertex.
  This agrees with Lemma~\ref{lem:heegner-diagrams}.
\end{example}

\subsection{Lattices on the $g=0$ line}
\label{sec:vinberg-g=0}

Several of the lattices on the $g=0$ line have finite Coxeter
diagrams. They are quite large and complicated. We don't need them for
the present paper since they don't appear as targets of the mirror
moves~\ref{def:mirror-moves}, see Remark~\ref{rem:mirror-targets}. So
we don't include them here; they will appear elsewhere.

\section{K3 surfaces, their quotients,  and the nef cones}
\label{sec:Y-nef-cone}

We give a brief description for the K3 surfaces $X$ and their
quotients $Y$ that appear in the $75$ families of \fign.

\subsection{Surfaces for $S$ with $g\ge2$, excluding
  $(10,8,0)$}
\label{sec:surfaces-g>1}

For the $50$ lattices with $g\ge2$, excluding $(10,8,0)$, a
satisfactory description for the quotients $Y$ is given in
\cite{alexeev2006del-pezzo}. Indeed, all the possibilities for the
exceptional curves on the surfaces $Y$ appearing in these
families were found. From the graph of exceptional curves one can
realize $Y$ as an explicit blowup of $\bP^2$ or $\bF_n$. 
The K3 surfaces $X$ are double covers of $Y$ branched in some divisor
$B$ lying in the linear system $|-2K_Y|$.

For the most basic lattices $(r,r,\delta)$ with $r\le 9$, the surfaces
$Y$ are weak del Pezzo surfaces with big and nef $-K_Y$ and with $ADE$
singularities. Thus, $Y=\Bl_{r-1}\bP^2$ for $(r,r,1)$ and
$Y=\bP^1\times\bP^1$ or $\bF_2$ for $(2,2,0)$.

\subsection{Surfaces for $S$ on the $g=1$ line and $(10,8,0)$}
\label{sec:k3-elliptic-pencil}

This case is especially important for us since it serves as the base
case for the mirror symmetry constructions and all the other cases are
derived from it.

\begin{theorem}[Nikulin]\label{thm:nikulin-elliptic}
  Let $X$ be a K3 surface with a nonsymplectic involution~$\iota$ and
  $2$-elementary Picard lattice $S=\Pic X$. Denote
  $\pi\colon X\to Y= X/\iota$.  Then the automorphism group $\Aut X$
  is infinite and preserves a (necessarily unique) elliptic pencil
  $f\colon X\to\bP^1$ if and only if $S$ is one of the following lattices:
  \begin{enumerate}
  \item $\I_{2k}\I_0$:
    $(10+k, 10-k, \delta)\ne (14,6,0)$ and
    $k>0$. $Y$ is a rational elliptic surface with a section and 
    the ramification divisor $$R=X^\iota = F_0\cup\cup_{i=0}^{k-1} E_{2i},$$ where $F_0$ is a 
    smooth elliptic fiber, and $E_i$ are disjoint $\bP^1$s which are the
    alternating curves in the $I_{2k}$ fiber $F$ of Kodaira
    type $I_{2k}$ (a wheel of $2k$ rational curves).
  \item $\I_0^2$: $(10,8,0)$. $R=F_0\cup F$ is a union of two smooth
    fibers of $f$.
  \item Halphen: $(10,10,1)$. $Y$ is an index~$2$ Halphen pencil and
    $R = F_0$ is a smooth elliptic fiber which does not ramify over
    the unique multiple fiber $G=2D$ on $Y$. 
      \item Enriques:
    $(10,10,0)$. $Y$ is an Enriques surface and $R=\emptyset$.

  \item $\wE_6$: $(14,6,0)$: $R$ is the sum of a smooth elliptic fiber $F_0$ 
  and $4$ disjoint $\bP^1$s, alternating curves in a $\IV^*$ fiber.
  \end{enumerate}
  In the $\I_0^2$ case one has $g=2$, in the other cases one has $g=1$.
\end{theorem}
\begin{proof}
  \cite[Sec.~4]{nikulin1981quotient-groups} and
  \cite{nikulin2020some-examples}. 
\end{proof}

We will call (1,2,3) \emph{the ordinary cases}. The case (5) does not
appear on the mirror side in our constructions.

\subsection{Surfaces for $S$ on the $g=0$ line}
\label{sec:k3-g=0}

The lattices 
$(10+k,12-k,\delta)$ for $1\le k\le 9$
are in a
bijection with the lattices 
$(10+k,10-k,\delta)$
on the line below in \fign,
excluding $(14,6,0)$, which is the ($\wE_6$) case of
Theorem~\ref{thm:nikulin-elliptic}.

Consider one of the lattices 
$S'=(10+k,10-k,\delta)$, $1\le k\le 9$ 
on the $g=1$ line. This is the case (1) of
Theorem~\ref{thm:nikulin-elliptic}. The quotient $Y'$ is a rational
elliptic surface with a section and $\pi\colon X'\to Y'$ is ramified in
a smooth elliptic fiber $F_0$ and ramified in the $n$ alternating
curves $(-4)$-curves $E_{2i}$. It follows that $S'=U\oplus K$ for some
negative definite lattice $K$. (Indeed, these are the lattices
$A_{2k-1}E_{9-k}(2)$ of Table~\ref{tab:neg-def}.)

Now let $S=U(2)\oplus K$. Then 
$S=(10+k,12-k,\delta)$
is the corresponding
lattice on the $g=0$ line. The line bundle corresponding to an
isotropic vector in $U(2)$ defines an elliptic fibration
$f\colon X\to\bP^1$ without a section whose jacobian fibration is
$f'\colon X'\to\bP^1$. The fiber corresponding to $F$ is a double
fiber of $f$ and the ramification divisor of $\pi\colon X\to Y$ is
the union of the $n$ curves $E_{2i}$. The surface $X$ can be obtained
from $X'$ directly by a logarithmic transformation at $F$
using \cite[Cor.~5.4.7]{cossec1989enriques-surfaces}.

So these K3 surfaces $X$ are the index~2 Halphen K3 surfaces with an
$I_{2k}$ fiber. The surfaces $\oY$ obtained by contracting the
$(-1)$-curves in the special fiber
are the rational index~$2$ Halphen pencils with
an $I_k$ fiber. Some explicit constructions for such pencils were
given in \cite{kimura2018k3-surfaces, kimura2020sun-times,
  zanardini2020explicit-constructions}.

\smallskip

Finally, the moduli space for the lattice $(20,2,1)$ is a point. It is the
unique ``most algebraic'' $2$-elementary K3 surface of
\cite{vinberg1983two-most}. It admits no Halphen pencils but we
computed that up to automorphisms, it has $13$ elliptic pencils with a
section.

\subsection{Nef cones and exceptional curves on  $X$ and $Y$}
\label{sec:nef-XY}

It is well known that the nef cone of a K3 surface $X$ is
defined by linear inequalities in the positive cone
$\ocC=\{v\in S_X\otimes\bR \mid x^2\ge0, \ x\cdot h\ge0\}$, where $h$ is a
K\"ahler class, with facets $x\cdot E_i\ge0$ for the smooth
rational curves $E_i$ with $E_i^2=-2$, which we call the
$(-2)$-curves. Their classes in $\NS(X)$ are the positive roots for the
Weyl group $W_2(S_X)$.  For any $(-2)$-vector $v\in\NS(X)$, exactly one
of $\pm v$ is effective and is a sum of $(-2)$-curves.

Since a $(-2)$-curve is uniquely determined by its class in $\NS(X)$,
any involution preserving its class preserves the $(-2)$-curve, but may
not fix it pointwise.

We call a curve \emph{exceptional} if it is irreducible and has
negative self-intersection. $R=X^\iota$ will denote the ramification
divisor of $\pi\colon X\to Y$ and $B\subset Y$ the branch divisor. 

\begin{lemma}\label{lem:exc-curves-Y}
  Exceptional curves $F$ on $Y$ are of three types:
  \begin{enumerate}
  \item[(1)] $F^2=-4$, $\pi^*(F) = 2E$, and the $(-2)$-curve
    $E\subset X^\iota$ is fixed pointwise by the involution.
  \item[(2a)] $F^2=-1$, $\pi^*(F)=E$, $|F\cap B| = |E\cap R|=2$.
  \item[(2b)] $F^2=-1$, $\pi^*(F)=E_1+E_2$, $F$ is tangent to $B$, the
    curves $E_1$, $E_2$ are exchanged by the involution and
    $E_1\cdot E_2=1$.
  \item[(3)] $F^2=-2$, $\pi^*(F)=E_1+E_2$,
    $F\cap B = E_i\cap R=\emptyset$, the curves $E_1$, $E_2$ are
    disjoint and are exchanged by the involution.
  \end{enumerate}
\end{lemma}
\begin{proof}
  \cite[Sec.~2.4]{alexeev2006del-pezzo}.
\end{proof}

\begin{lemma}\label{lem:nef-Y-trivial}
  $\Nef Y = (\Nef X)\cap S_\bR$, where $S= S_X^+$
\end{lemma}
\begin{proof}
  This follows from the identities $\iota^*\circ\iota_* = 1+\iota$
  and $\iota_*\circ\iota^* = 2$. 
\end{proof}

\begin{definition}
  Let $\Cur_2=\{F_i\}$ be the set of the $(-2)$-curves on $Y$, i.e. of
  type (3) in Lemma~\ref{lem:exc-curves-Y}. Let $\pi^*(F)=E_1+E_2$ be
  the corresponding $(-4)$-vectors in~$S$. It is obvious that for any
  $v\in S$ one has $\pi^*(F)\cdot v = 2E_1\cdot v\in 2\bZ$. They are
  also primitive since $(\pi^*(F)/2)^2=-1$ and $S$ is an even
  lattice. So they are simple roots for the full reflection group
  $\wref(S)$ and they correspond to a subset $\bB$ of black vertices in
  the Coxeter diagram of $S$ as in Definition~\ref{def:refl-chambers}.
\end{definition}

\begin{lemma}\label{lem:nef-Y}
  $\Nef Y=W(\bB).\chref$, where $\chref$ is a fundamental chamber
  for $\wref(S)$.
  The set of exceptional curves on $Y$ is identified with
  the union of $W(\bB)$-orbits
  of white vertices in the Coxeter diagram of $S$.
\end{lemma}
\begin{proof}
  This follows by Corollary~\ref{cor:refl-chambers}
  and the description of $\Nef X$ above.   
\end{proof}

\subsection{Surfaces $Y$ with the smallest nef cone}
\label{sec:smallest-nef}

\begin{proposition}\label{prop:small-nef-cone}
  For each lattice $S\ne(10,10,0)$,  with $g\ge1$ of \fign\ there exists
  a K3 surface with $S_X^+=S$ such that $\Nef Y$ can be identified
  with the Coxeter chamber for the full reflection group $\wref$ if
  $S\ne(12,8,1)$, and the Coxeter chamber for an index~$2$ subgroup
  $W'\subset \wref$ if $S=(12,8,1)$.
\end{proposition}
\begin{proof}
  In view of Lemma~\ref{lem:nef-Y} we need to find a quotient surface
  $Y$ on which the $(-2)$-curves form the entire black subdiagram of
  the Coxeter diagram, for $S\ne (12,8,1)$. 

  In the $(12,8,1)$ case we consider the fundamental chamber
  $\ch'\dref$ for the index~$2$ subgroup of $W'\subset \wref$, a
  union of two fundamental chambers $\chref$ and $w\chref$ where $w$
  is the reflection in the isolated $(-4)$-root. This modified chamber
  is also pictured in Fig.~\ref{fig:coxeter1}. The corresponding
  surface has an $I_2$ fiber. 

  For the $50$ lattices with $g\ge2$ and
  different from $(10,8,0)$ existence of such $Y$ is a small part of
  \cite{alexeev2006del-pezzo} where \emph{all} possibilities for the
  sets of exceptional curves were classified. The surfaces we need
  here are ``the most degenerate'', they all appear in
  \cite[Table~3]{alexeev2006del-pezzo}.

  For the surfaces  with an elliptic pencil with a
  section we take for $Y$ one of the surfaces of
  Table~\ref{tab:extremal-surfaces}.  They exist by
  \cite{persson1990configurations-of-kodaira}; see also
  \cite{miranda1986on-extremal-rational, oguiso1991mordell-weil}.

  \begin{table}[htpb]
    \centering
    \begin{tabular}{l@{\hspace{2em}}lll}
      Case & Singular fibers &Fiber root lattice &$MW(Y)$\\
      \hline
      $(10,10,1)$ &$_2\I_0\II^* + \I_1^2$ &$E_8$&$\emptyset$\\
      $(10,8,0)$  &$\II^* + \I_1^2$ &$E_8$&$0$\\
      $(11,9,1)\ A_0E_8$ &$\I_1\II^*+ \I_1$ &$E_8$&$\bZ_1$\\
      $(12,8,1)\ A_1E_7$ &$\I_2\,\III^*+ \I_1$ &$A_1E_7$ &$\bZ_2$\\
      $(13,7,1)\ A_2E_6$ &$\I_3\IV^*+ \I_1$ &$A_2E_6$ &$\bZ_3$\\
      $(14,6,1)\ A_3E_5$ &$\I_4\I_1^*+ \I_1$ &$A_3D_5$ &$\bZ_4$\\
      $(15,5,1)\ A_4E_4$ &$\I_5\I_5+I_1^2$ &$A_4A_4$ &$\bZ_5$\\
      $(16,4,1)\ A_5E_3$ &$\I_6\I_3\I_2+ \I_1$ &$A_5A_2A_1$ &$\bZ_6$\\
      $(17,3,1)\ A_6E_2$ &$\I_7\I_2+\I_1\II$ &$A_6A_1$ &$\bZ$ \\
      $(18,2,1)\ A_7E_1$ &$\I_8+\I_1^2\II$ &$A_7$ &$\bZ$ \\
      $(18,2,0)\ A_7E'_1$ &$\I_8\I_2+ \I_1^2$ &$A_7A_1$ &$\bZ_2^2$\\
      $(19,1,1)\ A_8E_0$ &$\I_9+\I_1^3$ &$A_8$ &$\bZ_3$ \\
    \end{tabular}
    \medskip
    \caption{Special rational elliptic surfaces $Y$}
    \label{tab:extremal-surfaces}
  \end{table}

  In the $(10,10,1)$ Halphen case, the surface with a double
  $_2I_0$ fiber can be obtained from a $(10,8,0)$ surface by
  a logarithmic transformation along a smooth $I_0$ fiber using
  \cite[Cor.~5.4.7]{cossec1989enriques-surfaces}.
\end{proof}

\begin{remark}
  Most of the surfaces of Table~\ref{tab:extremal-surfaces} are the
  maximally degenerate ones in their families but some are not: For
  $(12,8,1)$ the most degenerate surface has $III^*III$ fibers, and
  for $(10,10,1)$ the $_2I_1 II^* I_1$ fibers.

  We use a logical notation $A_{k-1}E_{9-k}$ to denote the cases
  $(10+k,10-k,\delta)$ for $1\le k\le 9$. Here, $E_k$ is the lattice
  $K^\perp$ in the Picard lattice for a a del Pezzo surface of degree
  $K^2=9-k$. For $k=1$ there are two cases, $E_1$ for $\bF_1$ and
  $E'_1$ for $\bF_0$.
\end{remark}

\begin{corollary}\label{cor:exc-curves-degenerate-Y}
  For the surfaces $Y$ of Proposition~\ref{prop:small-nef-cone}, the
  Coxeter diagram of $S$ also serves as the dual graph of exceptional
  curves, with the following modifications:
  \begin{enumerate}
  \item Vertices: the double circled white, single white circles, and
    black vertices respectively correspond to $\bP^1$ with $F^2=-4$,
    $-1$, $-2$ respectively.
  \item Edges: $F_i\cdot F_j=1$ for a white vertex $F_i$ and a black vertex
    $F_j$ or for two single circled white vertices. The other
    intersection numbers are: $F_i\cdot F_j=1$ for a single edge, and $2$ for
    a bold edge in the diagram.
  \end{enumerate}
\end{corollary}

\subsection{The Heegner divisor hierarchy}
\label{sec:heegner}

\begin{definition}[Heegner moves]\label{def:heegner-moves}
  The $(-1,-1)$-move goes from a node $(r,a,1)$ of \fign\ to the node
  $(r-1,a-1,\delta)$. We call it \emph{ordinary} if $\delta=1$ and
  \emph{characteristic} if $\delta=0$. 
  The opposite $(+1,+1)$-move goes from $(r,a,\delta)$ to
  $(r+1,a+1,1)$.
  The $(+1,-1)$-move is from $(r,a)$ to $(r+1,a-1)$ and the opposite
  $(-1,+1)$-move is from $(r,a)$ to $(r-1,a+1)$. 
\end{definition}

For the rest of this section, we exclude the lattice $S=(10,8,0)$
which is in many ways exceptional, cf. Remark~\ref{rem:10-8-0}.

\begin{lemma}\label{heegner-move-interp}
  Let $S=(r,a,\delta)$ and $\bD_S$ be its period domain,
  $S\rightarrowtail S'=(r',a',\delta')$ be a $(+1,+1)$ or
  $(+1,-1)$-move.  Then $S'$ defines a Heegner divisor
  $\bD_{S'}\subset \bD_S$ on which the K3 surfaces acquire an
  additional $(-2)$-curve $E$, preserved by the involution on K3 surfaces over
  $S'$. For the $(+1,+1)$-move, this involution preserves
  but does not fix $E$ and for the $(+1,-1)$-move, it fixes $E$ pointwise.
\end{lemma}
\begin{proof}
  Let $X\to (C,0)$ be a smooth family of K3 surfaces with $[X_0]$
  generic in $\bD_{S'}$, i.e. $H^2(X_0,\bZ)^+ = S'$ and with
  $[X_t]\in \bD_S\setminus\bD_{S'}$ for $t\ne0$.
  Let $R_t$ be the ramification divisor on $X_t$, $t\ne0$, $\oR_0$ be
  its flat limit, and $R_0$ be the ramification divisor of the
  involution of $X_0$ determined by $S'$. Then for a 
    \begin{enumerate}
  \item[] $(+1,+1)$-move: $k_0=k$, $g_0=g-1$, and $\oR_0 = R_0 \cup E$. 
  \item[] $(+1,-1)$-move: $g_0=g$, $k_0=k+1$ and $R_0 = \oR_0 \sqcup
    E$. 
  \end{enumerate}
  
  This is proven by considering the small contraction $X\to \oX$ which contracts
   $E$ to a point. The birational involution on $X$ equaling $\iota_t$
   on the general fiber extends to a regular involution of $\oX$ and the
   two cases are distinguished simply by whether the contraction of $E$
   lies on the limit of $R_t$ (necessarily a node of the limiting curve)
   or is disjoint from the limit. In the former $(+1,+1)$ case, the involution on the minimal resolution
   $X_0\to \oX_0$
   preserves $E$ but only fixes the two branches of the node. In the latter $(+1,-1)$ 
   case, the involution on $X_0$ fixes $E$ pointwise,
   because the contraction of $E$ is isolated in the fixed locus of the involution
   on $\oX_0$.
\end{proof}

\begin{lemma}
 Any $2$-elementary lattice with $g\ge1$ has the form
 $$S = \big(10+k-(g-1), 10-k-(g-1), \delta\big)$$ and it can be reached from
 one of the lattices of Section~\ref{sec:k3-elliptic-pencil} with
 $g=1$ by $g-1$ total $(-1,-1)$-moves.
\end{lemma}
\begin{proof}
 One look at \fign\ confirms this.
\end{proof}

Next, we want to understand how Coxeter diagrams change.

\begin{lemma}\label{lem:heegner-general}
  For the lattices related by a $(-1,-1)$-move one has
  $S=S'\oplus A_1$ and a generator $r$ of $A_1$ is a $(-2)$-root of
  divisibility~$2$.
  For any hyperbolic lattice $S=(r,a,1)$ there exists exactly one or
  two $O(S)$-orbits of $(-2)$-roots of divisibility~$2$, and these
  vectors are in a bijection with the $(-1,-1)$-moves down from $S$.

  In terms of the Coxeter diagram $\Gamma$, for any $S\ne(10,8,0)$
  they are in bijection
  with the $\Aut\Gamma$-orbits of white vertices in $\Gamma$ which are
  not connected to some neighbor by a single (i.e.~weight~$1$) edge.
\end{lemma}
\begin{proof}
  One has
  \begin{displaymath}
  S'\oplus A_1 = (r-1,a-1,\delta)_1 + (1,1,1)_0 = (r,a,1)_1 \simeq S.    
  \end{displaymath}

  For $S=S'\oplus A_1= S''\oplus A_1$ with $S',S''$ 
  of the same type, an isometry
  $S'\to S''$ defines an isometry $S\to S$, so there is exactly one
  $O(S)$-orbit.
  The generator of $A_1$ is a $(-2)$-root and it clearly has
  divisibility~$2$. By
  Prop. on p.2 of \cite{vinberg1983two-most} such roots lie in different
  $\wref$-orbits. One has $O(S)=\Aut\Gamma\ltimes \wref$, so two
  such roots in the same orbit must differ by a diagram
  symmetry. Finally, for all the lattices in \fign\ with the exception
  of $(10,8,0)$ the roots generate the lattice, so $\di(r)=2$ iff
  $r\cdot r_i$ is even for the other roots, which is read off directly
  from the diagram. 
\end{proof}

\begin{corollary}\label{cor:reach-by-moves}
  Any lattice in \fign\ with $g\ge1$ is of the form
  $$S =\big(10+k-(g-1),10-k-(g-1),\delta\big).$$
  It can be reached from the ``base'' lattice
  $S_1 = (10+k,10-k,1) = S\oplus A_1^{g-1}$ by a sequence of
  $(-1,-1)$-moves which corresponding to a chain of vertices
  $\alpha_1,\dotsc,\alpha_{g-1}$ with colors white-black-\dots-black,
  in the Coxeter diagram $\gref(S)$, with $\alpha_1$ an even $(-2)$-root.
  This chain is unique up to $\Aut\Gamma$.
\end{corollary}

\begin{remark} The $(+1,-1)$ moves play a particularly important
role from the perspective of moduli, because of Lemma \ref{heegner-move-interp}. 
Recall that $F_S = (\bD_S\setminus \Delta)/O(T)$. The divisors in $\Delta/O(T)$
correspond to Heegner divisor moves of either $(+1,-1)$ type or $(+1,+1)$-type.
For a Heegner divisor of $(+1,-1)$-type, the involution on the limiting of K3 surface
has an extra $(-2)$-curve, fixed pointwise by the involution, but the
flat limit of a positive genus fixed component $C_g$ equals the 
positive genus fixed component. Thus from the perspective of the KSBA
compactification $F_S\hookrightarrow \oF_S$
for pairs $(\oX,\epsilon \oC_g)$, the $(+1,-1)$-move plays
essentially no role, except to restrict to a Noether-Lefschetz subdomain of $\oF_S$.

Thus, it suffices to compactify the moduli space $F_S$
for lattices $S$ on the $k=0$ line, and the other cases follow by induction. 
For the hyperbolic $2$-elementary lattices $\oT = I^\perp/I$ associated to
 Type III cusps of $F_S$ (cf. Sec.~\ref{sec:cusps}),
 the action of a $(+1,-1)$ Heegner move on $S$
 is a $(-1,-1)$ move on $\oT$. This is why Corollary
  \ref{cor:reach-by-moves} is relevant: We can reach all necessary
  hyperbolic lattices for semitoroidal compactifications from
  those on the $g=1$ line.
 \end{remark}

\begin{lemma}\label{lem:heegner-diagrams}
  The Coxeter diagram $\Gamma'$ of $S'$ is
  obtained from the diagram $\Gamma$ of $S=S'\oplus A_1$ 
  by removing the vertex $r$, removing the adjacent
  white vertices, turning black adjacent vertices to white and, if
  there exist $(-4)$-roots $\alpha_1,\alpha_2$ with $\alpha\cdot \alpha_1=\alpha\cdot \alpha_2=2$, connecting
  their images in $\Gamma_1$ by a double line.
\end{lemma}
\begin{proof}
  The fundamental chamber for $S'$ is the face $\alpha^\perp$ of the
  fundamental chamber for $S'\oplus A_1$. The facets of this chamber
  are of the form $p(\alpha_i)^\perp$ where $p(\alpha_i)$ 
  are the orthogonal projections of the roots defining the facets for
  $S$ for which $p(\alpha_i)^2<0$. The formula
  $p(\alpha_1)\cdot p(\alpha_2) = \alpha_1\cdot \alpha_2  + \frac12 (\alpha\cdot \alpha_1) (\alpha\cdot \alpha_2)$
  implies the rest: For the $(-2)$-neighbors $\alpha_1$ of $\alpha$ with $\alpha\cdot \alpha_1=2$
  one gets $p(\alpha_1)^2=0$, so they are not faces of $\alpha^\perp$, and for
  $(-4)$-neighbors with $\alpha\cdot \alpha_1=2$ one gets $p(\alpha_1)^2=-2$, so they change 
  their color.
\end{proof}

\begin{lemma}\label{lem:subgraphs-G-G1}
  Let $S,S_1$ and $\alpha_1,\dotsc,\alpha_{g-1}$ be as in
  Corollary~\ref{cor:reach-by-moves}. Let $G\subset\Gamma(S)$ and
  $G_1\subset\Gamma(S_1)$ be the subdiagrams such that for the
  vertices one has $V(S_1) = V(S) \cup \{\alpha_1,\dotsc,\alpha_g\}$.
  Then
  \begin{enumerate}
  \item $G_1$ is elliptic iff $G$ is too, and no vertex of $G$ is a
    neighbor of $\alpha_1,\dotsc,\alpha_{g-1}$.
  \item $G_1$ is maximal parabolic if $G$ contains a parabolic
    subdiagram and either no vertex of $G$ is a
    neighbor of $\alpha_1,\dotsc,\alpha_{g-1}$ or $g-1=1$ and
    $\alpha_1$ neighbors one isolated vertex of~$G$. 
  \end{enumerate}
\end{lemma}
\begin{proof}
  If $\beta$ is a vertex of $G$ attached to $\alpha_1$ then
  $\beta^2=-2$ and $\beta\cdot \alpha_1=2$, so
  $\la\beta,\alpha_1\ra = \wA_1$ is already parabolic. (1) immediately
  follows. In the parabolic case, maximal parabolic subgraphs have rank
  $\rk S-1$, so the remaining part is nonempty and contains a
  nonempty maximal parabolic subdiagram. For the exceptional
  cases $S_1=(18,2,1)$ and $(17,3,1)$ the diagrams for the lattice
  $S_2=(17,1,1)$, resp. $(16,2,1)$ contain parabolic subdiagrams
  $\wD_{14}$, resp. $\wD_{12}$ disjoint from~$\wA_1$. 
\end{proof}

\begin{lemma}\label{lem:heegner-surfaces}
  For the surfaces of Corollary~\ref{cor:exc-curves-degenerate-Y} for
  which the Coxeter diagram is the dual graph of the exceptional
  curves on the quotients, the surface $Y_{1}$ is obtained from $Y$ by
  contracting a $(-1)$-curve $F$ not contained in the branch divisor $B$.
  The surface $X_{1}$ is obtained from $X$ by contracting a $(-2)$-curve
  not contained in the ramification divisor $R$, and then smoothing
  the singular K3 surface together with an involution.
\end{lemma}

\begin{remark}\label{rem:the-other-move}
  The $(+1,-1)$ and $(-1,+1)$ moves in \fign\ do not admit such an
  easy description. In those cases $S$ corresponds to an index~$2$
  overlattice of $S'\oplus A_1$.  The
  $(r,a,\delta_1)\rightarrowtail (r-1,a+1,\delta)$ move can be understood as
  contracting a $(-4)$-curve on $Y'$ and then smoothing. For example,
  $(2,0,0)\rightarrowtail (1,1,1)$ is a smoothing of $\bF_4^0=\bP(1,1,2)$ to
  $\bP^2$. But that is a far trickier operation than contracting a
  $(-1)$-curve. 
  There is also no obvious relation between the Coxeter
  diagrams. For example, in the sequence $(19,1,1) \rightarrowtail (18,2,1) \rightarrowtail
  (17,3,1) \rightarrowtail (16,4,1)$ the diagrams go from being finite to infinite to
  finite again, see Sections~\ref{sec:vinberg-g=1} and
  \ref{sec:vinberg-exceptional}. 
\end{remark}

\section{The cusps of $\bD_S/O(T)$}
\label{sec:cusps}

Let $S=(r,a,\delta)$ be one of the $2$-elementary hyperbolic lattices
of \fign, $T=S^\perp=(22-r,a,\delta)_2$, and $\rho$ the corresponding
involution of $\lk$ for which $S$ and $T$ are the
$(\pm 1)$-eigenspaces. By Section~\ref{sec:moduli-with-involution} we
have a moduli space $F_S = (\bD_S\setminus\Delta)/O(T)$.  As in
Section~\ref{sec:bb-etc-compactifications}, the boundary of the
Baily-Borel compactification consists of

\begin{enumerate}
\item points, called $0$-cusps, in bijection with the $O(T)$-orbits of
  isotropic lines $I=\bZ e\subset T$, $e^2=0$, \vspace{1pt}
\item modular curves, called $1$-cusps, in bijection with the $O(T)$-orbits
of isotropic planes $J\subset T$.
\end{enumerate}

A $0$-cusp appearing in the compactification of a modular curve $1$-cusp
corresponds to an inclusion $I\subset J$. In this section we find all cusps,
together with their incidence relations.

\subsection{Isotropic vectors in $2$-elementary discriminant groups}
\label{sec:iso-AT}

For a nonzero vector $T$ its divisibility $\di(v)\in\bN$ is defined by
$v \cdot T = \di(v)\bZ$. Since $T$ is $2$-elementary, $\di(v)=1$ or
$2$. Define $v^* = v/\di(v) \in A_T$; one has $v^*=0$ iff
$\di(v)=1$. If $e$ is a primitive isotropic vector then one certainly
has $q_T(e^*)=0$. Thus, the first step in classifying the $0$-cusps is
to understand the isotropic vectors in the finite discriminant group
$A_T$. For this, one has the following result of Nikulin.

\begin{definition}
  For an $2$-elementary lattice $H$ the form $q \pmod\bZ\colon A_H\to
  \frac12\bZ / \bZ$ is linear. A nonzero vector $v^*\in A_H$ is called
  \emph{characteristic} if $q(x) = (x,v^*) \pmod \bZ$ for all $x\in
  A_T$. It is called \emph{ordinary} otherwise. Note that if the
  lattice $H$ is co-even then $q\pmod\bZ\equiv 0$ and there are no
  characteristic vectors.
\end{definition}

\begin{lemma}[\cite{nikulin1979integer-symmetric}, Lemma 3.9.1]
  \label{lem:disc-0-cusps}
  Let $(A_T,q_T)$ be the discriminant group of an even $2$-elementary
  lattice. Then there are at most two orbits of nonzero isotropic
  vectors in $A_T$: ordinary and characteristic (thus, at most three
  including $e^*=0$).
\end{lemma}

\begin{definition}\label{def:types-e*}
  Let $H$ be a $2$-elementary lattice and $e\in T$ be a primitive
  isotropic vector. We say that $e$ is \emph{odd} or \emph{simple} if
  $\di(e)=1$; $e$ is \emph{even ordinary} if $\di(e)=2$ and $e^*$ is
  ordinary; and $e$ is \emph{even characteristic} if $\di(e)=2$ and
  $e^*$ is characteristic.
\end{definition}

\subsection{Isotropic vectors in $2$-elementary lattices}
\label{sec:iso-T}

\begin{lemma}\label{lem:perp-p-elem}
  Let $K\subset T$ be a primitive sublattice. If $K$ and $T$ are
  $p$-elementary for some prime $p$ then so is $K^\perp$.
\end{lemma}
\begin{proof}
  By \cite[1.15.1]{nikulin1979integer-symmetric} the discriminant
  group of $K^\perp$ is $G^\perp/G$ for a
  certain subgroup $G \subset A_T \oplus A_K$. If $A_T$,
  $A_K$ are vector spaces over $\bF_p$ then so is $A_{K^\perp}$.
\end{proof}

\begin{proposition}
  \label{prop:isotropic-vecs}
  Let $T$ be an even indefinite $2$-elementary lattice
  of signature $(n_+,n_-)$,
  $e\in T$ a nonzero primitive isotropic vector, and let
  $\oT = e^\perp/e$. Then there exist sublattices $H,K\subset T$ such
  that $T=H\oplus K$, $K\simeq\oT$, $e\in H$, and exactly one of the
  following holds.
  \noindent If $\delta_T=\delta_\oT$: 
  \begin{enumerate}
  \item $H=U$,  $a_\oT=a_T$ and $e$ is odd
    (Def.~\ref{def:types-e*}). 
  \item $H=U(2)$,  $a_\oT = a_T-2$ and $e$ is even ordinary.
  \end{enumerate}
  If $\delta_T=1$ and $\delta_\oT=0$: 
  \begin{enumerate}\setcounter{enumi}{2}
  \item $H=I_{1,1}(2)$, $a_\oT = a_T-2$ and $e$ is even characteristic.
  \end{enumerate}

  An isotropic vector of type (1--3) exists iff there exists a
  $2$-elementary lattice $\oT$ of signature $(n_+-1,n_--1)$ with the
  invariants $(r_\oT=r_T-2, a_\oT, \delta_\oT)$, and then it is unique
  up to $O(T)$-action.
\end{proposition}

Here, $\I_{1,1}=\la 1\ra \oplus\la -1\ra$ is the odd hyperbolic
unimodular hyperbolic rank-$2$ lattice; recall that $\II_{1,1}=U$ is
the even one.

\begin{proof}
  Let $T = (r_T,a_T,\delta_T)_{n_+}$ and $e^* = e/\di(e)\in A_T$. The
  lattice $\oT = e^\perp/e$ has signature $(n_+-1,n_--1)$
  and its discriminant group is $A_\oT = (e^*)^\perp/e^*$. If
  $\di(e)=1$ then there exists an isotropic vector $f$ such that $\la
  e, f\ra = U$, $T=U\oplus U^\perp$ and we are done. This is case (1).

  So suppose that $\di(e)=2$. Then $a_\oT=a_T-2$. Pick any lift
  $\oT \to K \subset e^\perp$ of the projection $e^\perp\to \oT$. It
  exists and it is automatically an isometry. So we got a primitive
  sublattice $K\simeq\oT$ of $T$. Let $H=K^\perp\subset T$. One has
  $e\in H$.  By Lemma~\ref{lem:perp-p-elem}, $H$ is $2$-elementary. It
  is also hyperbolic. So $H=U$, $U(2)$ or $I_{1,1}(2)$.
  
  The case $H=U$ is impossible since $\di(e)=2$.  In the other two
  cases $H\oplus K\subset T$ have the same $a_{H\oplus K}=a_T$, so they are equal.  If
  $H=U(2)$ or if $H=I_{1,1}(2)$, $\delta_T=1$ and $\delta_\oT=0$ then we
  are done. So suppose that $H=I_{1,1}(2)$ and
  $\delta_T=\delta_\oT=1$.

  Write $\I_{1,1}(2) = \la e_1, e_2\ra$, $e_1^2=2$, $e_2^2=-2$ and
  $e=e_1+e_2$. One has $e\cdot e_1=2$. Since $K$ is co-odd, there exists
  $x\in K$ of divisibility~$2$ such that $x^2 \equiv 2\pmod 4$. Then
  $f = e_1 + x$ has divisibility~$2$ and satisfies $e\cdot f=2$, $f^2 = 4n$.
  Then $(f-ne)^2=0$ and $\la e,f-ne\ra$ splits off a copy of $U(2)$ so
  that $T = U(2)\oplus K'$ as well, as in case (2).

  Vice versa, if a lattice $\oT$ with the invariants as in cases
  (1--3) exists then $T$ and $H\oplus \oT$ are $2$-elementary, even,
  indefinite and have the same $(r,a,\delta)$. So they are isomorphic.
  The statement about the types of $e^*\in A_T$ is immediate.

  If there are two vectors $e_1,e_2$ of the same type with isomorphic
  $\oT_i$ then the isomorphisms $H_1\to H_2$, $K_1\to K_2$ define an
  isometry $T\to T$. Noting that the primitive isotropic vectors in
  each $H_i$, $i=1,2$, are in the same $O(H_i)$-orbit, this shows that
  $e_1,e_2$ and the splittings $T=H_i\oplus K_i$ are in the same
  $O(T)$-orbit.
\end{proof}

\subsection{The $0$-cusps}
\label{sec:0-cusps}

\begin{definition}[Mirror moves]\label{def:mirror-moves}
  We define three ``mirror moves'' from a node
  $(r,a,\delta)$ to a node $(\mr,\ma, \mdelta)$ of \fign.
  \begin{displaymath}
    S = (r,a,\delta)_1 \to
    S^\perp = T= (22-r,a,\delta)_2 \leadsto
    \oT = (\bar r, \bar a,\bar\delta)_1,
  \end{displaymath}
  where the move $T\leadsto \oT$ is one of the following:
  \begin{enumerate}
  \item odd or simple: $(22-r,a,\delta)_2 \to (20-r,a,\delta)_1$,
  \item even ordinary: $(22-r,a,\delta)_2 \dto (20-r,a-2,\delta)_1$,
  \item even characteristic: $(22-r,a,1)_2 \dto (20-r,a-2,0)_1$.
  \end{enumerate}
  
  \includegraphics{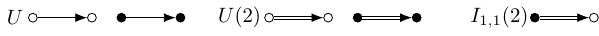}
  
  This makes \fign\ into a directed graph in which every vertex has in-
  and out-degrees equal to $0$, $1$, $2$ or $3$.
\end{definition}

\begin{remark}\label{rem:mirror-targets}
  The only nodes which are \emph{not} targets of any mirror move
  are those on the line $g=0$ and the point $(14,6,0)$. In \fign\ they
  are marked in red.
\end{remark}

\begin{theorem}\label{thm:0-cusps}
  Let $S$ be one of the $2$-elementary lattices of \fign\ and $e\in T$ a
  primitive isotropic vector with $\oT = e^\perp/e$. Then the move
  $S \leadsto \oT$ is one of the mirror moves of
  Def.~\ref{def:mirror-moves}. Moreover, if \fign\ allows for a particular
  move then there exists a unique $0$-cusp of $\bD_S/O(T)$ of that
  type.
\end{theorem}
\begin{proof}
  This follows immediately from Proposition~\ref{prop:isotropic-vecs}
  applied to $T=S^\perp$.  Indeed, in this case $\oT$ is again
  hyperbolic and it admits a primitive embedding into $\lk$. So it
  corresponds to a node in \fign.
\end{proof}

\subsection{The $1$-cusps}
\label{sec:1-cusps}

\begin{lemma}
  \label{lem:1-cusps}
  Let $T$ be an even indefinite $2$-elementary lattice of signature
  $(n_+,n_-)$, $J\subset T$ a primitive isotropic plane, and let
  $\ooT = J^\perp/J$. Then there exist sublattices $P,N\subset T$ such
  that $T=P\oplus N$, $N\simeq\ooT$, $J\subset P$, and exactly one of
  the following holds.  
  \noindent If $\delta_T=\delta_\oT$: 
  \begin{enumerate}
  \item $P=U^2$ and  $a_\ooT=a_T$.
  \item $P=U\oplus U(2)$ and  $a_\ooT=a_T-2$.
  \item $P=U(2)^2$ and  $a_\ooT=a_T-4$.
  \end{enumerate}
  If $\delta_T=1$ and $\delta_\ooT=0$: 
  \begin{enumerate}\setcounter{enumi}{3}
  \item $H=U \oplus I_{1,1}(2)$ and $a_\oT = a_T-2$.
  \item $H=U(2) \oplus I_{1,1}(2)$ and $a_\oT = a_T-4$.
  \end{enumerate}
  
  An isotropic plane of type (1--5) exists iff there exists a
  $2$-elementary lattice $\ooT$ of signature $(n_+-2,n_--2)$ with the
  invariants $(r_\ooT=r_T-4, a_\ooT, \delta_\ooT)$, and then it is unique
  up to $O(T)$-action.
\end{lemma}
\begin{proof}
  We apply Proposition~\ref{prop:isotropic-vecs} twice.
\end{proof}

\label{sec:neg-def}
\begin{table}[hp!]
  \centering
\renewcommand{\arraystretch}{1.1}
  \begin{tabular}[t]{|c>{$}l<{$}|}
  \hline
  \text{Case} &\text{Lattice} \\
  \hline
  \ccy (0,0,0) &0\\
  \hline
   (1,1,1) &A_1\\
  \hline
   (2,2,1) &A_1^2\\
  \hline
   (3,3,1) &A_1^3\\
  \hline
  \ccy (4,2,0) &D_4\\
   (4,4,1) &A_1^4\\
  \hline
   (5,3,1) &D_4 A_1\\
   (5,5,1) &A_1^5\\  
  \hline
   (6,2,1) &D_6\\
   (6,4,1) &D_4 A_1^2\\
   (6,6,1) &A_1^6\\
  \hline
   (7,1,1) &E_7\\
   (7,3,1) &D_6A_1\\
   (7,5,1) &D_4 A_1^3\\
   (7,7,1) &A_1^7\\
  \hline
  \ccy (8,0,0) &E_8\\
  \ccy (8,2,0) &D_8\\
   (8,2,1) &E_7A_1\\
  \ccy (8,4,0) &D_4^2\\
   (8,4,1) &D_6 A_1^2\\
  \ccy (8,6,0) &A_1^8 *\\
   (8,6,1) &D_4 A_1^4\\
  \ccy (8,8,0) &E_8(2)\\
   (8,8,1) &A_1^8\\
  \hline
   (9,1,1) &E_8A_1\\
   (9,3,1) &E_7 A_1^2\\
  \ditto &D_8A_1\\
   (9,5,1) &D_6 A_1^3\\
  \ditto &D_4^2A_1\\
   (9,7,1) &A_1^9*\\
  \ditto &A_1^5 D_4\\
   (9,9,1) &A_1^9\\
  \ditto &A_1 E_8(2)\\
  \hline
   (10,2,1) &D_{10}\\
  \ditto &E_8 A_1^2\\
   (10,4,1) &E_7 A_1^3\\
  \ditto &D_8 A_1^2\\
  \ditto &D_6 D_4\\
   (10,6,1) &D_4^2 A_1^2*\\
  \ditto &A_1^6 D_4 *\\
  \ditto &D_6 A_1^4\\
   (10,8,1) &D_4 A_1^6 *\\
  \ditto &A_1^{10} *\\
  \ditto &A_3 E_7(2) *\\
  \hline
\end{tabular}
\renewcommand{\arraystretch}{1.157}
\kern -6.5pt
\begin{tabular}[t]{c>{$}l<{$}|}
  \hline
  \text{Case} &\text{Lattice} \\
  \hline
   (11,3,1) &D_{10}A_1\\
  \ditto &E_8 A_1^3\\
  \ditto &E_7 D_4\\
   (11,5,1) &D_6D_4A_1\\
  \ditto &D_8A_1^3\\
  \ditto &E_7A_1^4\\
  \ditto &D_6A_1^5*\\
   (11,7,1) &D_6A_1^5\\
  \ditto &A_1^7 D_4 *\\
  \ditto &D_4^2A_1^3\\
  \ditto &A_5 E_6(2) * \\
  \hline
  \ccy (12,2,0) &E_8 D_4\\
  \ccy\ditto &D_{12}\\
  \ccy (12,4,0) &E_7 A_1^5 *\\
  \ccy\ditto &D_8 D_4\\
   (12,4,1) &D_8 A_1^4 *\\
  \ditto &E_8 A_1^4\\
  \ditto &D_6^2\\
  \ditto &D_{10} A_1^2\\
  \ditto &E_7 D_4 A_1\\
  \ccy (12,6,0) &D_4^3\\
  \ccy\ditto &E_6 E_6(2) *\\
  \ccy\ditto &D_6 A_1^6*\\
   (12,6,1) &D_6A_1^6*\\
  \ditto &E_7A_1^5\\
  \ditto &D_4^2 A_1^4*\\
  \ditto &D_8A_1^4\\
  \ditto &D_6D_4A_1^2\\
  \ditto &A_7 D_5(2)*\\
  \hline
   (13,3,1) &D_{12}A_1\\
  \ditto &E_7D_6\\
  \ditto &E_8D_4A_1\\
  \ditto &D_{10}A_1^3*\\
   (13,5,1) &D_8 D_4 A_1\\
  \ditto &E_7 D_4 A_1^2\\
  \ditto &D_6^2 A_1\\
  \ditto &E_7 A_1^6 *\\
  \ditto &D_6 D_4 A_1^3 *\\
  \ditto &D_{10} A_1^3\\
  \ditto &E_8 A_1^5\\
  \ditto &D_8 A_1^5 *\\
  \ditto &A_9 A_4(2) *\\
  \hline
\end{tabular}
\renewcommand{\arraystretch}{1.156}
\kern -6.5pt
\begin{tabular}[t]{c>{$}l<{$}|}
  \hline
  \text{Case} &\text{Lattice} \\
  \hline
   (14,2,1) &D_{12} A_1^2*\\
  \ditto &D_{14}\\
  \ditto &E_8 D_6\\
  \ditto &E_7^2\\
   (14,4,1) &D_6^2 A_1^2 *\\
  \ditto &E_8 D_4 A_1^2\\
  \ditto &E_7 D_6 A_1\\
  \ditto &D_{10} A_1^4 *\\
  \ditto &D_8 D_4 A_1^2 *\\
  \ditto &E_7 D_4 A_1^3 *\\
  \ditto &D_{12} A_1^2\\
  \ditto &D_{10} D_4\\
  \ditto &D_8 D_6\\
  \ditto &A_{11} (A_2 A_1)(2) *\\
  \hline
   (15,1,1) &E_8 E_7\\
  \ditto &D_{14} A_1*\\
   (15,3,1) &E_7^2 A_1\\
  \ditto &D_8 E_7\\
  \ditto &D_{14} A_1\\
  \ditto &E_8 D_6A_1\\
  \ditto &D_{12} A_1^3*\\
  \ditto &D_{10} D_4A_1*\\
  \ditto &D_8 D_6 A_1*\\
  \ditto &E_7 D_6 A_1^2*\\
  \ditto &A_{13} A_1(2)\!*\!* \\
  \hline
  \ccy (16,0,0) &E_8^2\\
  \ccy\ditto &D_{16}* \\
  \ccy (16,2,0) &D_8^2*\\
  \ccy\ditto &E_7^2 A_1^2*\\
  \ccy\ditto &E_8 D_8\\
  \ccy\ditto &D_{12} D_4*\\
  \ccy\ditto &D_{16}\\
  \ccy\ditto &A_{15} A_1(2)*\\
   (16,2,1) &E_8 E_7 A_1\\
  \ditto &D_8 E_7 A_1 *\\
  \ditto &D_{10} D_6 *\\
  \ditto &D_{14} A_1^2 *\\
  \ditto &A_{15}\!*\!*\\
  \hline
   (17,1,1) &E_8^2 A_1\\
  \ditto &D_{16} A_1*\\
  \ditto &D_{10} E_7*\\
  \ditto &A_{17}*\\
  \hline
\end{tabular}
\medskip
  \caption[Negative definite $2$-elementary lattices]{$2$-elementary negative
    definite lattices appearing at $1$-cusps}
  \label{tab:neg-def}
\end{table}

\begin{theorem}\label{thm:1-cusps}
  On the diagram \fign\ the move $S\to T\leadsto\oT\leadsto\ooT$
  can be seen by
  \begin{enumerate}
  \item making one of the three mirror moves
    $(r,a,\delta)\leadsto (\bar r,\bar a,\bar\delta)$ of
    Def.~\ref{def:mirror-moves}, 
  \item and then doing one of the following three moves:
    \begin{enumerate}
    \item odd or simple: staying at the same vertex; we set
      $\bar{\bar a}=\bar a$, $\bar{\bar\delta}=\bar\delta$.
    \item even ordinary: dropping down by $2$
      and keeping the color; here, we set
      $\bar{\bar a} = \bar a -2$,
      $\bar{\bar\delta} = \bar\delta$. 
    \item even characteristic: dropping down by $2$ and changing the
      color from black $\bar\delta=1$ to white $\bar{\bar\delta}=0$;
      we set $\bar{\bar a}=\bar a-2$.       
    \end{enumerate}
    The final vertex of \fign\ is interpreted as a negative definite
    lattice $\ooT$ with the invariants
    $(\bar{\bar r}, \bar{\bar a},\bar{\bar\delta})$, where $\bar{\bar
      r}= \bar r-2=18-r$, and $\bar{\bar a}$, $\bar{\bar\delta}$ are as
    set above. 
  \end{enumerate}

  \begin{center}
    \includegraphics{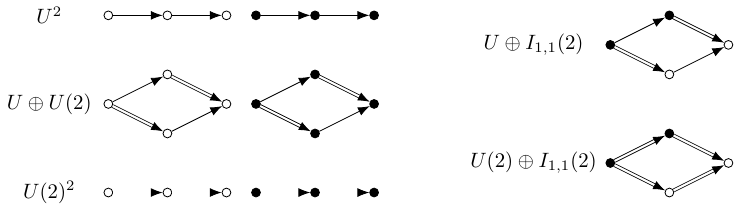}
  \end{center}

  A $1$-cusp of this form exists iff \fign\ allows the two-move
  combination. For each isomorphism class of $\ooT$ there is a unique
  $O(T)$-orbit of the $1$-cusps.

  A $0$-cusp corresponding to a mirror move $S\leadsto\oT$ contains the
  cusps $\ooT$ above iff $\ooT$ can be reached by going
  through $S\leadsto\oT$ as the first step.
\end{theorem}
\begin{proof}
  This follows by applying Lemma~\ref{lem:1-cusps} and 
  using the fact that by
  \cite[Thm.~3.6.2]{nikulin1979integer-symmetric} the allowed
  invariants of even $2$-elementary hyperbolic lattices and of
  negative definite lattices are in a bijection with a shift
  $\bar{\bar r} = \bar r-2$, so we can reuse \fign\ instead of making
  a new figure for the invariants of $\ooT$ that occur.
\end{proof}

\begin{theorem}\label{thm:2-elem-negdef}
  For the $2$-elementary lattices $S$ of \fign, the isomorphism
  classes of the even negative definite $2$-elementary lattices $\ooT$
  appearing for the $1$-cusps of $\bD_S/O(T)$ are uniquely determined
  by their invariants $(r,a,\delta)$ and the root sublattices $R$, 
  as listed in Table~\ref{tab:neg-def}. 
\end{theorem}
Notations of Table~\ref{tab:neg-def} are as follows: $R$ means
that $\ooT$ is the root lattice $R$; $R*$ that the root sublattice
$R\subset \ooT$ is a finite index sublattice and $\ooT$ is
obtained from $R*$ by adding a glue.  $R\!*\!*$ means that $R$ has infinite
index in $\ooT$. The root lattices are given for the entire reflection
group $\wref$ as explained in Section~\ref{sec:reflection-groups}.

\begin{proof}
  Most of these lattices are found by listing the maximal
  parabolic subdiagrams of Coxeter diagrams of hyperbolic lattices, as
  described in Section~\ref{sec:reflection-groups}. The two
  exceptional cases are $(16,2,1)$ and $(15,3,1)$ where Vinberg's
  method does not apply. For these cases we applied Kneser's gluing
  method \cite[15.10.2]{conway1999sphere-packings}.  
\end{proof}

\begin{corollary}
  $\bD_S/O(T)$ can have a maximum three of $0$-cusps and a
  maximum $14$ of $1$-cusps, the latter happening only for $S=(4,4,1)$.
\end{corollary}

\begin{proposition}\label{prop:1-cusps-modular}
  The modular curve in $\overline{\bD/O(T)}\ubb$ corresponding to a cusp $\ooT$ is
  isomorphic to $\bH/\SL(2,\bZ)$, resp. $\bH/\Gamma_0(2)$
  if it is incident to a single $0$-cusp, resp. to two $0$-cusps.
\end{proposition}
Here, $\bH$ is the upper half plane and $\Gamma_0(2)\subset
\SL(2,\bZ)$ is the standard level-$2$ modular subgroup of index~$3$.

\begin{proof}
  Let $J\subset T$ be an isotropic plane corresponding to the
  $1$-cusp. Then the corresponding modular curve is $\bH/\Gamma$,
  where $\Gamma$ is the image of the stabilizer of $J$ in $\SL(T)$ in
  $\SL(J)\simeq\SL(2,\bZ)$. By Theorem~\ref{thm:1-cusps},
  $T=P\oplus\ooT$, with $P$ one of the five signature $(2,2)$ lattices
  listed there. Then $J$ is an isotropic plane in $P$. So the
  statement is reduced for the five cases $T=P$ for which the check is
  immediate. 
\end{proof}

\subsection{$0$-cusps and involutions of $\lias = {\rm II}_{2,18}$}
\label{sec:reverse-moves}

As before, let $\rhok$ be an involution on $\lk$, the lattices
$S,T=\lk^{\pm}$, and an isotropic vector $e\in T$. We have
$e^\perp_T\subset e^\perp_{\lk}$, where the perps are taken in
$T$, resp. in $\lk$, and
$$\oT = e^\perp_T/e \subset e^\perp_{\lk}/e \simeq {\rm II}_{2,18} =
U^2\oplus E_8^2.$$

Let us denote the last lattice by $\lias$. It comes with an induced
involution $\rhoias$ and a pair of orthogonal lattices $\lias^{\pm}$,
both hyperbolic. The $(-1)$-eigenspace is $\oT$. Since
$S\subset e^\perp_{\lk}$ and $e\notin S$, the projection
$e^\perp_{\lk} \to \lias$ embeds $S$ into $\lias^+$. But it need not
be a primitive sublattice, only its saturation is $S\usat= \lias^+$.

In this section we determine exactly which sublattices of $\lias^+$
appear as images of the lattices $S = \lk^+$.

\begin{lemma} \label{lem:invo-lias-forward}
  The inclusion $S\to \liasplus$ identifies $S$ with $\liasplus$ for
  an odd $0$-cusp, or an index~$2$ sublattice of $\liasplus$ that
  contains $(1 + \rhoias)\lias = 2(\liasplus)^*$ for an even $0$-cusp.

\end{lemma}
\begin{proof}
  The
  identity $(1+\rhoias)\lias = 2(\liasplus)^*$ follows because $\lias$
  is unimodular.
  Let $|\liasplus/S|=n$. Then $|A_S| : |A_{\liasplus}| = n^2$. We have
  $|A_{\liasplus}| = |A_{\liasminus}| = |A_\oT|$. Thus,
  $n^2 = |A_S| : |A_\oT| = 2^{a_s} : 2^{a_\oT}$.  Since $a_S = a_\oT$
  for an odd $0$-cusp, resp. $a_S = a_\oT+2$ for an even cusp,
  it follows that $n=1$, resp. $n=2$.
\end{proof}

\begin{lemma}
  Let $K$ be an even $2$-elementary lattice. Then any index $2$
  sublattice $S\subset K$ that contains $2K^*$ is $2$-elementary. Such
  sublattices are in a bijection with nonzero elements $x\in
  A_{K^\dag}$. Moreover,
  \begin{enumerate}
  \item If $\delta_K=1$ then there exists a unique
    $O(q_{K^\dag})$-orbit of $x$'s.
  \item If $\delta_K=0$ then there exist at most two
    $O(q_{K^\dag})$-orbit of $x$'s: for $x$ with
    $q_{K^\dag}(x)\equiv 0\pmod{2\bZ}$ and
    $q_{K^\dag}(x)\equiv 1\pmod{2\bZ}$. In the first case one has
    $\delta_S=0$ and in the second case $\delta_S=1$.
  \end{enumerate}
  If $K$ is indefinite, then the $O(q_{K^\dag})$-orbits are
  $O(K)$-orbits. 
\end{lemma}
\begin{proof}
  Any index $2$ subgroup of $K$ is of the form
  \begin{displaymath}
    K_x = \{ v \in K \mid v \cdot x\in \bZ \}
    \quad \text{for some } x\in \tfrac12 K^*
  \end{displaymath}
  and the condition that $2K^*\subset K_x$ means that moreover
  $x\in (2K^*)^* = \frac12 K \subset\frac12 K^*$.
  For the dual lattices we have $K_x^* = K^* + x$.

  The lattice $K_x$ is $2$-elementary $\iff$
  $2K_x^* \subset K_x$ $\iff$ $2K^* + 2x \subset K_x$. Since $K$ is
  $2$-elementary, we have $2K^* \subset K$. And $2K^* \cdot x = K^*
  \cdot (2x) \in\bZ$, so $2K^* \subset K_x$ as well.
  One has $2x \in K$, and $2x \in K_x$ iff $2x\cdot x = (2x)^2/2
  \in\bZ$. This is true because $K$ is an even lattice. This proves
  that the lattice $K_x$ is indeed $2$-elementary.

  Two elements $x_1,x_2\in\frac12 K$ define the same sublattice iff
  $x_1-x_2\in K^*$. So the set of the distinct sublattices $K_x$ is in
  a bijection with 
  \begin{displaymath}
    \tfrac12 K / K^* = K(\tfrac14) / K^* \simeq K(\tfrac12) / K^*(2)
    = A_{K^\dag},
  \end{displaymath}
  see Lemma~\ref{lem:doubled-dual}. Since $K$ is even, $K^\dag$ is
  co-even, so $q_{K^\dag}(x)\in\bZ$ for any $x\in A_{K^\dag}$.

  If $\delta_K=1$ then $K^\dag$ is odd and $q_{K^\dag}$ is well
  defined only mod $\bZ$, so $q_{K^\dag}(x)\equiv 0 \pmod\bZ$ for any $x\in
  A_{K^\dag}$. 
  If $\delta_K=0$ then $K^\dag$ is even and $q_{K^\dag}$ is well
  defined mod ${2\bZ}$, and $q_{K^\dag}(x)\equiv 0$ or $1 \pmod\bZ$. Then
  $\delta_{K_x}=1$ iff $K_x^\dag$ is odd $\iff$ $q_{K^\dag}(x)\equiv 1$.

  The statement about the $O(q_{K^\dag})$-orbits follows by 
  \cite[3.9.1]{nikulin1979integer-symmetric}.
  We have $O(q_{K^\dag}) = O(q_K)$, and if $K$ is indefinite then
  $O(K)\to O(q_K)$ is surjective. 
\end{proof}

\begin{theorem}\label{thm:reverse-direction}
  Let $\rhoias\colon\lias\to\lias$ be an involution.  Then any
  sublattice $S\subset K$ of index~$1$ or $2$ containing $2K^*$ is
  $2$-elementary, and the $O(q_{K^\dag})$-orbits of such sublattices
  are in a bijection with the sources of mirror moves
  $S\leadsto\oT$ with 
  $\oT\simeq\liasminus$. The lattice $\liasplus$ itself corresponds to
  the simple mirror of $\oT$ and the sublattices of index~$2$
  correspond to the even, non-simple mirror moves.
\end{theorem}
\begin{proof}
  We apply the previous lemma to $K=\lias^+$ and 
  check that the index~$2$ sublattices corresponding to
  $0\ne x\in A_{K^\dag}$ are indeed in a bijection with those that are
  allowed by \fign.
  There are two special cases to consider:
  \begin{enumerate}
  \item The lattices $S$ on the $r=a$ line, where according to \fignb
    there should be no index~$2$ sublattices. Indeed, in this case
    $S=L(2)$ with a unimodular $L$, so $S^\dag = L$ and
    $A_{S^\dag} =0$ has no nonzero elements.
  \item $(6,4,0)$ where all the index~$2$ sublattices must be
    characteristic. In this case $S=U(2)\oplus D_4$. So $\delta_S=0$ and
    $S^\dag = U\oplus D_4$. Then the discriminant form is
    $q(A_{S^\dag})=v$ and every nonzero element $x\in A_{S^\dag}$
    satisfies $q(x)\equiv 1\pmod {2\bZ}$. So this is consistent with \fign.
  \end{enumerate}

  As we noted at the end of Section~\ref{sec:2-elem-lattices}, all
  lattices in \fignb can be written as direct sums of the standard
  ones. Computing the doubled duals for them one checks that in the
  remaining, non special cases one has $A_{S^\dag}\ne 0$, and if
  $\delta_{S^\dag}=0$ then there are elements both with
  $q(x)\equiv 0\pmod {2\bZ}$ and $q(x)\equiv 1\pmod{2\bZ}$
  corresponding to both even ordinary and even characteristic
  moves. The answer given by \fignb is the same in all of these cases.
\end{proof}

\section{Degenerations and integral affine spheres} 
\label{sec:kulikov+ias}

The material of this section is well explained in
\cite{engel2018looijenga, engel2021smoothings, alexeev2019stable-pair,
  alexeev2021compact}. So we give a brief summary and fix notations.

\subsection{Kulikov models}
\label{sec:kulikov}

We discuss one parameter degenerations of K3 surfaces.
\cite{friedman1986type-III} is a useful reference.

\begin{definition}
  Let $X^*\to C^*$ be a family of smooth complex K3 surfaces over a
  punctured curve or disk $C^* = C\setminus 0$. A {\it Kulikov}, or
  {\it Kulikov-Persson-Pinkham, model} is a proper extension $X\to C$ such that
  $X\to C$ is semistable, i.e. $X$ is nonsingular and the central
  fiber $X_0 = \cup V_i$ is a reduced normal crossing union of
  divisors, and additionally, one has $K_X \sim_C 0$. 
  The central fiber $X_0$ is called a {\it Kulikov surface}.
\end{definition}
A Kulikov model exists after a finite base change $C'\to C$ by
\cite{kulikov1977degenerations-of-k3-surfaces,
  persson1981degeneration-of-surfaces}.

There are three types of Kulikov models, depending on the image of
$0\in C$ under the extended period morphism
$$f\colon C\to \overline{\bD/\Gamma}\ubb$$ to the Baily-Borel
compactification:
\begin{enumerate}
\item[(I)] $f(0)$ lies in the interior $\bD/\Gamma$: $X_0$ is smooth.
\item[(II)]
  $f(0)$ lies in a $1$-cusp.
  The dual complex of $X_0$ is an interval of length $n$,
  $D_{0,1} = \dotsb = D_{n-1,n} = E$ is the same elliptic
  curve, $V_0$ and $V_n$ are rational, and for $0<i<n$ the surfaces
  $V_i\to E$ are generically ruled. The double locus
  is an anticanonical divisor on each component $V_i$.
\item [(III)]
  $f(0)$ lies in a $0$-cusp.
  The dual complex $\Gamma(X_0)$ is a triangulation of the
  sphere~$S^2$. Each $(V_i, D_i)$ is an anticanonical pair,
  i.e. $K_{V_i} + D_i\sim 0$, where the part of the double locus $D_i$ contained
  in $V_i$ is a wheel of rational curves. 
\end{enumerate}

In Type III, all components $V_i$ are necessarily rational.
The central fibers $X_0$ of Kulikov models are called Type I, II, III Kulikov
surfaces, respectively. Denote $D_{ij} = V_i\cap V_j$
so that $D_i = \sum_j D_{ij}$. Then the dual complex $\Gamma(X_0)$
of a Type III Kulikov surface consists of vertices $v_i$ corresponding to components $V_i$,
edges $e_{ij}$ corresponding to double curves $D_{ij}$, and triangles $t_{ijk}$ corresponding
to triple points $T_{ijk}=V_i\cap V_j\cap V_k$.

The Picard-Lefschetz transform
$T\colon H^2(X_t,\bZ)\to H^2(X_t,\bZ)$ takes
the form $T={\rm exp}(N)$ where $$N(x) = (x\cdot \lambda)e-(x\cdot e)\lambda$$
for $e\in H^2(X,\bZ)$ primitive isotropic and $\lambda\in e^\perp/e$ a vector
satisfying $\lambda^2=\textrm{the}$ number of triple points of $X_0$. We call $\lambda$
the {\it monodromy invariant}. When $\lambda^2>0$ (so $X_0$ is of Type III),
$I=\bZ e$ defines the $0$-cusp $f(0)$, and when $\lambda^2=0$, $\lambda\neq 0$
(so $X_0$ is of Type II),
$J=(\bZ e\oplus \bZ \lambda)^{\rm sat}$ defines the $1$-cusp containing $f(0)$.

\medskip

Two Kulikov models $X$ of the same degeneration $X^*\to C^*$
differ by a sequence of flops in curves $E\subset X_0$.
The central fiber $X_0$ is then changed by a
sequence of elementary modifications of the following types:
\begin{enumerate}
\item[(M0)] $E\subset V_i$, $E\cap D_i=\emptyset$ is an interior
  $(-2)$-curve. The flop is a nontrivial birational transformation
  $X\ratmap X'$ but $X'_0=X_0$ are canonically identified.
\item[(M1)] $E\subset V_i$, $E^2=-1$, $E\cap D_i=p\in D_{ij}$ is a
  smooth point of $D_i$. The flop contracts $E$ on $V_i$ to $p$ and
  blows up $p\in V_j$ to create a $(-1)$-curve $E'\subset V_j$.
\item[(M2)] The flop contracts $E=D_{ij}$ which is a $(-1)$-curve on
  both $V_i$ and $V_j$ and inserts a curve $D_{kl}$ between their
  neighbors $V_k,V_l$.
\end{enumerate}

\medskip

\begin{definition}
  Let $(V, D=\sum D_j)$ be an anticanonical pair.  A
  \emph{corner blowup} is a blowup $f\colon V'\to V$ at a node of $D$.
  The anticanonical divisor of $V'$ is $D'=f\inv(D)$. 
  
  An \emph{interior, or almost toric blowup} $V''\to V$ is a blowup at an
  interior point of a curve $D_{j}$, i.e. at a point
  $p\in D_j \setminus \cup_{k\ne j} D_k$. The anticanonical
  divisor $D''\subset V''$ is the strict preimage of $D$.
\end{definition}

\begin{lemma}[\cite{gross2015moduli-of-surfaces}]
  \label{def:toric-model}
  For any anticanonical pair $(V,D)$ there exists a
  diagram $V \gets V' \to \oV$, called a {\rm toric model}, such that $V'\to V$ is a
  sequence of corner blowups and $V'\to \oV$ is a sequence of
  interior blowups. 
\end{lemma}

We order the blowups and call $V \gets V' \to \oV$ \emph{the ordered
  toric model} of $V$. We also fix the origin $1\in
(\bC^*)^2\subset\oV$. This defines a choice of origin on every
boundary curve $D_j$ of $V$.

\begin{definition}\label{def:charge}
  The \emph{charge} $Q(V,D)$ of an anticanonical pair is 
  the number of the interior blowups in a toric model.
  Equivalently, one has
  \begin{displaymath}
    Q(V,D) =
    \begin{cases}
      12 + \sum(-D_j^2-3) &\text{ if } D \text{ is nodal with } \ge2
      \text{ components,}  \\
      11-D^2 & \text{ if } D \text{ is irreducible nodal.}\\
    \end{cases}
  \end{displaymath}
\end{definition}

Type III surfaces are in a sense $24$ steps away from being toric:

\begin{theorem}[Friedman-Miranda \cite{friedman1983smoothing-cusp}]
  \label{thm:friedman-miranda}
   For a Type III Kulikov surface,
  one has $\sum Q(V_i,D_i)=24$. 
\end{theorem}

Conversely, suppose we have a collection of anticanonical pairs $(V_i,D_i)$
and identifications $D_{ij}\to D_{ji}$ whose dual complex forms a $2$-sphere,
such that $D_{ij}^2+D_{ji}^2=-2$.
Then a $d$-semistable (in the sense of Friedman
\cite{friedman1983global-smoothings}) gluing gives a
Type III surface which admits a smoothing to a  
K3 surface. We refer to \cite{alexeev2021compact} for details.

If we fix the numerical type and the ordered toric models for each
$(V_i,D_i)$, then we can construct \emph{the standard}
Type III surface as follows: the interior blowups
$V'\to\oV$ are all done at the point $-1$ on each boundary component
$D_{ij}\simeq \bP^1$,
with respect to the chosen origins. Each
identification $D_{ij}\to D_{ji}$ is chosen to be the unique
isomorphism matching the origins and the triple points. The
standard surface is always $d$-semistable.

All the other gluings are defined by varying the points of
nontoric blowups and the differences between the origins in $D_{ij}$
and $D_{ji}$, modulo the changes of the origins in each $\oV_i$.  This
defines a gluing complex \cite[Def.~5.10]{alexeev2021compact}. The
final result is that the $d$-semistable Type III surfaces of a fixed
numerical/combinatorial type are parameterized by the $19$-dimensional torus
$\Hom(\Lambda, \bC^*)$, where $\Lambda\simeq\bZ^{19}$ can be defined
from the gluing complex or, equivalently, from the Picard complex:
\begin{equation}\label{eq:picard-complex}
  \wt\Lambda = \ker \left(\oplus_i\Pic V_i \to \oplus_{i<j} \Pic D_{ij}\right),
  \quad \Lambda = \wt\Lambda/\Xi, 
\end{equation}
where $\Xi = \la \xi_i \ra / (\sum \xi_i)$ 
and $\xi_i = \sum_j (D_{ij}-D_{ji}) \in \bigoplus_i \Pic V_i$.  For a
given smoothing with a $0$-cusp $e$ and monodromy vector
$\lambda\in e^\perp/e$ one has $\Lambda = \lambda^\perp$ in
$e^\perp/e \simeq {\rm II}_{2,18}$. See \cite{alexeev2021compact}, Sec.~5B
and Prop.~3.29.

The lattice $\wt\Lambda$ is of \emph{numerical Cartier
  divisors}, which are the numerical possibilities for the
restrictions of a line bundle on $X_0$ to $V_i$ and $D_{ij}$. The
elements~$\xi_i$ represent the line bundles $\cO_{X_0}(-V_i)$ which
are defined on any Kulikov model $X$ containing $X_0$ as the central
fiber.  The homomorphism
\begin{equation}\label{eq:period}
  \psi\colon\Lambda\to\bC^*, \ \text{resp. } E
  \quad\text{for } X_0 \text{ of type III, resp. II}
\end{equation}
is the \emph{period} of~$X_0$.  For a Type II surface $X_0$,
$E=\Pic^0 D_i$ is the elliptic curve for any of the isomorphic
double curves $D_i$.

The Picard group of $X_0$ is $\ker\wt\psi$,
where $\wt\psi$ is the composition
$\wt\Lambda \to \Lambda \xrightarrow{\psi}\bC^*\textrm{ or }E$.

\begin{definition}\label{def:reduced-picard}
  The \emph{reduced Picard group} is
  \begin{math}
    \oPic (X_0) = \ker\psi = \Pic(X_0) / \Xi.
  \end{math}
  If $X_0$ is the central fiber of a smoothing, it is the quotient of
  $\Pic(X_0)$ by $\xi_i=\cO_{X_0}(-V_i)$.
\end{definition}
For a standard surface one has $\psi\equiv 1$, so $\oPic(X_0) =
\Lambda$.

\subsection{Nef, divisor, and stable models}
\label{sec:nef-divisor-models}

\begin{definition}
  Let $L^*$ be a line bundle on $X^*$, relatively nef and big over $C^*$.
  A relatively nef extension $L$ to a Kulikov model $X\to C$
  is called a {\it nef model}.
\end{definition}

\begin{definition}
  Let $R^*\subset X^*$ be the vanishing locus of a
  section of $L^*$ as above, containing no vertical
  components. A {\it divisor model} is an extension $R\subset X$ to a relatively
  nef divisor $R\in |L|$ for which $R_0$ contains no strata of $X_0$.
\end{definition}

\begin{definition}\label{pass-to-stable} The {\it (KSBA-)stable model} $(\oX,\epsilon \oR)$
is ${\rm Proj}_C\,\bigoplus_{n\geq 0}\pi_*(nR)$ for some divisor model $\pi\colon (X,R)\to (C,0)$.
It is unique, depending only on the family $(\oX^*,\oR^*)\to C^*$, and stable under base change.
We call $(\oX_0,\epsilon \oR_0)$ the {\it stable limit}.
 \end{definition}

\begin{definition}
  For an arbitrary, not necessarily nef effective divisor $R^*$ on
  $X^*$, we say that a Kulikov model $X\to C$ is \emph{compatible with
    a divisor} if its closure $R$ does not contain any strata of the
  central fiber $X_0$. 
\end{definition}

If $X\to C$ is a family with an involution $\iota$ and $R=X^\iota$ is
the fixed locus then $(X,R)$ is usually not a nef model, since $R_t$
contain $(-2)$ curves $E_{i,t}$. Only the part of $R$ which is the
family $C_g$ of curves of genus $g\ge2$ may give a divisor or stable model. 

\subsection{$\ias$ from Kulikov surfaces}
\label{sec:ias}

More details of the constructions in the following three sections,
with pictures, are given in \cite{engel2018looijenga,
  engel2021smoothings, alexeev2019stable-pair}.

\smallskip

For a Type III Kulikov surface $X_0=\cup V_i$ the dual graph is a
triangulation of a sphere. This is a very rough, partial
description. To describe the combinatorial type of $X_0$ precisely, one
has to specify the deformation types of each pair $(V_i,D_i)$.
There is an economical way to do so, using the language of
integral affine structures $B$ on the complement of finitely many
points in a sphere $S^2$.

\begin{definition} 
An \emph{integral affine structure} on a real oriented surface $S$ is
a collection of charts to open subsets of $\bR^2$, with transition
functions in $\SL(2,\bZ)\ltimes\bR^2$.
\end{definition}

For a Type III Kulikov model, we endow $\Gamma(X_0)\setminus \{v_i\,\big{|}\,Q(V_i,D_i)>0\}$
with an integral-affine structure as follows. Each triangle is declared equivalent
to a lattice triangle of the smallest possible lattice volume $1$. Any two such are equivalent
up to $\SL_2(\bZ)\ltimes \bZ^2$. Cyclically order the directed edges $\vec{e}_{ij}$ 
emanating from a vertex $v_i$ so that $j$ is increasing by $1$ on successively
counterclockwise edges. Then, to extend the integral-affine structure to the interiors of edges,
we glue two lattice triangles together according to the formula
\begin{displaymath}
  \vec e_{i, j-1} + \vec e_{i, j+1} = d_{ij}\vec e_{ij}, \quad
  \text{where }
  d_{ij}=\begin{cases}
    -D_{ij}^2 & \text{ if } D_{ij} \text{ is smooth,}\\
    -D_{ij}^2+2 & \text{ if } D_{ij} \text{ is rational nodal.}
    \end{cases}
\end{displaymath}

Let ${\rm star}(v_i)$ denote the union of the triangles containing 
$v_i$. The integral-affine structure on $\Gamma(X_0)\setminus \{v_i\}$ extends to the vertices
$v_i$ for which $Q(V_i,D_i)=0$, i.e.~when $(V_i,D_i)$ is toric. By a well-known
formula in toric geometry, ${\rm star}(v_i)$
admits a chart to a polygon in $\bR^2$ whose vertices are the endpoints of
the primitive integral vectors in the fan of $(V_i,D_i)$.
Thus, in analogy with the toric case, we define (dropping
the index $i$ for notational convenience):

\begin{definition} The {\it pseudofan} of $(V,D)$ is
the integral-affine surface ${\rm star}(v)$ constructed from gluing
lattice triangles as above, one for each node of $D$. \end{definition}

It is an integral-affine surface with boundary,
PL isomorphic to the cone over the dual complex of $D$ with (up to)
one singularity at the cone point.

An alternative description is in terms of a toric model
$V\gets V'\to \oV$ of in Definition~\ref{def:toric-model}. In the fan
of $\oV$, let $\vec e_{j}$ denote the primitive integral generators of the rays, in
the counterclockwise order. The morphism $V'\to\oV$ is a sequence of
interior blowups, say $n_{j}$ times on the side $\oD_{j}$. Then the
pseudofan $\fF(V',D')$ is obtained from $\fF(\oV,\oD)$ by regluing
along each edge $\vec e_{j}$ with a shearing transformation
$I_{n_j}(\vec e_j) = I(\vec e_j)^{n_j}$, see \cite[Prop.~3.13]{engel2018looijenga}.
Here, $$I(\vec e_j)\sim \begin{pmatrix} 1&1\\0&1 \end{pmatrix}$$
is the unique linear transformation in $\SL_2(\bZ)$ conjugate to a unit shear
which has $\vec e_j$ as an eigenvector.
The pairs $(V,D)$ and $(V',D')$ differ by corner blowups and the
neighborhoods of the origins in $\fF(V,D)$ and $\fF(V',D')$ are
isomorphic. Only the polyhedral subdivision changes.
So they define the same integral-affine singularity.

\begin{definition}\label{def:cbec}
  Two pseudofans belong to the same \emph{corner blowup equivalence
    class (cbec)} if they correspond to different toric models of
  the same anticanonical pair, possibly after some corner blowups and topologically
  trivial deformations.
\end{definition}

\begin{definition}\label{def:ias}
  An $\ias$ is an integral-affine structure on
  $S^2\setminus \{p_1,\dotsc, p_k\}$, together with
  the data of a cbec for which $\fF(V,D)$ models a neighborhood of $p_i$.
    \end{definition}

\begin{construction}\label{con:ias-from-kulikov}
  Each $d$-semistable Type III surface $X_0=\cup V_i$ defines a
  triangulated (by unit lattice triangles) $\ias$ $(B,\cT)$ as follows: $B= \cup_i \fF(V_i,D_i)$, and
  $\cT$ is the dual complex $\Gamma(X_0)$. By
  \cite[Prop.~2.2]{engel2018looijenga}, the consistency of the
  integral-affine structure on the union of adjacent pseudofans
  is equivalent to the triple point formula $D_{ij}^2 + D_{ji}^2=-2$ which
  holds on any Kulikov surface.
  
  Vice versa, from any triangulated $\ias$, one can
  construct a Type III surface by interpreting the star of each vertex as
  an anticanonical pair $(V_i,D_i)$ and gluing them along
  identifications $D_{ij}\to D_{ji}$. As explained in
  Section~\ref{sec:kulikov}, this way one obtains a family of type III
  surfaces parameterized by the torus
  $\Hom(\Lambda,\bC^*) \simeq (\bC^*)^{19}$.
\end{construction}

  By Theorem~\ref{thm:friedman-miranda}, the sum of charges of
  singularities in $24$. An $\ias$ is called \emph{generic} if there
  are $24$ distinct $I_1$ singularities.

\subsection{$\ias$ from symplectic geometry}
\label{sec:ias-symplectic}

A second source of integral-affine structures is symplectic geometry.
Let $(\hX,\omega)\to B$ be a smooth symplectic $4$-manifold with a
Lagrangian torus fibration such that the singular fibers are necklaces
of spheres. Then it defines a natural integral-affine structure on the
base minus finitely many points and with the $I_n$ singularity
$b\in B$ at a point where the fiber $\hX_b$ is a necklace of $n$
spheres. Vice versa, any integral affine structure $B$ on a sphere of
total charge $24$ with only $I_n$ singularities defines a unique
symplectic manifold $(\hX,\omega)\to B$ with $\hX$ diffeomorphic to a
K3 surface.  See \cite{engel2021smoothings},
\cite[Sec.~8E]{alexeev2019stable-pair} for more details.

The easiest examples of integral-affine structures
coming from symplectic geometry are those from the following construction
due to Symington \cite{symington2003four-dimensions}:

\begin{construction}\label{con:ias-antican}
  Let $(\hV,\hD)$ be an anticanonical pair with a toric model
  $$\hV \xleftarrow{f} \hV'\xrightarrow{g} \ohV.$$ Choose a
  big and nef line bundle $L$ on $\hV$ and let $L'=f^*L$ and
  $\oL=g_*L'$; they are big and nef as well.  The line bundle $\oL$
  defines the moment map $$\bar\mu\colon\ohV\to \oP$$ to the moment
  polytope. It is a Lagrangian torus fibration, with circle and point fibers
  over the edges and vertices of $\oP$, respectively.
  
  Suppose that $g\colon \hV'\to \hV$ contains $n$ internal blowups on
  the side $\oD_j$ and that $L' = g^*(L) - \sum a_{jk}E_{jk}$, where
  $E_{jk}$ are the exceptional divisors. Then, as described in \cite{symington2003four-dimensions},
 one can define a \emph{Symington polytope} $P$, obtained from $\oP$ by
  cutting $n_j$ triangles $t_{jk}$ of sizes $a_{jk} = L'\cdot E_{jk}$ resting
  on the edge of $\oP$ over which $\oD_j$ fibers, and
  then gluing the remaining two edges by a unit shear. The resulting
  integral affine structure has $n_j$  $I_1$ singularities of
  integral-affine structure at interior points, with monodromy-invariant
  direction parallel to the edge on which the triangles rested.
  
Assuming all the introduced singularities are distinct,
there is a fibration $$\mu'\colon \hV'\to V$$ whose fiber
over an $I_1$ singularity is an irreducible nodal sphere.
 It then descends to
  a Lagrangian torus fibration $$\mu\colon (\hV,\omega)\to P$$
  because $L'\cdot \hD_j'=0$ for the components of $\hD'$
  introduced by the corner blow-ups $f$.

  This construction is only possible if there exists a toric model for
  which the triangles can be fit into $\oP$ without intersections. By
  \cite[Thm.~5.3]{engel2021smoothings}
  such a toric model always exists because $L$ is big and nef.
  
  An important generalization of this construction is to the case of
  an anticanonical pair $(\hV,\hD)$ with a smooth boundary and a big and nef
  line bundle $L$ on it.  If $(\hV,\hD)$ is a deformation of a Looijenga pair
  $(\hV,\hD')$ with a singular boundary $\hD'$, one
  can define a Symington polytope $P$ by a \emph{node smoothing}
  surgery from a Symington polytope $P'$ of $(\hV,\hD')$ and a
  Lagrangian torus fibration $\mu\colon \hV\to P$.
\end{construction}

For K3 surfaces, the basic example of
$\ias$ from symplectic geometry comes from the following construction:

\begin{proposition}\label{def:ias-k3-symplectic}
  Let $(\hX,L)\to (C,0)$ be a nef model with a
  central fiber $\hX_0=\cup V_i$ and $L\big{|}_{\hV_i}$ big.
  Then there is a Lagrangian torus
  fibration $$\mu\colon (\hX_t,\omega_t)\to B=\cup_i P_i$$ from a smooth fiber $\hX_t$ 
  for which $[\omega_t]=L$, defined as a
  composition of two maps:
  \begin{enumerate}
  \item the Clemens collapse $c_t\colon \hX_t\to \hX_0$ and 
  \item the union of moment maps $\mu_i\colon \hV_i\to P_i$ from
    Construction~\ref{con:ias-antican}, associated to the big and nef
    line bundle $L_i$.
  \end{enumerate}
\end{proposition}

\begin{proof} The Lagrangian torus fibrations $\mu_i$ 
undergo symplectic boundary reduction, collapsing to circles
over the edges of $P_i$ and to points over vertices of $P_i$.
The Clemens collapse $c_t\colon \hX_t\to \hX_0$
is the restriction of a retraction of $\hX\to \hX_0$. The
fibers of $c_t$ over the double locus are circles, and the fibers
over the triple points are $2$-tori. Thus, the composition 
$(\mu_i)\circ c_t$ has $2$-torus fibers, even over the edges
and vertices of $P_i$.

The general fiber $\hX_t$ can be constructed as a fiber connect-sum
of Lagrangian torus fibrations over $P_i$ which undergo no boundary reduction.
Since we additionally have $L\cdot \hD_{ij} = L\cdot \hD_{ji}$,
this fiber connect-sum can be performed as
a symplectic fiber connect-sum of Lagrangian torus fibrations, by slightly enlarging
the bases of the Lagrangian torus fibrations, and then gluing over the neighborhood of a
glued edge $P_i\cap P_j$.

Thus $\hX_t$ is endowed with a symplectic form $\omega_t$ for which
the composition $(\mu_i)\circ c_t$ is a Lagrangian torus fibration
over $S^2$ with $I_1$ singular fibers over the singular points of the Symington
polytopes $P_i$. The arguments in \cite[Prop.~3.14]{engel2021smoothings},
\cite[Thm.~2.43]{alexeev2019stable-pair} show that $[\omega_t]=c_1(L)\in \hat{e}^\perp/\hat{e}$,
where $\hat{e}\in H^2(\hX,\bZ)$ denotes the Lagrangian fiber class of $\mu$.
\end{proof}

\subsection{Nodal slides and scaling $\ias$}
\label{sec:surgeries-ias}

A \emph{nodal slide} on an $\ias$ $B$ is an operation
$B\dashrightarrow B'$ which moves some $I_n$
singularity by a specified lattice length,
in the direction of its monodromy ray. The integral-affine
structure the same on the complement of the segment
along which the singularity moves, and is only modified
along the segment. It has an interpretation
both on the algebraic side for the dual complex $B=\Gamma(X_0)$,
and on the symplectic side for the base $\hX\to B$.

The symplectic $4$-manifolds $(\hX,\omega)\to B$ and $(\hX',\omega')$
are symplectomorphic with only the Lagrangian fibration deforming.
On the algebraic side a unit length nodal slide $B\dashrightarrow B'$ 
corresponds to applying an M1 modification $X_0\dashrightarrow X_0'$.
It moves the location of the singularity by a length $1$ nodal slide,
in the direction of the monodromy ray corresponding to the exceptional
curve $E$ which was flopped.

\smallskip

The \emph{scaling} operation on $B$ corresponds to post-composing
the charts with multiplication by $n>0$. The area of $B$ is multiplied by $n^2$. 

On the symplectic side, it corresponds
to replacing $\omega\mapsto n\omega$ or $L\mapsto nL$, while leaving the Lagrangian
torus fibration the same. On the algebraic
side, it corresponds to a ramified base change $(C',0)\to (C,0)$ of
degree $n$, and resolving $X\times_C C'$ in a standard way to a
Kulikov model. Each triangle in the triangulation $\cT$ is replaced by
the standard subdivision into $n^2$ triangles. 

\smallskip

\smallskip

A \emph{node smoothing}, called a nodal trade in 
 \cite[Sec.~6]{symington2003four-dimensions}, trades a corner of the
Symington polytope $P$ for a singularity inside $P$, smoothing the
corner. It occurs when a nodal slide hits a wall of $P$.


\begin{remark}\label{rem:assume-generic}
By nodal slides, any integral affine structure
with a given decomposition of singularities into
$\prod I(\vec e_j)^{n_j}$ can be replaced by an affine structure with
only $I_n$ singularities or, by further nodal slides with $24$
distinct $I_1$ singularities. On the algebraic side, in our
Definition~\ref{def:ias}, each singularity comes with a decomposition
$\prod I(\vec e_j)^{n_j}$. So after a base change any Kulikov model
can be replaced, after a sequence of M2 modifications (retriangulation)
and M1 modifications (nodal slides), 
with a generic Kulikov model with exactly $24$ non-toric components
$V_i$ defining $24$ distinct $I_1$ singularities of an $\ias$. 
\end{remark}

\subsection{The Mirror Theorem}
\label{sec:mirror-symmetry}

A key to our ability to understand Kulikov models is the mirror
symmetry between algebraic geometry of degenerations (the A-model) and
symplectic geometry (B-model) which is well-studied in the
literature. For example, it appears in
\cite{gross2003affine-manifolds, kontsevich2006affine-structures}. 
We use it in the following form:

\begin{theorem}
  [\cite{engel2021smoothings}, Prop.~3.14]
  \label{thm:EF-mirror}
  Let $B$ be a generic $\ias$ and let 
  \begin{enumerate}
  \item $X\to (C,0)$ be a type III Kulikov model for which $\Gamma(X_0)=B$.
  \item $\mu\colon (\hX,\omega)\to B$ be a Lagrangian torus fibration defining
    the same $B$.
  \end{enumerate}
  Then there exists a diffeomorphism $\phi\colon \hX\to X_t$ to a
  nearby fiber $t\ne0$ such that
  \begin{enumerate}
  \item [(a)] $\phi_* \hat e = e$, where $\hat e$ is a fiber of $\mu$ and $e\in
    T_{X_t}$ the isotropic vanishing cycle.
  \item [(b)] $\phi_*[\omega] = \lambda$, where $\lambda\in e^\perp/e$
    is the monodromy invariant.
  \end{enumerate}
\end{theorem}

This theorem reduced study of Kulikov models with a monodromy
invariant $\lambda$ to that of a mirror, symplectic K3 surface $\hX$
and a symplectic form on it. If $\hX$ is algebraic, then we can use a
nef line bundle $L$ instead of the form $\omega$. One deals with the
non-generic $\ias$ by Remark~\ref{rem:assume-generic}. 

\subsection{Visible curves on $\ias$}
\label{sec:visible-curves}

In this paper we use visible curves mostly for motivation, so we only
give a brief sketch. See more details in
\cite{alexeev2019stable-pair}.

Let $B$ be a generic $\ias$. The integral tangent sheaf $T_{B,\bZ}$ is
a constructible sheaf whose fiber is $\bZ^2$ at a smooth point and
$\bZ$ at an $I_1$ singularity. The Leray spectral sequence for
$\mu\colon\hX\to B$ and the sheaf $\bZ_{\hX}$ shows that
$H^1(B,T^*_{B,\bZ})$ has a natural bilinear product and is isomorphic to
$\me^\perp/\me \simeq \lias = {\rm II}_{2,18}$. By a Poincar\'e duality,
$H^1(B,T^*_{B,\bZ})$ can be identified with $H_1(B,T_{B,\bZ})$.
The elements of the latter group are cycles $\sum (\gamma_i,v_i)$
valued in the integral tangent bundle $T_{B,\bZ}$ which satisfy
balancing conditions at the boundaries of the $1$-chains $\gamma_i$ from the cycle.
These are called \emph{visible curves}, see \cite[Con.~2.39]{alexeev2019stable-pair}.

Let $p_1,p_2$ be two $I_1$ singularities connected by a path. Let
$v_1,v_2$ be the monodromy directions at these points and suppose that
there is a path $\gamma$ from $p_1$ to $p_2$ with a constant vector
field $v$ along it which at the ends equals $v_1$ and $v_2$. Then
$(\gamma,v)$ is a visible curve. Its square is $(-2)$, and if there
are three such $I_1$ singularities with the same monodromy rays then
$(\gamma(p_1,p_2), v) \cdot (\gamma(p_2,p_3), v)=1$.
Frequently, for a collection of several $I_1$ singularities with some
monodromy rays in common one can form a collection of visible curves
whose intersection matrix is an $ADE$ matrix, giving the dual $ADE$
graph. We give an example of this in Section~\ref{sec:type3-nongeneric}.

\section{Mirror symmetry for K3 surfaces with a nonsymplectic
  involution}
\label{sec:mirror-with-involution}

The Mirror Theorem~\ref{thm:EF-mirror} establishes a dictionary
between algebraic geometry of degenerations of surfaces in the
$S$-family and symplectic geometry of surfaces in the $\hS$-family,
where $\hS = \oT = e^\perp/e$. One special feature of the present case
is that the lattice $\hS$ is again $2$-elementary, so we can exploit
algebraic geometry of degenerations of surfaces in the
$\hS$-family as well. We do it in this section.

\smallskip

We begin with a hyperbolic lattice $S$ from \fign\ and the period
domain $\bD_S$ as in Section~\ref{sec:moduli-with-involution} parameterizing
K3 surfaces with involution whose generic Picard lattice is $S$ and
generic transcendental lattice is $T$. As in
Section~\ref{sec:0-cusps}, let $e\in T$, $e^2=0$ be a $0$-cusp
of~$\bD_S$ and $\oT=e^\perp/e$ the hyperbolic lattice corresponding to
it.

On the mirror side, we now consider the family of K3 surfaces with
involution whose generic Picard lattice is $\hS = \oT$.  Note that by
Remark~\ref{rem:mirror-targets} not every node of \fign\ can appear as
$\oT$.  If $(\hX,\hiota)$ is one of these surfaces then we get the
quotient surface $\hY = \hX / \hiota$, which is a rational surface for
$\hS\ne (10,10,0)$ and an Enriques surface for $\hS=(10,10,0)$.

\subsection{A special degeneration}
\label{sec:special-degeneration}

In this subsection we restrict ourselves to the elliptic surfaces
of Section~\ref{sec:k3-elliptic-pencil}, which we denote by $\hX$ since
this is the mirror side.

\begin{lemma}\label{lem:special-degeneration}
  In each of the cases (1,2,3) of Theorem~\ref{thm:nikulin-elliptic},
  there exists a one-parameter degeneration $\hX\to (C,0)$ with the
  central fiber $\hX_0 = \hY\cup_D \hY$. The $3$-fold $\hX$ can be chosen to be
  a smooth Kulikov degeneration in the cases (2,3); it is singular in
  the case (1).
  
  In the Enriques case (4) there exists such a degeneration with
  smooth $\hX$ and the central fiber $\hX_0= \hY\cup_D \hY$ such that $\hY$ is
  the surface from the Halphen case (3) but the involution on $\hX_0$ is
  base-point-free.
\end{lemma}
\begin{proof}
  ($I_{2k}I_0$). This Type III degeneration is achieved by a family
  in which the branch divisor $G_0$ on $\hY_0$ collides with the special
  fiber $G$. One has $G_0 + G\sim 2D$. In local coordinates $(x,y,s)$
  on $\hY$ and $s$ on $\bP^1$ the fibration $f'\colon \hY\to\bP^1$ can be
  written as $xy^2=s$, a $(-4)$ branch curve $E_i$ is $x=0$ and the
  fiber $F_0$ is $xy=t$. Then the double cover is locally given by the
  equation $z^2= x(xy-t)$.

  \smallskip
  
  ($I_0^2$). Collide $F$ and $F_0$.  In local coordinates
  $\hY\to\bP^1$ is $(x,y)\to s=x$, the branch divisor is $s^2=t$ and $\hX$
  is $z^2=x^2-t$. Since $F\sim F_0$, the branch divisor $F+F_0$ is a
  pullback of $2$ points on $\bP^1$, so this is just
  $\hY\times_{\bP^1} S$, where $S\to \bP^1\times C$ is the family of
  double covers of $\bP^1$.

  \smallskip
  
  (Halphen). Degenerate the branch locus into the multiple
  fiber $G=2D$. Locally $\hY\to\bP^1$ is
  $(x,y)\to s=x^2$ and the branch divisor is $s=t$. Thus, $\hX$ locally
  is $z^2=x^2-t$. The fiber $G=2D\sim -2K_\hY$ is divisible by $2$ in
  $\Pic \hY$, so the construction works globally.

  \smallskip
  
  (Enriques). We construct this degeneration ``by hand," by smoothing
  a central fiber to an Enriques K3 surface. Let $\hY\to\bP^1$ be an 
  index $2$ Halphen pencil. Consider the surface $$\hX'_0=\hY\cup_D \hY$$
  from the Halphen case with multiple fiber $G=2D$, $D\sim -K_{\hY}$. It is two copies of
  $\hY$ glued {\it by the identity map} along the elliptic curve $D$,
  and the involution $\iota$ exchanges the two copies of $\hY$ and fixes $D$ pointwise.

  Let $E=\Pic^0(D)$. This is an elliptic curve and $D$ is a torsor
  over $E$. Then $\cF:=\cO_{\hY}(D)|_D=a\in E[2]$ is nontrivial
  $2$-torsion because the pencil is Halphen.
  Now build a new surface $\hX_0= \hY\cup_D \hY$ in which the
  two copies of $\hY$ are glued with a twist, a translation of $D$ by 
  the element $a\in E[2]$. There is still an involution which exchanges
  the $\hY$s, but it now acts as translation by $a$ on $D$
  and thus is fixed point free on $\hX_0$.
  
  Since
  $T^*_a\cF\simeq\cF$ and $\cF^{\otimes 2}\simeq\cO_D$, $\hX_0$ is
  $d$-semistable. We will pick a generic smoothing of $\hX_0$
  preserving the $(+1)$-eigenspace of the reduced Picard group of
  $\hX_0$ modulo of the components of $\hX_0$ defined in
  \ref{def:reduced-picard}. 
  The lattice of numerical Cartier divisors is
  \begin{displaymath}
    \wt\Lambda = \ker\big( H^2(\hY) \oplus H^2(\hY) \to H^2(D) \big),
  \end{displaymath}
  $\Xi =\bZ\xi$, where $\xi = (D,-D)$, and
  $\Lambda = \wt\Lambda/\Xi$ is the reduced lattice of numerical
  Cartier divisors.  The surface $\hY$ is a blowup of $9$ points lying
  on a cubic, so $\Pic \hY\to\Pic D$ can be identified with the
  homomorphism
  \begin{displaymath}
    \I_{1,9}\to\bZ, \ L\mapsto L\cdot D, \quad
    \text{where } D = (-3,1,\dotsc, 1) = -3e_0 + \sum_{i=1}^9 e_i.
  \end{displaymath}
  Choose $\hY$ to be the $(10,10,1)$ surface from
  Table~\ref{tab:extremal-surfaces}. If $e_9$ is the last blowup of
  $\hY\to\bP^2$ then in $\la e_9, D\ra^\perp$ in $H^2(\hY)$ is the $E_8$
  lattice with a basis given by $8$ of the~$9$ components of
  the $\II^*$ fiber.

  Denote by $d$ and $v$ the vectors $(D,0)$ and $(e_9,e_9)$.  One has
  $v^2=-2$, $d^2=0$ and $v\cdot d = 1$.  Then $v$ and $d$ span a copy
  of $U$, and
  \begin{displaymath}
    \Lambda^+ = \diag E_8 \oplus \la v, d\ra = E_8(2) \oplus U.
  \end{displaymath}

  Let $\psi\colon\Lambda\to E$ be the period of $\hX_0$ as in
  Eq.~\eqref{eq:period}, and $\psi^+$ be its restriction to
  $\Lambda^+$.  We have $\oPic \,\hX_0 = \ker\psi$ and
  $(\oPic\, \hX_0)^+ = \ker \psi^+$. For our surface, $\psi(E_8(2))=1$
  (since the $\II^*$ fiber is disjoint from $D$), $\psi(v) = a$ and
  $\psi(d) = \cF$. Thus, when $a=\cF\in E[2]$, we have
  \begin{displaymath}
    \ker( U\to E) = \la v+d, 2d\ra \simeq U(2), 
    \quad (\oPic\, \hX_0)^+ \simeq E_8(2) \oplus U(2) = (10,10,0).
  \end{displaymath}
  So a generic smoothing of $\hX_0$ preserving $(\oPic \,\hX_0)^+$ is a
  family of K3 surfaces with generic Picard lattice $(10,10,0)$, i.e. a
  family of Enriques K3 surfaces. It comes with Enriques involution
  reducing to the involution on $\hX_0$. \end{proof}

\begin{remark}
  The last construction builds a degeneration of
  Enriques surfaces to a nonnormal surface obtained by gluing an index~$2$
  Halphen pencil to itself along the double fiber $2D$ by translating
  $D$ by $a = \cO_D(D)\in(\Pic^0D)[2]$. This is an interesting degeneration
  which we have not seen in the literature. 
\end{remark}

\begin{remark}
  There are two more ways to produce a $d$-semistable Type II surface
  $\hX_0=V\cup_D V$ with a base-point-free involution: We can glue two
  Halphen pencils by a nontrivial $2$-torsion $a\ne\cO_D(D)$. We can also glue
  by $2$-torsion two copies of a rational elliptic surface with a section. In the first
  case $\ker(U\to E) = 2U\simeq U(4)$. In the second case
  $\ker(U\to E) = \la 2v, d\ra$. Both of these are not $2$-elementary
  lattices, so the base-point-free involution on $\hX_0$ can not be
  extended to a smoothing.
\end{remark}

\subsection{Lagrangian torus fibration for the mirror K3 surfaces}
\label{sec:lagr-fibration}

\begin{theorem}\label{thm:special-fibration}
  Let $\hS=(r,a,\delta)=(10+k-(g-1),10-k-(g-1),\delta)$ be a lattice appearing
  as a target of one of the mirror moves $S\leadsto\oT=\hS$ of
  Definition~\ref{def:mirror-moves}, $\hX_t$ a surface with
  $\hS = (\Pic \hX_t)^+$, and $L\in\hS\otimes\bQ$ an ample $\bQ$-line
  bundle on $\hX_t$. Then there exists an involution-equivariant
  Lagrangian torus fibration $\mu\colon (\hX_t,\omega)\to B$ with
  $\omega=[L]$, where $B=P\cup P\uopp$ is a
  union of two Symington polytopes for the pairs $(\hY,D)$ of Lemma 
  \ref{lem:special-degeneration}, glued along their common
  boundary $\partial P = \partial P\uopp$, to form an equator
  of the sphere.

  If $\hS\ne (10,10,0)$ then $\partial P$ and $\partial P\uopp$ are glued by the identity map,
  and if $\hS= (10,10,0)$, then $B=P\cup_{\rm twist} P\uopp$ results from gluing
  the Symington polytope by a half-twist along its equator, so that $B/\iota\simeq \bR\bP^2$.

  If $\hS\ne (10,8,0)$ then the fibration can be chosen so that $B$ has
  $2k$ $I_1$ singularities on the equator at the vertices of $P$, with
  the monodromy rays transverse to the equator, an
  $I_{2(g-1)}$ singularity with the monodromy ray parallel to the equator,
  and $12-k-(g-1)$ total
  $I_1$ singularities in each of the hemispheres.
  For $\hS=(10,8,0)$, there are $12$ $I_1$ singularities in each of the
  hemispheres.
  
  In particular, there no singularities on the equator when $\hS = (10,10,0)$, $(10,10,1)$, 
  $(10,8,0)$, and the equator is an embedded integral-affine circle.
\end{theorem}
\begin{proof}
  \emph{The base case.}
  We begin with the base case when $\hS$ is a lattice in
  Section~\ref{sec:k3-elliptic-pencil}. 
  The map $\mu$ is given by Construction~\ref{def:ias-k3-symplectic}
  for the special degeneration of
  Section~\ref{sec:special-degeneration}. 
  In the $\I_0^2$, Halphen, and Enriques cases the degeneration of
  Lemma~\ref{lem:special-degeneration} is already Kulikov and
  immediately gives the required Lagrangian torus fibration~$\mu$.
  It is, notably, a Type II degeneration.

  We now consider the $\I_{2k}\I_0$ case. In local coordinates the
  family $\hX\to (C,0)$ is
\begin{displaymath}
  \hX = \{ z^2 = x(xy^2 - t) \} \subset \bA^4_{x,y,z,t}
  \quad \text{or} \quad (z+xy)(z-xy) = tx.
\end{displaymath}
This $3$-fold is singular along the line $\ell = \{x=z=0\}$. The
central fiber $\hX_0$ is a union of two $\bA^2$s glued along $xy=0$. 
Let $\wX$ be the blowup of $\hX$ along $\ell$. It is covered by two charts:
\begin{enumerate}
\item[Chart $1$] $z=xz_1$, $(z_1+y)(z_1-y)x=t$. This is a smooth
  $3$-fold and the central fiber is a normal crossing divisor with a
  single triple point.
\item[Chart $2$] $x=zx_1$, $z(1+x_1y)(1-x_1y) = tx_1$. If $g$ is the
  difference of the two sides, then
  \begin{math}
    \frac{\partial g}{\partial t} = -x_1, \,\,
    \frac{\partial g}{\partial z} = (1+x_1y)(1-x_1y)
  \end{math}
  which have no common zeros. Thus, this $3$-fold is smooth as
  well. The central fiber consists of three irreducible components,
  two of which do not intersect: $1+x_1y=0$ and $1-x_1y=0$.
\end{enumerate}

We conclude that $\wX$ is a Kulikov model of $\hX\to C$. Globally, the
blowup $\wX\to \hX$ has $n$ exceptional divisors $Z_i$, one for each
$(-4)$ curve $E_i$ in the special fiber $F$ of~$\hY$. Each $Z_i$ is an
anticanonical pair with two boundary divisors $D_1$, $D_2$, $D_1\cdot D_2=2$
and $D_i^2=2$. The surface $Z_i$ is geometrically ruled and $D_1,D_2$
are sections. It follows that $Z_i\simeq \bP^1\times\bP^1$ and
$D_1,D_2\in |\cO(1,1)|$, or $\bF_2$ and $D_1\sim D_2\sim
s_\infty$. The latter is in the same deformation type, so for the
construction of the Lagrangian fibration $\mu$ we can assume
$Z_i=\bP^1\times\bP^1$. The central fiber $\wX_0$ of this Kulikov
model is pictured in Fig.~\ref{fig-kulikov1}.

\begin{figure}[htp]
  \includegraphics[width=3.8in]{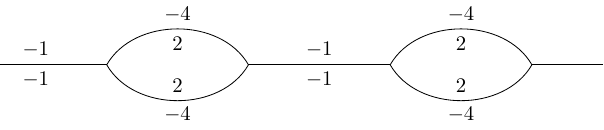}
  \caption{The local behavior along the equator
  of the Kulikov model for a degeneration of $(10+k,10-k,\delta)$.}
  \label{fig-kulikov1}
\end{figure}

The central fiber $\wX_0$ has $2+k$ irreducible components $V_i$. The two
hemispherical components are isomorphic to $\hY$ with the Symington polytope $P$ and $k$
are isomorphic to $(\bP^1\times\bP^1,D_1+D_2)$.  The Clemens collapse $\hX_t\to \wX_0$
composed with the moment maps $V_i\to P_i$ for appropriately chosen
polarizations on $V_i$ gives a Lagrangian torus
fibration $\wt\mu\colon \hX\to \wB = \cup P_i$. We construct the
claimed Lagrangian torus fibration $\mu\colon \hX\to B$ from it by nodal
slides defined in Section~\ref{sec:surgeries-ias}.

\begin{figure}[htp]
  \includegraphics{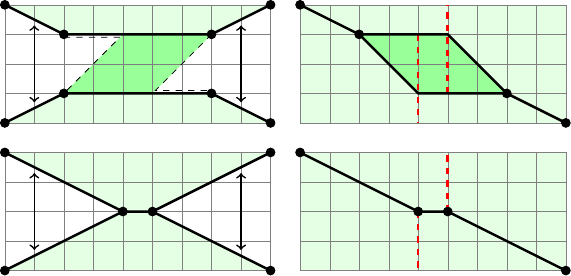}
  \caption[Surgery to make $\ias$ into $P\cup P\uopp$]
  {Surgery to make $\ias$ into $P\cup P\uopp$:
  $(\ell'_{i-1},\ell'_i,\ell'_{i+1})=(1,5,1)\to
    (\ell_{i-1},\ell_i,\ell_{i+1})=(2,1,2)$}
  \label{fig-surgery}
\end{figure}

As in Prop.~\ref{prop:small-nef-cone}, let $L=\pi^*L_\hY$ for some
ample $\bQ$-line bundle on $\hY$. We modify $L_\hY$ by considering an
ample $\bQ$-line bundle $L'_\hY= L_\hY-\sum_{i=1}^k c_iE_i$ for some small
positive rational numbers $0<c_i\ll1$. Here, $E_i$ are the $k$
$(-4)$-curves.  Below, we consider the case of a single $(-4)$-curve
$E_i$, with obvious modifications for the general case. Let
$\ell_i=L_\hY \cdot E_i$ and $\ell'_i = L'_\hY \cdot E_i$. Then $\ell_i = \ell'_i-4c_i$
and $\ell_{i\pm1} = \ell'_{i\pm1}+c_i$.

We consider the polarization $\cO(\ell_i,c_i)$ on
$Z_i=\bP^1\times\bP^1$. The Symington polytope is obtained from a
parallelogram with sides of lattice length $\ell_i'$ and $c_i$ by
cutting two triangles of side $c_i$ erasing two of the four sides, as in
the first picture of Fig.~\ref{fig-surgery}.
This produces an $\ias$
as in the second picture
(going clockwise)
with two $I_1$ singularities off the
equator. Node sliding each of them by lattice distance $c_i$ gives the
$\ias$ $B=P\cup P\uopp$ with sides $\ell_i$ and $2k$ $I_1$
singularities on the equator. The effect of this nodal
slide on the lengths is $(\ell'_{i-1},\ell'_i,\ell'_{i+1}) \to
(\ell'_{i-1}+c_i,\ell'_i-4c_i,\ell'_{i+1}+c_i)$, and gives the symplectic form
corresponding to $L_\hY$.

\emph{General case.} We use the reduction given in
Corollary~\ref{cor:reach-by-moves} and
Lemma~\ref{lem:heegner-surfaces}.  Let $\hX$, $\hX_1$ be K3 surfaces
as in Lemma~\ref{lem:heegner-surfaces}.  Vary the $\bQ$-line bundle
$L$ on $\hX$ continuously until it becomes zero on the $(-2)$-curve
$E$ and gives a contraction $\hX\to\hX'$ to a K3 surface with a node.
This produces a continuous family of integral affine structures. In
the limiting integral-affine structure on $B$,
two $I_1$ singularities with monodromy rays parallel to the
equatorial edge collide to give an $I_2$ singularity with the
monodromy ray in the same direction. Smoothing $\hX'$ to a surface
$\hX_1$ in the $\hS_1$ family gives a Lagrangian torus fibration
$\mu\colon \hX_1\to B$ with same $B$. We repeat this procedure $(g-1)$
times to complete the proof.

One can produce either $2(g-1)$ $I_1$
singularities in this way, or a single $I_{2(g-1)}$ singularity, they
are equivalent up to nodal slides along the equator. 
\end{proof}

\begin{example}[An $\ias$ for $\hS=(12,6,1)$]
  Let $(\overline{\oV}, \overline{\oD}) =  (\bP^2,L_1+L_2+L_3)$ be a
toric anticanonical pair with the $L_i$ lines forming a triangle.
Let $(\oV,\oD)\to (\overline{\oV}, \overline{\oD})$ be the blow up at the three corners
of the triangle, introducing three exceptional classes $e_1$, $e_2$, $e_3$ to get
a Cremona surface. Then,
let $(V',D')\to (\oV,\oD)$ be the blow up of $9$ total points, $3$ on each strict transform
of $L_i$, with exceptional classes
$e_4, \,\dots,\, e_{12}$.

The pair $(V',D')$ arises as an $I_6I_0$ type quotient of an
elliptic K3 surface, as in Theorem \ref{thm:nikulin-elliptic},
with the length $6$ cycle $D'$ of alternating $(-1)$- and $(-4)$-curves representing
the quotient of the $I_6$ cycle. Hence $(V',D')$ is the quotient of an elliptic
K3 surface with involution, with generic Picard group $(13,7,1)$.  

\begin{figure}
\includegraphics[width=5in]{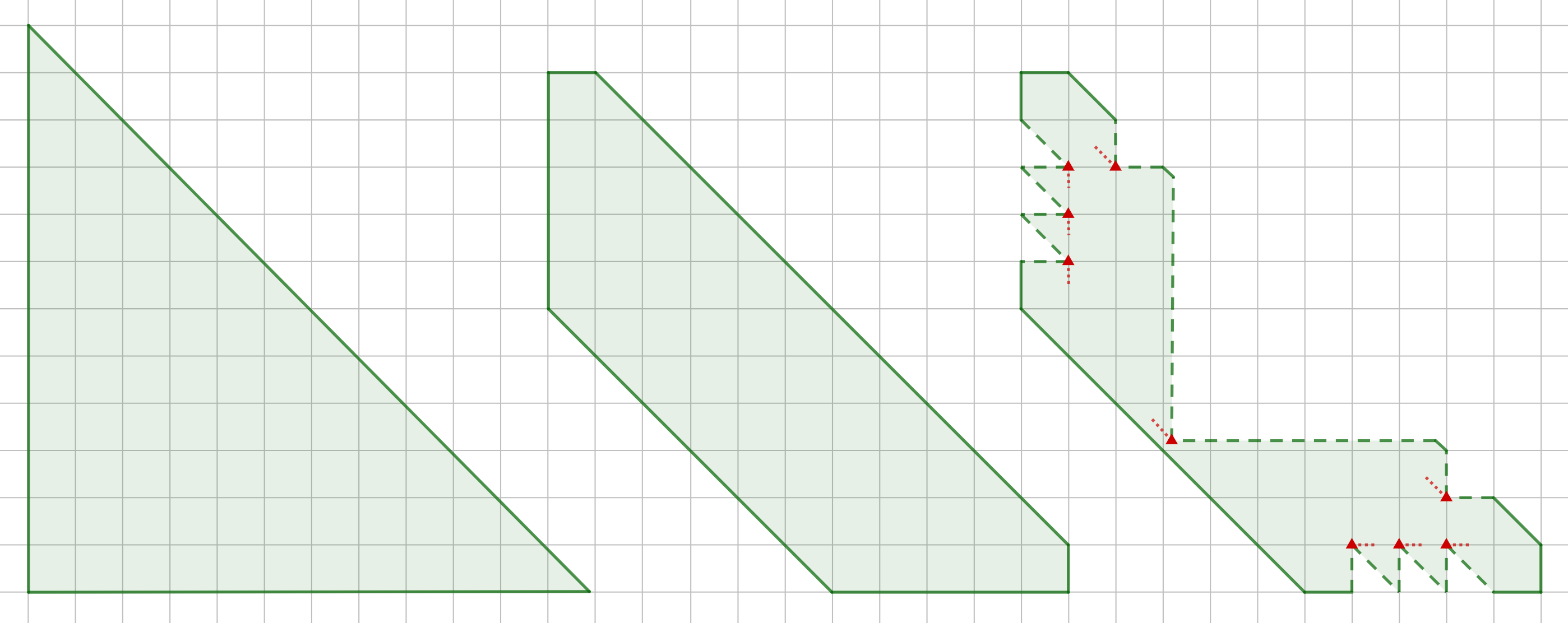}
\caption[Moment polytopes for $\hS=(12,6,1)$]%
{Left: the moment polytope for $(\mathbb{P}^2, 12h)$. Center:
the moment polytope for $(Bl_{p_1,p_2,p_3}\bP^2, 12h - 6e_1-e_2-e_3)$. Right:
the Symington polytope for $(Bl_{p_1,\cdots,p_{12}}\bP^2, L_{\epsilon})$.}
\label{example}
\vspace{5pt}
\includegraphics[width=3in]{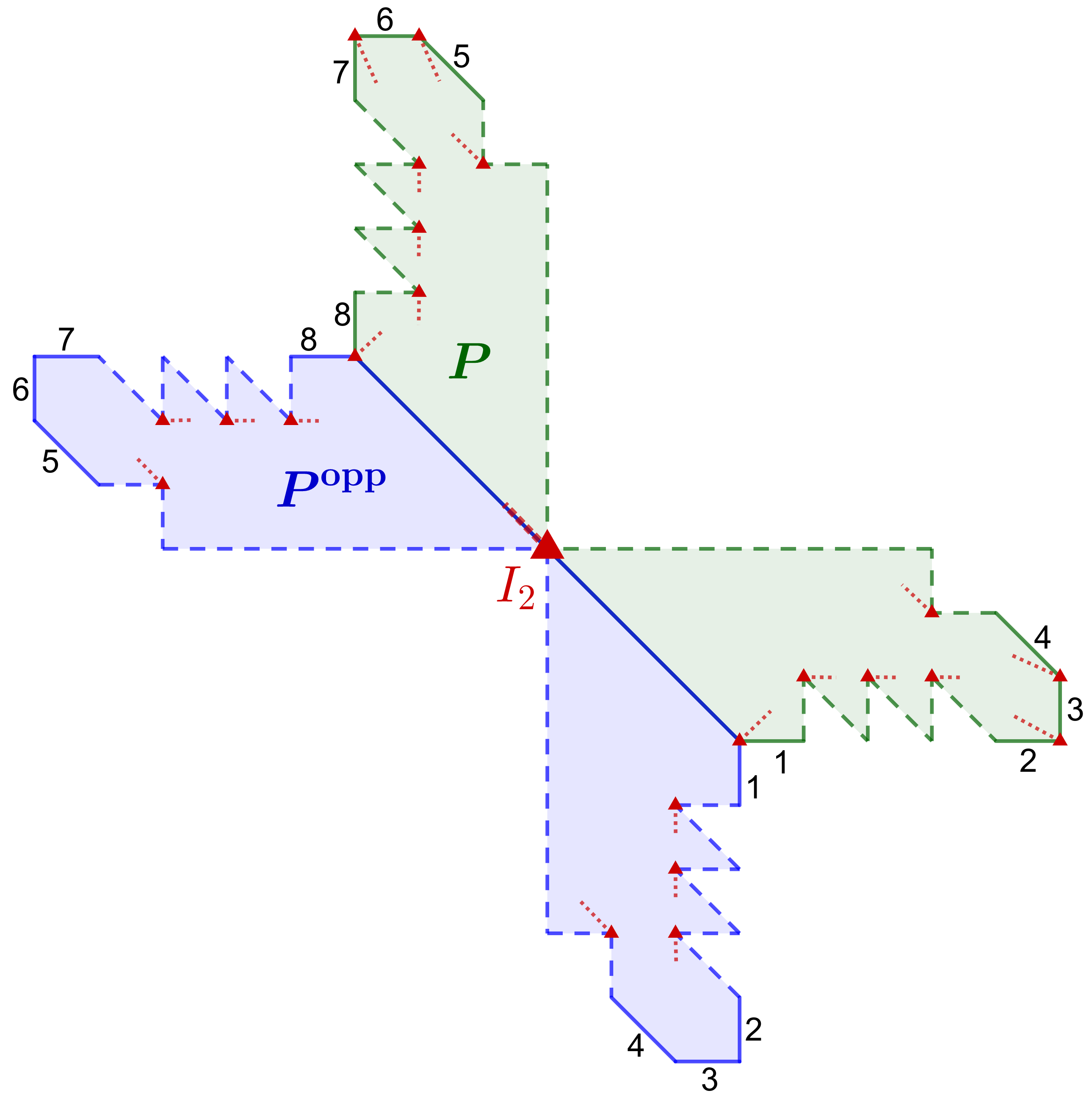}
\caption[An integral-affine sphere associated to an ample class
  in $\hS = (12, 6,1)$]%
{An integral-affine sphere associated to an ample class
  in $\hS = (12, 6,1).$
  There are $22$ $I_1$ singularities and $1$ $I_2$ singularity.}
\label{example2}
\end{figure}

Consider the ample class $12h$ on $\overline{\oV}=\bP^2$.
After the three corner blow-ups, we have an ample class $12h-6e_1-e_2-e_3$ on
$\oV$. Their moment polytopes are depicted in the left and center of Figure \ref{example}.
After the next nine interior blow-ups, we have an ample class $$L_\epsilon := 
12h-6e_1-(6-\epsilon)e_8-\sum_{i\geq 2,\, i\neq 6} e_i$$ on $(V',D')$ whose corresponding
Symington polytope $P$ is shown on the right of Figure \ref{example}.
As $\epsilon\to 0$, this class becomes only big and nef, because $L_\epsilon \cdot (h-e_1-e_8)=\epsilon$
and $h-e_1-e_8$ represents the class of an exceptional curve.
Thus, $L_0$ represents an ample class on the pair $(V,D)\leftarrow (V',D')$
resulting from the blow-down of $h-e_1-e_8$.
Set $(\hY,D) = (V,D)$. Then $\hY$ is the quotient of a K3 surface
$\hX$ with generic Picard group $(12,6,1)$, or alternatively $k=3$, $g=2$.

We glue two copies $B= P\cup P\uopp$ together along their common
boundary to form a sphere, and let $\epsilon \to 0$. At $\epsilon=0$, 
the two $I_1$ singularities in $P$ and $P\uopp$ collide to form an $I_2$ singularity
on the equator. See Figure \ref{example2}.

As in Lemma \ref{lem:special-degeneration},
we have a degeneration $(\hX',L)\to (C,0)$ of K3 surface with involution,
where $\hX_0' = V'\cup_{D'}V'$ and the generic Picard lattice is $\hS' = (13,7,1)$,
and with $L\big{|}_{\hX_0'}=(L_\epsilon,L_\epsilon)$. At $\epsilon=0$, the class $L$
is only big and nef on the general fiber, contracting the $(-2)$-curve which
results from smoothing the union of $h-e_1-e_8$ and its opposite. But
instead smoothing $\hX_0'$ into the larger moduli space $(12,6,1)$, we get a 
degeneration $\hX$ for which the class $L$ is ample on the general fiber.
By Theorem \ref{thm:special-fibration}, we have a Lagrangian torus fibration
$\hX_t\to B$, and as predicted, there $2k=6$ $I_1$ singularities on the equator
with monodromies transverse to the equator, and there is one $I_{2(g-1)}=I_2$
singularity on the equator, with monodromy parallel to it.
As in Corollary \ref{cor:reach-by-moves}, the $(12,6,1)$ lattice was reached
by a single $(-1,-1)$ Heegner move $(13,7,1)\rightarrowtail (12,6,1)$, corresponding a white vertex
of the $(13,7,1)$ Coxeter diagram attached to the outer $6$-cycle by a bolded edge
(see Fig.~\ref{fig:coxeter1}).
 \end{example}

\subsection{Mirror symmetry and involutions}
\label{sec:mirror-involutions}

Let $\mu\colon \hX\to B=P\cup P\uopp$ be the Lagrangian torus fibration
of Theorem~\ref{thm:special-fibration}, associated to some $2$-elementary
lattice $\hS$. Suppose that the same 
$B=\Gamma(X_0)$ is the $\ias$ obtained as in
Construction~\ref{con:ias-from-kulikov}, from the dual complex
of a Kulikov surface $X_0$. We may smooth $X_0$ to a Kulikov model $X\to (C,0)$.
By the results of the following section, whenever $\hS = \oT = e^\perp/e$
is the target of a mirror move from $S$, we may choosing the gluings
of $X_0$ appropriately so that $X\to (C,0)$ admits an involution
acting on $\Gamma(X_0)=B$ by the specified one,
exchanging $P$ and $P\uopp$, with generic Picard group $S$.

By the Mirror Theorem~\ref{thm:EF-mirror} there is a diffeomorphism
$\phi\colon\hX \to X_t$ inducing an identification between a nef 
class $L\in \me^\perp/\me$ on $\hX$ and the monodromy vector $\lambda\in\oT = e^\perp/e$.  Now on
$\hX$ there are two involutions: $\miota$ and
$\iota_\phi:=\phi\inv \circ \iota \circ \phi$.

\begin{theorem}\label{thm:mirror-and-involution}
  The two involutions $\miota$ and $\iota_\phi$
  differ by negation (with respect to some section)
  $z\to -z$ in the
  generic fiber $\me\simeq (S^1)^2$ of $\mu$, composed with a
  translation in the fibers. The diffeomorphism $\phi$ identifies
  $(e^\perp/e)^+$ with $(\me^\perp/\me)^-$ and $(e^\perp/e)^-$ with
  $(\me^\perp/\me)^+$.
\end{theorem}
\begin{proof}
The Mirror Theorem \cite[Prop.~3.14]{engel2021smoothings} proceeds
by identifying $\hX$ and $X_t$ as isomorphic torus bundles over the same base $B$
(this is essentially topological SYZ mirror symmetry). The torus bundle on $\hX$
is the Lagrangian torus fibration $\mu$, and Proposition \ref{def:ias-k3-symplectic}
serves to construct the Lagrangian torus fibration on $X_t$ as the composition
of a Clemens collapse, and almost toric fibrations of components $V_i\subset X_0$
(note that our fibration $\mu$ was also constructed this way, but {\it from a different
Kulikov model $\hX_0 = \hY\cup_D \hY$}). Since the smoothing
$X\to (C,0)$ extends an involution on $X_0$, we can make the Clemens collapse
and almost toric fibrations involution equivariant, and so can ensure both $\iota$ and
the diffeomorphism $\phi$ respect the torus bundle structure over $B$. 

  Thus, $\miota$ and $\iota_\phi$ respect the torus bundle structure,
  and define the same
  involution on the base $B$, exchanging $P$ and $P\uopp$. So the composition
  $\iota_\phi \circ\miota\inv$ is a fiber-preserving diffeomorphism of the
  torus bundle, and defines an
  element of the mapping class group $\GL(2,\bZ)$
  of the general fiber.
  Since not all of the monodromy rays of the singularities of $B$ are
  parallel, the only element centralizing the monodromy
  on $S^2\setminus\{p_i\}$ is $\pm\id$.
  
  So $\iota_\phi\circ\miota$ must give $\pm\id$ in the mapping class group
  of the general fiber.
  By the Mirror Theorem, $\phi$ exchanges the class $L\in S$ of an ample line
  bundle in the $(+1)$-eigenspace of $\iota$ with
  the monodromy invariant $\lambda\in\hS=\oT$, which is in the $(-1)$-eigenspace of
  $\miota$. Thus, $\iota_\phi\circ\miota=-\id$ composed with a
  translation. The negation $z\to -z$ in the torus fibers acts by
  multiplication by $(-1)$ on $\me^\perp/\me$, as can be
  seen e.g.~from the Leray spectral sequence 
  $H^p(B, R^q\mu_*\underline{\bZ}) \implies H^{p+q}(\hX,
  \underline{\bZ})$. The eigenspaces of $\miota$ and $-\miota$ are
  opposite of each other. A translation in the fibers acts on
  $\me^\perp/\me$ as identity. The theorem follows. 
\end{proof}

\section{Kulikov models of K3 surfaces with a nonsymplectic involution}
\label{sec:kulikov-models-all}

In this section for each of the $75$ lattices of \fign\ we construct
a family of models $(X,R)$ adapted to the
ramification divisor of the involution. Moreover, for the $50$
cases when $R$ contains a curve $C_g$ with $g\ge2$, we construct the
divisor models $(X,C_g)$. We do it first for a generic $\lambda$ in the
interior of the fundamental domain~$\ch_2$, and then in
Sections~\ref{sec:type3-nongeneric}, \ref{sec:type2-models} for 
$\lambda$ on a face of $\ch_2$. 

\subsection{The main construction}
\label{sec:main-construction}
Let $S$ be a lattice from \fign, $e\in T$ a $0$-cusp, and
$\oT=e^\perp/e$. Let $\lambda\in\oT\cap\cC$ be a monodromy
invariant. Recall that $\cC$ is the positive cone.
Our goal now is to explicitly construct a Type III
$d$-semistable Kulikov surface $X_0$ with the monodromy
invariant~$\lambda$, admitting an involution $\iota_0$
which extends to an involution on a Kulikov model
$X\to (C,0)$ for which $X_t\in F_S$.

\begin{proposition}\label{stratum-function} For $g\ge 2$, $S\neq (10,8,0)$, and for all $0$-cusps of $F_S$,
there is a well-defined {\rm stratum function}
$$\bS\colon \{\lambda\colon \oT\cap \cC\}\to \{\textrm{combinatorial types of KSBA-stable surfaces}\}$$
which assigns to each monodromy invariant $\lambda$ of a Type III degeneration,
the combinatorial type of the stable limit of a degeneration $(\oX,\epsilon \oC_g)$, 
whose monodromy is $\lambda$.
Furthermore, the loci on which $\bS$ is locally constant form the cones of a semifan $\mathfrak{F}$
for which $(\oF_S)^\nu = \oF_S^{\mathfrak{F}}$.

Dropping the condition on $g$, we have that any degeneration
with monodromy invariant $\lambda$ admits a Kulikov model of fixed
combinatorial type, adapted to $R$. 
\end{proposition}

Recall, the Kulikov model is adapted to $R$ if the flat limit contains no strata, and any
components of positive genus have a nef limit.

\begin{proof}
It is already a consequence of the general theory, see
\cite[Thm.~1]{alexeev2021compact}, \cite[Thm.~3.24]{alexeev2021nonsymplectic},
that when $g\ge 2$, the KSBA compactification $\oF_S$ 
is normalized by a semitoroidal compactification 
associated to some semifan $\mathfrak{F}$. This is because the
fixed locus of an involution is a so-called ``recognizable divisor." 

The main tool
to find the associated semifan is \cite[Thm.~9.3]{alexeev2021compact}.
It states that the cones $\sigma$ of the semifan $\mathfrak{F}$ 
are determined as the (closures of) collections of monodromy
invariants $\lambda\in \sigma$ for which the KSBA-stable model
of a degeneration with monodromy invariant $\lambda$
has a fixed combinatorial type, see \cite[Def.~8.12]{alexeev2021compact}.
Furthermore, recognizability ensures
that {\it any degeneration} with a given monodromy invariant in the projective
class of $\lambda$ has the same combinatorial type \cite[Cor.~8.13]{alexeev2021compact}.
The first statement follows.

Even when $R$ contains no components of genus $g\ge 2$, many of the same
properties of recognizability hold, in particular, the existence of a combinatorially
constant Kulikov model adapted to $R$ for all divisors with a fixed monodromy
invariant $\lambda$. Specifically, the second statement
follows from \cite[Prop.~8.16]{alexeev2021compact}
and the arguments in \cite[Thm.~3.24]{alexeev2021nonsymplectic}.
\end{proof}

With these results in mind, we need only construct some generic divisor model
with monodromy invariant $m\lambda$, for all $\lambda\in \oT\cap \cC$,
and identify when the combinatorial type of the stable model
changes (as a function of $\lambda$). 

It is sufficient to consider
$\lambda$ up to the isometry group $O(\oT)$. Thus, we can take
$\lambda$ in a fundamental chamber $\ch_2$ for the reflection group
$W_2$ or, more economically, in a fundamental chamber $\chref$
for the full reflection group $\wref$.
We only care about a generic $X_0$ within its combinatorial type.
This is quite useful, as even a smooth, non-generic 
K3 surface $X_0\in \Delta\subset \bD_S$ 
does not smooth {\it with its involution} into $F_S$.

\medskip

Now we set $\hS=\oT$ and consider the family of K3 surfaces with the
generic Picard lattice $\hS$. Let $\hX$ be one
such surface with an involution $\hiota$. We have $(\Pic \hX)^+=\hS$. 
By Theorem~\ref{thm:special-fibration}, for an ample line bundle $L$ on
$\hX$ there exists a Lagrangian torus fibration
$\mu\colon(\hX,\omega)\to B$ with $B$ an $\ias$ of a special type:
$B=P\cup P\uopp$, a union of two Symington polytopes interchanged by
the involution $\hat{\iota}$, and $[\omega]=L$. 

By Lemma~\ref{lem:nef-Y-trivial} and
Proposition~\ref{prop:small-nef-cone}, assuming $S\ne(10,10,0)$,
$\chref$ can be identified with $(\Nef \hX)\cap \hS_\bR = \Nef\hY$ for
a particular quite degenerate rational surface $\hY$, so that
$L=\pi^*L_\hY$.  Replacing $L$ by $2L$ we can assume $L_\hY$ to be
integral.
Alternatively, we can work with a generic K3 surface in the $\hS$-family and
identify $\ch_2\simeq (\Nef \hX)\cap \hS_\bR = \Nef\hY$. Working with
$\ch_2$ or $\cref$ is purely
a matter a convenience, except for $S=(17,3,1)$ where we only computed
$\ch_2$ and not $\chref$ in Section~\ref{sec:vinberg-exceptional}.  

\medskip

Next, pick a triangulation $\cT$ of $B$ into unit lattice volume triangles, in
such a way that: the vertices are lattice points, $\cT$ is
involution-invariant, and the edges contain the equator. For instance, such
a triangulation arises from triangulating $P$ entirely.
Now interpret $(B,\cT)$ as an $\ias$ coming from a Type III surface
$X_0$ via Construction~\ref{con:ias-from-kulikov}. That is, $B=\Gamma(X_0)$
and for each vertex of $\cT$, we have an anticanonical pair $(V_i,D_i)$
whose pseudofan $\fF(V_i,D_i)$ models the star of the vertex. Then,
we glue $X_0=\cup V_i$ along $D_i$ according to the triangulation. 

For a fixed such $B$ with triangulation, the family of $d$-semistable Type III
surfaces $X_0$ satisfying $\Gamma(X_0)=B$
is $(\bC^*)^{19} = \Hom(\Lambda,\bC^*)$. Here, $\Lambda$ can
be computed directly from $X_0$ by Eq.~\eqref{eq:picard-complex}.
We have $\Lambda=\lambda^\perp$ in $e^\perp/e \simeq \lias =
{\rm II}_{2,18}$, where $e$ corresponds to the $0$-cusp
into which $0\in C$ maps via the extension
of the period map, and $\lambda$ is the monodromy invariant.

\begin{example} Taking $B= P\cup P\uopp$ from Figure \ref{example2},
we may triangulate it in an involution-invariant manner,
and realize it as the dual complex of a Type III Kulikov surface. See Figure \ref{example3}.
Then $22$ components $(V_i,D_i)$ have charge $Q(V_i,D_i)=1$, one component
has $Q=2$, and the remaining components, with $Q=0$, are toric.

\begin{figure}
\includegraphics[width=5 in]{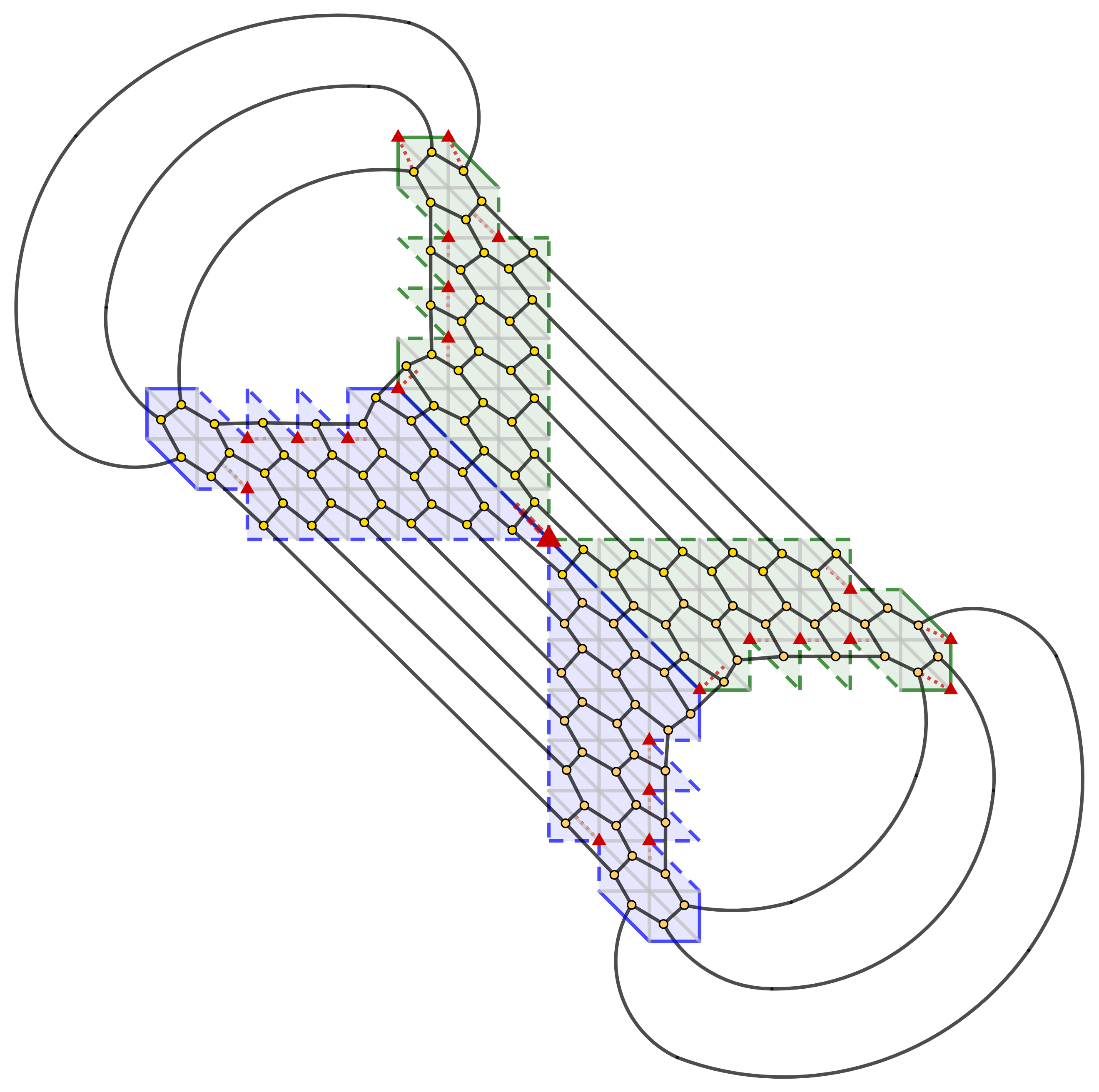}
\caption[A Kulikov surface $X_0$ for $\hS=(12,6,1)$]%
{A Kulikov surface $X_0$ for $\hS=(12,6,1)$ with $\Gamma(X_0)=B$, with triple
points shown in yellow and double curves in black.}
\label{example3}
\end{figure}
\end{example}

\begin{theorem}\label{smoothable} Let $S\leadsto \oT \simeq \hS$ be a mirror move
let $L\in \hS$ be an ample line bundle on $\hX\in\hS$. Let $\lambda = L$
under the identification $\oT\simeq \hS$. Let $P$ be the Symington polytope
of $\hY = \hX/\hat{\iota}$ associated to the image $L_\hY$ 
and let $B= P\cup P\uopp$ be the integral-affine sphere built from
gluing two copies of $P$ (see Section \ref{sec:gluingPs} for further
details on these gluings). Let $\cT$ be an involution-invariant
triangulation of $B$ containing the equator.

Then $B=\Gamma(X_0)$
for a Type III Kulikov surface $X_0$ admitting an involution $\iota_0$
inducing the involution on $B$.
Furthermore, by choosing the period of $X_0$
appropriately, there is a smoothing $X\to (C,0)$ with involution $\iota$
whose general fiber
$X_t$ lies in $F_S$ and whose monodromy invariant is $\lambda\in \oT$.
  \end{theorem}
  
  It is worth remarking that the combinatorial type of $X_0$
  is completely determined by the vector $\lambda\in \oT$
  and has nothing to do with the source $S$ of the mirror move.
  This source $S$ is determined by choosing the period $\psi_{X_0}$ appropriately.
  
  \begin{proof} We observe that the gluings of the
  standard surface $X_0$ (satisfying $\psi_{X_0}=1$)
  are unique, hence invariant under the involution on $B$.
  So $X_0$ admits an involution $\iota_0$ acting on $B$
  by the involution $\iota_{\rm IAS}$ switching $P$ and $P\uopp$.
  
  We analyze
  the deformations $X_0\rightarrowtail X_0'$ for which $X_0'$
  retains an involution $\iota_0'$. For there to be an involution on $X_0'$ 
  the period point $\psi_{X_0'}\in  \Hom(\Lambda,\bC^*)$
  must be {\it anti-invariant} with respect to $\iota_0^*$, under identification
  $\Lambda(X_0')\simeq \Lambda(X_0)$.
This is because the target $\bC^*$ depends on a choice of orientation
on $B$, which is reversed by $\iota_{\rm IAS}$, see \cite[Sec.~6G]{alexeev2019stable-pair}.
We conclude that a generic involutive $X_0'$ deforming $X_0$
is one for which ${\rm ker}(\psi_{X_0'})\supset \Lambda^+$
where $\Lambda^+$ is the $(+1)$-eigenspace of $\iota_0^*$ acting 
on $\Lambda(X_0)$. The period torus for these involutive Kulikov surfaces
is ${\rm Hom}(\Lambda/\Lambda^+,\bC^*)$. These deformations of the standard
surface $X_0$ will become the Kulikov surfaces corresponding to a simple mirror move $S\to \hS$.

The involution $\iota_{\rm IAS}$ on $B$ has an induced action $\iota_{\rm IAS}^*$
on $H^1(B,T_\bZ^*)\simeq {\rm II}_{2,18}=\lias$ which is easily
seen to be an isometry with respect to the intersection form
\cite[Ex.~7.27]{alexeev2020compactifications-moduli}
on visible curves, see Section \ref{sec:visible-curves}.
The induced involution on $\Lambda = \lambda^\perp$ agrees
with the action of $\iota_0^*$ on $\Lambda(X_0)$ for the standard surface
$X_0$. Here $\lambda$, too, can be interpreted purely in integral-affine terms
as the so-called {\it radiance obstruction}
in $H^1(B,T_\bZ^*)$ of the integral-affine structure. So it
is automatically preserved by $\iota_{\rm IAS}$.
We define $\rhoias:=-\iota_{\rm IAS}^*\in O(\lias)$. The motivation
for this definition is Theorem \ref{thm:mirror-and-involution}
(which notably relies
on this theorem).

Let $\lias^{\pm}$ denote the eigenspaces of $\rhoias$.
We claim that $\lias^-=\hS$ or equivalently $(\iota_{\rm IAS}^*)^+=\hS$.
This can be verified purely on the symplectic 
side, by observing that we have an isomorphism $\me^\perp/\me\cong \lias$
where $\me$ is the Lagrangian fiber class of $\mu\colon \hX\to B$. 
The involution $\hat{\iota}$ quotients this fibration of $\hX$ to 
the Lagrangian torus fibration $\hY\to P$. Thus, $\lias^-\otimes \bQ\simeq H^2(Y,\bQ)$
and so $\hS \subset \me^\perp/\me$ must equal $\lias^-$.

Denote $\lias^+:=S^{\rm sat}$ (this choice of notation will be made reasonable soon).
It is a $2$-elementary lattice in $\lias \simeq {\rm II}_{2,18}$ perpendicular to $\hS$
and so $S^{\rm sat}$ is the source of a simple mirror move with
target $\hS$. Smoothing $X_0'$ (a deformation
of the standard surface $X_0$ as above), keeping the classes in $S^{\rm sat}$ Cartier,
we produce the claimed degeneration of the theorem, in the case of a simple mirror move.
The smoothings that keep $S^{\rm sat}$ Cartier (at least for $X_0'$ period-generic) are identified with
the deformations extending the involution $\iota_0$ by \cite[Thm.~6.35]{alexeev2019stable-pair}.

For the non-simple mirrors, we observe that an involutive $X_0$ need not
be a deformation of the standard surface: There may be other connected
components in ${\rm Hom}(\Lambda,\bC^*)$ of the space of anti-invariant periods.
These connected components also admit an involution, because
\begin{enumerate} 
\item $\rho_{\rm IAS}$
acts on the (re)gluing complex of \cite[Sec.~7A]{alexeev2021compact} and an involution
anti-invariant regluing of an involutive $X_0$ will be involutive, and 
\item the period
and gluing complexes are quasi-isomorphic \cite[Thm.~7.9]{alexeev2021compact}.
\end{enumerate}

The connected components of the anti-invariant periods are determined by a lattice
$S=({\ker \psi})^+$ which has index at most $2$ in $S^{\rm sat}=\lias^+$ and furthermore contains
$(1+\rhoias)\lias^+$. Up to isometry, these are in bijection
with the sources $S$ of mirror moves with target $\hS$, by
Theorem \ref{thm:reverse-direction}. 

So these components of anti-invariant periods also parameterize involutive
Kulikov surfaces, and the same argument as in the previous paragraph applies
to prove that such an $X_0'$ smooths, with its involution, into moduli of $S$-polarized K3 surfaces.

Finally, we must verify that the monodromy invariant $\lambda$ is as stated.
This is follows from the Monodromy Theorem \ref{thm:EF-mirror}, which identifies $[\omega]=L\in \me^\perp/\me$
on $\hX$ with $\lambda\in e^\perp/e$ on the smoothing $X_t'$ of $X_0'$.
  \end{proof}


\bigskip

By Proposition \ref{stratum-function}, we must understand 
the stable models associated to each $\lambda\in \ch_2$ (including
for $\lambda$ only big and nef, on faces of $\ch_2$). For this, it will
suffice to show that the Kulikov models of Theorem \ref{smoothable}
are, after some appropriate M1 modifications (Section \ref{sec:models-remaining}),
adapted to $R$. Then, by passing to the stable model (\ref{pass-to-stable}),
we can explicitly determine the combinatorial
types of stable models on each face of $\ch_2$
in terms of the ADE surfaces of \cite{alexeev17ade-surfaces}
to find the semifan $\fF$.

This requires understanding in some detail how the singularities on the $\ias$ collide.
When $\lambda=L$ is ample, in the interior of $\ch_2$, the parameters
$\ell_i := \lambda\cdot\alpha_i$ are all positive, where $\alpha_i$ are
the simple roots. Then, no singularities $B(\lambda)$ have collided,
except the $I_{2g-2}$ singularity which forms on an edge of the equator.

For a non-general $\lambda$ with $\lambda^2>0$
lying on a face of $\ch_2$, some $\ell_i=0$. The corresponding $\ias$ $B(\lambda)$
is obtained from a generic $\ias$ $B(\lambda+t\lambda_0)$ for some $\lambda_0\in (\ch_2)^{\rm int}$ 
in the interior, by letting $t\to 0$. Then,
 the $I_1$ and $I_{2g-2}$ singularities on $B(\lambda+t\lambda_0)$ collide
 as $t\to 0$ to give more complicated singularities associated to higher charge
anticanonical pairs $(V_i,D_i)$. We discuss this in
Section~\ref{sec:type3-nongeneric}.

Ultimately, only the facets of $\chref$ which
are also facets of $\ch_2$ matter; for the other facets, 
the singularity collision can be avoided, and the stable models will not change.
For the (highly non-generic) $\lambda\in\chref$ satisfying
$\lambda^2=0$, the $\ias$ collapses to an interval, the dual graph of a Type II
degeneration. These are discussed in Section~\ref{sec:type2-models}.

\subsection{Edge behavior in the gluing of $P$ and $P\uopp$}
\label{sec:gluingPs}

Consider an $\ias$ $B=P\cup P\uopp$ with a set of integral points and
an involution $\iota$ exchanging $P$ and $P\uopp$. Let $C$ be a side
of $P$. The gluing along $C$ could be of two types. We say that it is
\emph{even} if for any lattice point $u\in P$ the lattice distance
between $u$ and $\iota(u)$ is always even, and \emph{odd} if it is odd for
some~$u$.  This is illustrated in Fig.~\ref{fig-equator}.

\begin{figure}[h]
  \includegraphics{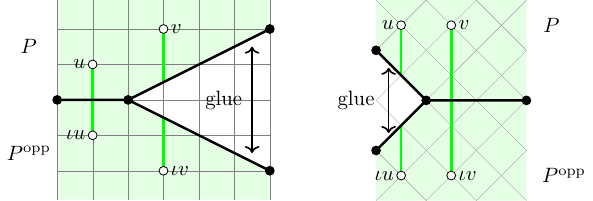}
  \caption{Equatorial behavior: even vs odd}
  \label{fig-equator}
\end{figure}

\begin{lemma}
  In the $I_{2\mk}I_0$ case, when $B$ has $2\mk>0$ $I_1$ singularities on
  the equator at the vertices of $P$, 
  gluing along the side corresponding to a curve $E_i$ on $\hY$ with
  $E_i^2=-4$ (resp. $E_i^2=-1$) is even (resp. odd).
\end{lemma}
\begin{proof}
  The statement is obvious from Fig.~\ref{fig-surgery}. 
\end{proof}

Note: when there are no singularities on the equator, there are no such
restrictions and below we will give examples of both even and odd
behaviors.

\smallskip

Next, we find anticanonical pairs $(V,D)$ corresponding to the
nonsingular lattice points of $B$ on the equator. Nonsingularity implies
means that these are toric pairs. We say that an involution on a toric
pair is {\it nonsymplectic} if the action on the cocharacter lattice $\bZ^2$
has determinant $-1$. Equivalently, it acts by multiplying a generator
of $H^0(V, \omega_V(D))$ by $-1$. 

\begin{lemma}\label{lem:equatorial-toric}
  Let $\iota$ be a nonsymplectic involution on a toric pair $(V,D)$
  such that the fixed locus $V^\iota$ does not contain any torus-fixed
  points of $V$. Then $V$ admits a generic ruling
  $\pi\colon V\to\bP^1$ with two sections $s_1$, $s_2$. The map
  commutes with a nontrivial involution on the base $\bP^1$ and there
  are three cases:
  \begin{enumerate}
  \item[(E0)] $V^\iota$ consists of four points, two on each $s_i$. 
  \item[(E1)] $V^\iota$ is a fiber of $\pi$ plus two points, one on
    each $s_i$.
  \item[(E2)] $V^\iota$ is two fibers of $\pi$. 
  \end{enumerate}
  The cases E0 and E2 give the even equatorial behavior, and
  E1 the odd equatorial behavior.
  The fan of $V$ is an involution-invariant subdivision 
  of the fan for $\bF_n$ with $n$ even for E0, E2, resp. with $n$ odd for E1.
\end{lemma}
\begin{proof}
  Consider the fan $\fF$ of $V$. The involution acting on the
  cocharacter lattice $N=\bZ^2$ has a fixed line and it does not
  intersect the interior of a two-dimensional cone in $\fF$ by the
  condition on $V^\iota$. So in $\fF$ there are two rays $\pm e_1$,
  fixed by the involution. Let $\la e_1,e_2\ra$ be a basis.
  Then $e_2\to -e_2 + ke_1$ with $k$ either even or odd, and we can
  assume that $k=0$ or $1$.
  Any nonsingular involution-invariant fan is a
  symmetric subdivision of the fan with the rays $\pm e_1$,
  $e_2 + me_1$, $-e_2+(k+m)e_1$ which is the fan of $\bF_n$ for
  $n=|k+2m|$. 
  This defines the action on monomials up to constants: $x^m\to c(m)
  x^{\iota m}$. It is easy to see that up to rescaling there are the
  following three possibilities, giving the three cases above:
  \noindent
  (E0) $(x,y)\to (-x,y\inv)$, or 
  (E1) $(x,y)\to (x, xy\inv)$, or
  (E2) $(x,y)\to (x,y\inv)$.
\end{proof}

\bigskip

We begin by describing Kulikov models for the lattices $\oT$ on the
$\mg=1$ line. Note that these only appear as the targets of simple mirror moves.

\subsection{Models for $\oT=(10+\mk,10-\mk,\delta)$, $1\le \mk\le 9$}
\label{sec:models-1}

By Theorem~\ref{thm:special-fibration}, for an ample line bundle $L$,
i.e. a generic monodromy vector $\lambda$, the $\ias$ $B$ has $24$
distinct $I_1$ singularities. Consider the Type III surface $X_0$
of Theorem \ref{smoothable} with its involution $\iota_0$ and involution-equivariant
smoothing $(X,\iota)$. The flat limit of $R_t={\rm Fix}(\iota_t)$ is the divisorial
locus $R_0$ in ${\rm Fix}(\iota_0)$. Note that since $X$ is smooth,
${\rm Fix}(\iota)$ as a set consists of relative curves $R\subset X$ and isolated
points in $X_0$.

Let us first deal with the pairs $(V_j,D_j)$ corresponding to the
lattice points in the interior of either hemisphere. They fall into
disjoint pairs $V_j$, $V_j\uopp$ exchanged by $\iota_0$.
Most of the pairs are
toric, entirely determined by the triangulation $\cT$.  There are also
$12-\mk$ pairs of $I_1$ singularities of charge $Q=1$. These are almost
toric pairs, with a single internal blowup at the side determined by
the monodromy ray.  The divisor $R_0$
is empty on these components.

\begin{figure}[htpb]
  \centering
  \includegraphics{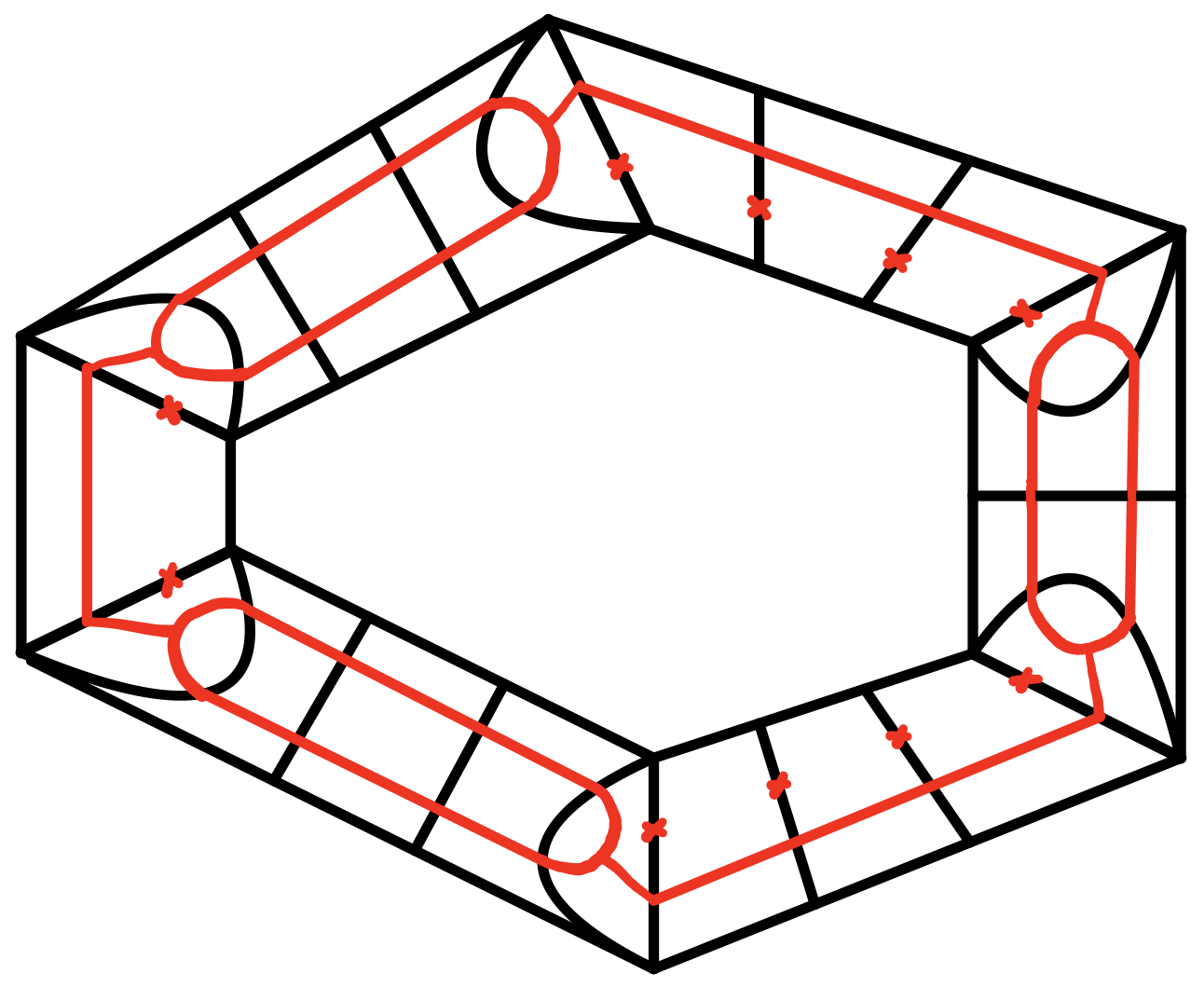}  
  \caption{Kulikov models for the $\mg=1$ case}
  \label{fig:type3-surfaces-generic}
\end{figure}

Next are the pairs corresponding to the nonsingular lattice points on
the equator. These are of types E1 or E2 of
Lemma~\ref{lem:equatorial-toric} respectively for the sides of $\hY$
with $E_i^2=-1$ and $E_i^2=-4$.

Finally, we need to determine the cbec, together with the
involution, of an $I_1$ singularity at a vertex of the polytope
$P$. From either Fig.~\ref{fig-surgery} or \ref{fig-equator} we see
that the minimal pseudofan for this singularity consists of two rays
$v_1=(-2,1)$ on the side with $E_1^2=-1$, $v_2=(1,0)$ on the side with
$E_2^2=-4$, and the monodromy ray $v\dmon=(0,-1)$. The balancing
condition for these vectors is $v_1+2v_2 = -v\dmon$.  This is the
pseudofan of $(\bP^2,D_1+D_2)$ with $D_1$ a line (so $D_1^2=1$),
$D_2$ a conic (so $D_2^2=4$), and an involution whose fixed
locus is a line $R$ and an isolated point.  The balancing condition
\cite[Sec.~2E]{alexeev2019stable-pair}
reads $\sum_{i=1}^2 (R\cdot D_i)v_i \in\bZ v\dmon$.

Therefore, $B$ is the $\ias$ given by
Construction~\ref{con:ias-from-kulikov} for the surfaces $X_0$ glued
from $2k$ pairs in the cbec of $(\bP^2,D_1+D_2)$, $\mk$ chains of
lengths $\ell_i=L\cdot E_i$ of type E1 for odd
$i$, $\mk$ chains of lengths $\ell_i=L\cdot E_i$ of type E2 for even $i$,
$12-\mk$ pairs of almost toric surfaces with a single blowup, and the
rest toric pairs. The triangulation $\cT$ specifies the combinatorial type of
each pair precisely.

\medskip

This surface can be read of directly from the Coxeter diagram for the
lattice $\oT$ given in Figs.~\ref{fig:coxeter1},
\ref{fig:coxeter2}. The most relevant part is the cycle of $2\mk$ white
vertices, alternating between double circled and single circled. Each
\emph{edge} on this cycle corresponds to a $\bP^2$. A single circled vertex,
with odd $i$ corresponds to a line, and a double circled vertex, with
even $i$ corresponds to a conic on $\bP^2$.

Each integral vector $\lambda\in\chref$ defines the nonnegative
integers $\ell_i=\lambda\cdot\alpha_i$. The numbers $\ell_i$ for these
$2\mk$ simple roots $\alpha_i$ on the $2\mk$-cycle specify the lengths of
chains as above. In addition, the other $\ell_i$ define the locations
of the nontoric surfaces in the hemispheres. On the mirror side, if
$\hY\to\oY$ is a toric model with exceptional curves $E_i$, $i>2\mk$, then
$E_i$ correspond to some of the vertices in the Coxeter diagram and
the lattice distance of a singularity from the equator is determined
in terms of $\ell_i=L\cdot C_i$. 
We give a cartoon picture for the equator of $X_0$ in Fig.~\ref{fig:type3-surfaces-generic},
with the fixed locus shown in red. The divisorial part $R_0$
of ${\rm Fix}(\iota_0)$ is a degeneration of genus $3$ curves
${\rm Fix}(\iota_t)$. 
\smallskip

A detailed example of this construction was given in
\cite{alexeev2019stable-pair} for $\oT=(19,1,1)$. In that case, there
are $18$ $I_1$ singularities on the equator, and $3$ pairs of
$I_1$ singularities in the hemispheres. The lengths
$\ell_0, \dotsc, \ell_{17}$ give the lengths of the chains of toric
surfaces on the equator. On the mirror side, the morphism to the toric
model $\hY\to\oY$ consists of three disjoint blowups, and the
distances of the nontoric surfaces $V_j$ from the equator are $L\cdot E_i$
for $i=18, 19, 20$.

\medskip

Each of the lattices $\oT$ on the $\mg=1$ line is the target of a
unique mirror move $S\to T\leadsto \oT$ as in
Definition~\ref{def:mirror-moves}, which is simple. So $S\subset\lias$
is saturated and the set of symmetric periods is the torus
$\bT_{\Lambda/S}$ and the family of Type III surfaces contains
the standard surface corresponding to the period $\psi=1$.

%

\subsection{Models for $\oT$ with $\mg\ge2$, excluding $(10,8,0)$}
\label{sec:models-remaining}

We use the reduction from the $\mg=1$ case given in
Section~\ref{sec:heegner}.  The lattice
$$\oT = \hS=(10+\mk-(\mg-1), 10-\mk-(\mg-1), \mdelta)$$ is obtained from
the lattice $\oT' = \hS'=(10+\mk,10-\mk,1)$ by a chain of $\mg-1$
Heegner $(-1,-1)$ moves. The set of possible monodromy invariants
$\lambda\in\oT$ lie in the fundamental chamber for $\oT$ which is a face
of the fundamental chamber for~$\oT'$, obtained by setting some of the
lengths $\ell_i$ with $i>2\mk$ to zero.

By Theorem~\ref{thm:special-fibration} the $\ias$ for $\oT$ is
obtained from the $\ias$ for $\oT'$ by descending $(\mg-1)$ pairs of
$I_1$ singularities in the interiors of the hemispheres symmetrically
to the equator, to a side of type E1 of
Lemma~\ref{lem:equatorial-toric} with odd equatorial behavior, to get
an $I_{2(\mg-1)}$ singularity with monodromy rays parallel to the
equator.

What remains to see is an involution on an actual surface $X_0$ and
specifically the surfaces $V_i$ with involution on the equator. To
describe them we introduce two moves, illustrated in
Fig.~\ref{fig:moves}.

\begin{figure}[htpb]
  \centering
  \includegraphics[width=3.5 in]{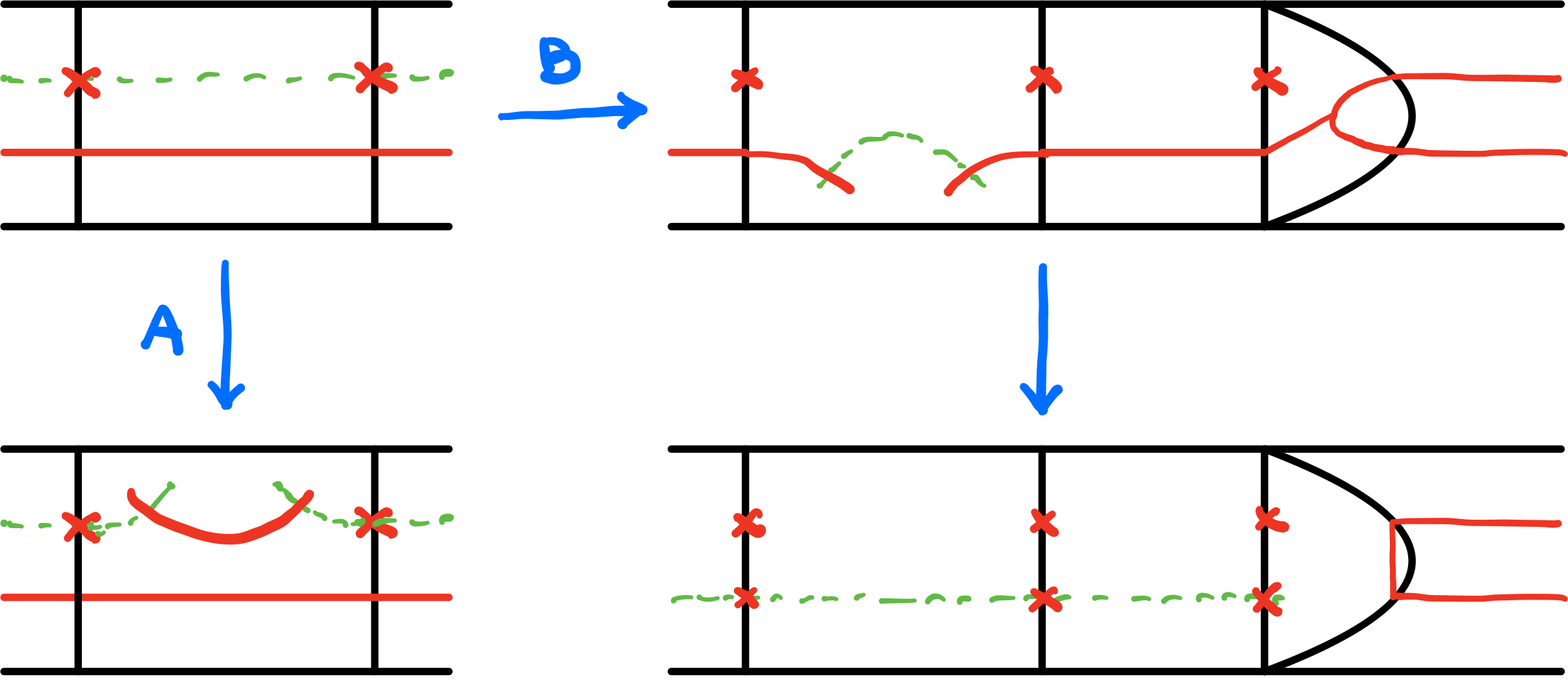} 
  \caption{A- and B-moves}
  \label{fig:moves}
\end{figure}

{\bf A-move:} Let $V'\to\bP^1$ be a toric surface of type E1 of
Lemma~\ref{lem:equatorial-toric}. The A-move is blowing up one of
the fixed points on a section, say $s_1$ twice. In the corresponding
fiber of $V\to\bP^1$ this creates the fiber $E_1+2E_2+E_3$ with
$E_1^2=E_3^2=-1$ and $E_2^2=-2$, the the interior $(-2)$-curve $E_2$
is fixed by the induced involution of~$V$. 

Let $(g,k)$ be the invariants of the K3 surfaces in the $S'$ family,
so that $\oT'$ was reached by the simple mirror move $S'\to
\oT$. Then the Type III surface we obtained belongs to the $S$-family
with the invariants $(g,k+1)$.

The point we have blown up twice in the toric model $\oV'$ of $V'$
corresponds to $-1$ with respect to the origin of the blown up side.
Therefore, we can for
example take $X_0$ of this type to be the standard surface, with
$\psi=1$. In particular, the A-move keeps a simple mirror move
$S\to \hS$ simple, in the sense that $S$ changes by a $(+1,-1)$ Heegner
move, while $\hS$ changes by a $(-1,-1)$ Heegner move. 

\smallskip

{\bf B-move:} For the B-move we instead blow up twice the other fixed point, the
one that lies on a fiber fixed by the involution pointwise. On the
surface $\Bl_2 V$ the fixed locus of the involution consists of two
$(-1)$-curves $E_1$, $E_3$ in the fiber $E_1+2E_2+E_3$. The divisor
$V^\iota$ fixed by the induced involution is not nef since it contains
$E_1$ and $E_3$ and not $E_2$. 

We flop $E_1$ to a neighboring surface $V_1$ by the
M1 modification. In $V_1$ this creates a fiber $E_1'+E_2'$
with $E'_1{}^2=E'_2{}^2=-1$, where $E'_1$ is not in $V_1^\iota$ and
$E'_2$ is. The ramification divisor is not nef again. We proceed by
flopping in $E'_2$ etc, until we reach a component in the cbec of
$(\bP^2,D_1+D_2)$ corresponding to the $I_1$ singularity of the
$\ias$. This is where the sequence of flops stops, making
$(\bP^2,D_1+D_2)$ with $D_1^2=1$, $D_2^2=4$ into $(\bF_1, D_1+D_2)$
with $D_1^2=0$, $D_2^2=4$ and the ramification into a fiber of
$\bF_1\to\bP^1$.

We do the same with the other $(-1)$-curve $E_3$ on $V$, flopping it
all the way to the surface in the cbec of $(\bP^2,D_1+D_2)$ on the
equator in the other direction. The end result is that the entire
chain of surfaces between the two $(\bP^2,D_1+D_2)$-components, which
in $X'_0$ was a chain of toric surfaces of equatorial type E1 is now a
chain of toric surfaces of equatorial type E0, with no divisorial part
in the fixed locus.

If $(g,k)$ are the invariants of the
K3 surfaces in the $S'$ family then the Type III surface $X_0$
obtained by the B-move has invariants $(g-1,k)$ and thus must be the
result of a non-simple mirror move.

\medskip

Now repeat this procedure $(\mg-1)$ times doing either the A-move
every time, or $(\mg-2)$-times the A-move and once the B-move.
Let $X_0$ be either of the surfaces obtained this way. There is an
entire $19$-dimensional family $\Hom(\Lambda,\bC^*)$ of $d$-semistable
surfaces $X_0(\psi)$ of the same deformation type.

As in Theorem \ref{smoothable}, 
the subset of anti-symmetric periods is a union of torus translates
of $\bT_{\Lambda/\Lambda^+}:={\rm Hom}(\Lambda/\Lambda^+,\bC^*)$. As before, $\Lambda=\lambda^\perp$ in
$\lias$, $\oT = \lias^-$, and $\Lambda^+ = \lias^+ = (\oT)^\perp=S^{\rm sat}$ in
$\lias$. The coset of $\bT_{\Lambda/\Lambda^+}$ containing the standard
surface is of the A...A type, and all other involution anti-invariant cosets
of $\bT_{\Lambda/\Lambda^+}$ are of the A...AB type.

\subsection{Monodromy invariants for $\oT=(10,10,0)$}
\label{sec:models-enriques}
In this case, the source of the mirror move must be $S= (10,10,0)$, so
$\oT$ corresponds to the self-mirror cusp of Enriques K3 moduli.

By Theorem~\ref{thm:special-fibration} and
Lemma~\ref{lem:special-degeneration}, the $\ias$ $B=P\cup P\uopp$ in
this case is obtained from the $\ias$ $B'=P\cup P\uopp$ of the
$(10,10,1)$ case by gluing $P\uopp$ to $P$ by twisting 
halfway along the circular boundary. This is possible, so long as the affine length of
the equator is even, which can always be achieved by replacing
$\lambda$ by $2\lambda$. So the Type III surface is obtained from a
Type III surface of the $(10,10,1)$ case by this rotation. The
involution $\iota_0$ on $X_0$ is base point free. 

$(10,10,0)$ is the only case when the
fundamental chambers $\ch_2$ and $\chref$ are
\emph{not} the nef cones of quotient surfaces $\hY$. Indeed, for a
general Enriques surface $\hY$ the nef cone is the round cone
$\ocC$. On the other hand, in the $\ias$ $B$, we used the
Symington polytope $P$ from the Halphen $(10,10,1)$ case,
which certainly has polyhedral walls, and so is not the entire round cone.

On the other hand, $\chref(10,10,0)$ is a subset of $\chref(10,10,1)$.
Indeed, the Coxeter
diagrams for both lattices are given in Fig.~\ref{fig:coxeter1} and
they are nearly identical. 
Denote the simple root vectors in $E_8(2)$ by
$\alpha_1,\dotsc,\alpha_8$, so that $\alpha_i^2=-4$ and
$\alpha_i\cdot \alpha_j\in\{0,2\}$ for $i\ne j$. Let $U=\la e,f\ra$ with
$e^2=f^2=0$, $e\cdot f=1$. One has $(10,10,0)=E_8(2)\oplus U(2)$ and
$(10,10,1) = E_8(2)\oplus I_{1,1}(2)$, where $U(2)=\la 2e,f\ra$ and
$I_{1,1} = \la 2e, e+f\ra$ are two different index~$2$ sublattices of
$U$.

Then the first $9$ vectors of the two Coxeter diagrams are exactly the
same: $\alpha_1, \dotsc, \alpha_8, \alpha_9=2e-\alpha_0$, where
$-\alpha_0$ is the longest root of $E_8$.
It is only the last, $10$th roots that are different. For $(10,10,0)$
it is $\alpha_{10}=-2e+f$, and for $(10,10,1)$ it is
$\alpha'_{10} = -e+f$. Thus, $\alpha'_{10} = \alpha_{10} + e$. Since
$e$ is a positive linear combination of $\alpha_1,\dotsc,\alpha_9$, if
$\lambda$ is a vector in $\chref(10,10,0)$, i.e. a vector satisfying
$\lambda\cdot \alpha_i\ge0$ for $i=1,\dotsc, 10$ then one moreover has
$\lambda\cdot \alpha'_{10}\ge 0$.

\smallskip

So for any $\lambda\in\chref(10,10,0) \subset \chref(10,10,1)$ we interpret
$\lambda$ as a vector defining the Symington polytope for $(10,10,1)$
by the previous construction, and take $B$ to be $P\cup P\uopp$ glued
with a half-circle rotation in the equator.

\subsection{Models for $\oT=(10,8,0)$, $(10,10,0)$, $(10,10,1)$}
\label{sec:models-res}

There are five mirror moves $S\leadsto \oT$ for which the integral
affine structure $B=P\cup P^{\rm opp}$ has a circular
equator with no singularities on it. They are:

\begin{align}\label{circles}\begin{aligned} &(10,10,0)\to (10,10,0), && (10,10,1)\to (10,10,1), \\
&(10,8,0)\to (10,8,0), && (10,10,0)\dto (10,8,0), && (10,10,1)\dto (10,8,0).
\end{aligned} \end{align} 

In particular, the lattice $\oT=(10,8,0)$ is the target of three mirror moves.
Respectively, there are three families of Type III
surfaces with anti-symmetric periods.
All of them share the same $\ias$ $B=P\cup P\uopp$, glued from two
copies of a Symington polytope $P$ for the anticanonical pair
$(\hY,\hD)$ with $\hY=Bl_{p_1,\dots,p_9}\bP^2$ and $\hD$ the strict
transform of a cubic through the nine points $p_i$. The gluing as in
Section~\ref{sec:gluingPs} is with the even equatorial behavior.
The only difference is in the periods.

For the $(10,10,1)\to (10,10,1)$
mirror move, the target $\ias$ determined by $\oT$ is $P\cup P\uopp$
with $P$ as in the previous paragraph, but with the odd equatorial behavior.

The $(10,10,0)\to (10,10,0)$ mirror move is the only one for which
the target $\ias$ satisfy $B/\iota \simeq \bR\bP^2$. There is no well-defined
equatorial edge behavior since the edge is not fixed, and the sphere $B$ is built
from Symington polytopes for $(\hY,\hD)$ with a half-twist, as in Section \ref{sec:models-enriques}.

The shape of the equator, and the fixed locus
of $\iota_0$ inside it, is depicted in Figure \ref{circular-equators},
in the respective positions of the mirror moves in Equation \ref{circles}.
In all cases, the equatorial behavior of $R_0$ depicted is uniquely determined by the
fact that the divisorial part of $R_0$ must be the flat limit of $R_t$.

\begin{figure}
\includegraphics[width=4in]{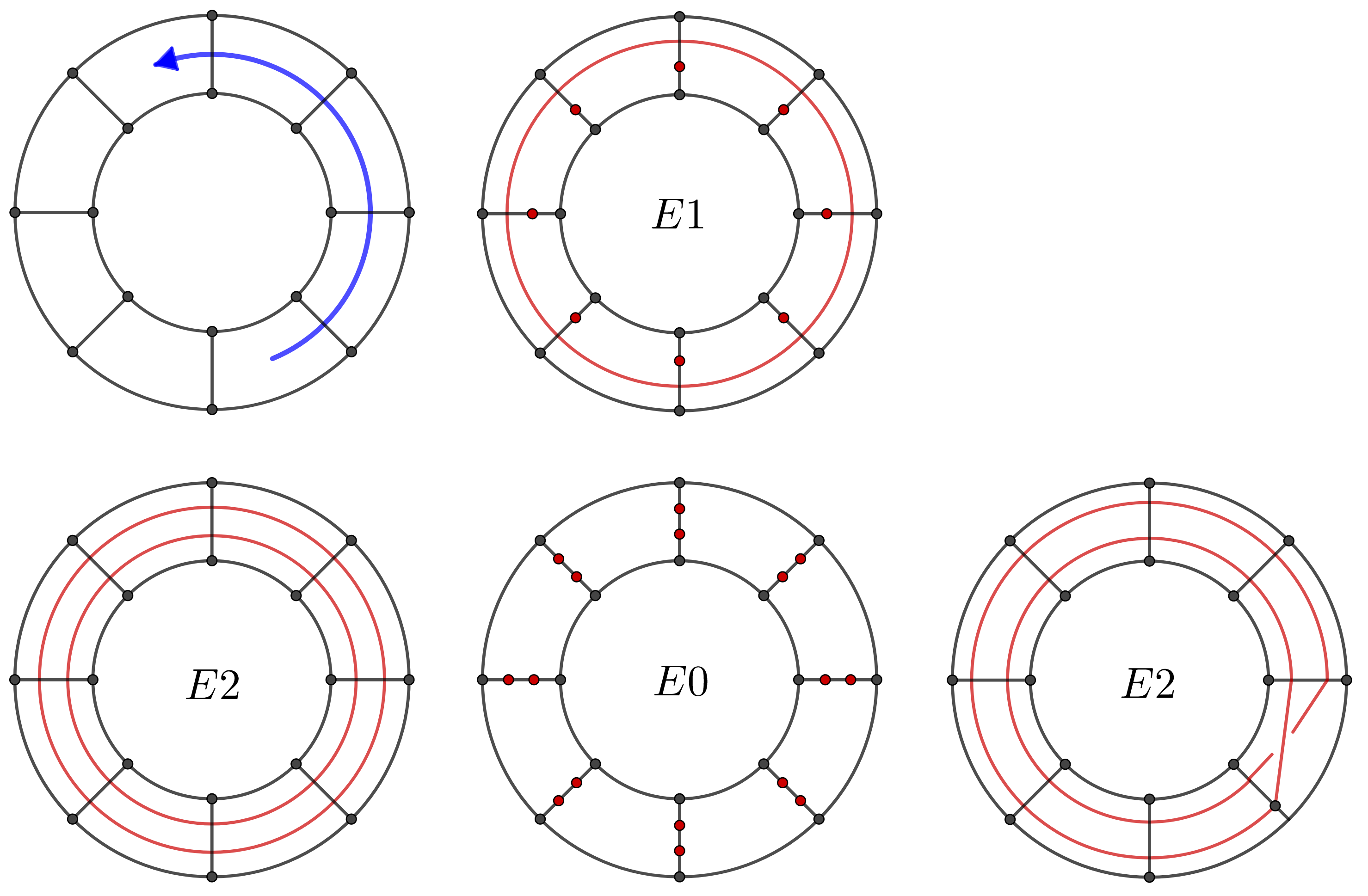}
\caption{The fixed loci in circular equators, shown in red,
and the corresponding edge behavior.}
\label{circular-equators}
\end{figure}

\subsection{Colliding singularities and non-generic type III models}
\label{sec:type3-nongeneric}

In the previous sections, we defined Kulikov models for the generic
$\lambda$. With the exception of Section~\ref{sec:models-remaining},
the $\ias$ had $24$ distinct $I_1$ singularities: $2\mk$ of them lie
on the equator, and $12-\mk$ in the interior of each of the hemispheres
$P$, $P\uopp$ in pairs, so that the whole collection is preserved by
the involution $P\leftrightarrow P\uopp$.

For a non-generic $\lambda$ and for the lattices with $\mg\ge2$, some
of these singularities collide to form integral-affine singularities
of higher charge.

Each of the $I_1$ singularities on the equator corresponds to
a surface in the cbec of $(\bP^2,D_1+D_2)$ and the ramification
divisor is a line $L$ of volume $L^2=1$. Also, after a B-move, the M1 modifications
convert each end $(\bP^2,D_1+D_2)$ into $(\bF_1,D_1+D_2)$ with the
ramification divisor a fiber, of volume $L^2=0$. 

On the other hand, the off-equator $I_1$ singularities correspond to 
surfaces with an empty ramification divisor, of volume $0$. When 
singularities collide, the volumes add. The anticanonical pairs
of volume $0$ get contracted on the stable models, which are
described in Section~\ref{sec:stable-models}.
This leads to three types of collisions:
  
\smallskip

(1) Collisions in the interiors of the hemispheres $P$,
$P\uopp$. These are completely irrelevant and can be avoided by nodal
slides, replacing them by isolated~$I_1$.
The corresponding anticanonical pairs
$(V_i,D_i)$ are disjoint from the $C_g$ component of the ramification
divisor $R$, so on the stable model $V_i$ will be contracted to points.

\smallskip
  
(2) Collisions obtained by descending pairs of $I_1$ singularities
to the equator. These are only
partly relevant: without further collisions, the curve $C_g$ has numerical
dimension $1$ or $0$ on the corresponding
pair $(V_i,D_i)$. So on the stable model, $V_i$ is contracted to $\bP^1$ or a point
and the collision is not detected by the stable model.

\smallskip

(3) The singularities obtained by colliding some of the $2\mk$ $I_1$
singularities on the equator, between themselves, or with some of the
pairs of $I_1$ singularities from case~(2). These are truly
relevant. Each isolated $I_1$ singularity on the equator corresponds to
a pair $(V_i,D_i)$ in the cbec of $(\bP^2,D_1+D_2)$ with a line and a
conic, and the curve $C_g$ restricts to this $\bP^2$ as a line. On the
stable model, this $V_i$ will not be contracted, mapping to a surface
whose normalization is $\bP^2$. Further collisions lead
to more complicated nontoric anticanonical pairs $(V_i,D_i)$ with big
and nef $C_g$, and to more complicated irreducible components of the
stable model. 

\medskip

Below, we work with the chamber $\ch_2$ and its Coxeter diagram
$\Gamma_2$. The translation between $\ch_2$ and $\chref$ and the
elliptic subgraphs of $\Gamma_2$ and $\gref$
is given by Eq.~\eqref{eq:conversion} and Fig.~\ref{fig:conversion}.

\begin{definition}\label{def:lambda-subgraph}
  For any $\lambda\in\ch_2$  define the
  subdiagram $\Gamma_2(\lambda)\subset\Gamma_2$ with the vertices
  \begin{displaymath}
    V_2(\lambda) = \{ i \mid \ell_i = 0 \},
    \quad\text{where } \ell_i = \lambda\cdot\alpha_i
    \text{ for simple roots } \alpha_i.
  \end{displaymath}
  Recall that $\ell_i\ge0$ for any $\lambda\in\ch_2$.  $\Gamma(\lambda)$
  specifies the face of $\ch_2$ to which $\lambda$ belongs.  

  Further, by Section~\ref{sec:heegner}, for a lattice $\oT$ of genus
  $\mg$ the chamber $\ch_2(\oT)$ 
  is a face of the chamber $\ch_2(\oT_1)$ for the ``parent'' lattice
  $\oT_1 = \oT\oplus A_1^{\mg-1}$. Let
  $\Gamma_2\upar(\lambda)\subset\Gamma_2(\oT_1)$ be the
  diagram of $\lambda$ considered as a vector in~$\oT_1$.
  With notations of Corollary~\ref{cor:reach-by-moves} one has
  $\Gamma_2\upar(\lambda) = \Gamma(\lambda) + A_1^{\mg-1}$, see
  Lemma~\ref{lem:subgraphs-G-G1}. 
\end{definition}

By Eq.~\eqref{eq:ch-faces} in Section~\ref{sec:reflection-groups}, if
$\lambda^2>0$ then $\Gamma_2\upar(\lambda)$ is elliptic, i.e. a
disjoint union of $ADE$ diagrams. As in
Section~\ref{sec:visible-curves}, these $ADE$ diagrams correspond to
collections of visible curves connecting the $I_1$
singularities. We illustrate this in Fig.~\ref{fig:collide} showing a
collision of type (2) that gives an $I_2$ of charge $Q=2$ and a collision
of type (3) that gives $D_5$ of charge $Q=8$. 

\begin{figure}[h]
  \includegraphics[width=3.5 in]{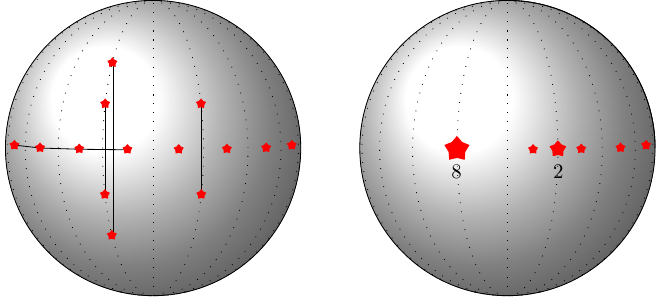}
  \caption[Colliding $I_1$ singularities]{Colliding $I_1$
    singularities:
    $D_5$ (charge $=8$, volume $=4$) and
    $A_1$ (charge $=2$, volume $=0$)}  
  \label{fig:collide}
\end{figure}

The visible curves $(\gamma_i,\alpha_i)$ connecting the pairs of $I_1$
singularities with the same monodromy directions are in a bijection
with some of the simple roots $\alpha_i$ of $\Gamma_2\upar$. The
corresponding lengths $\ell_i=\lambda\cdot\alpha_i$ measure the
lattice distances between these $I_1$ singularities. A collision
occurs when all of these distances become zero.

For each singularity of type (3) the divisor $C_g$ on the
corresponding anticanonical pair $(V,D)$ is big and nef, so
$K_V+D+\epsilon C_g$ is big and nef. Its canonical model
$(\oV,\oD + \epsilon \oC_g)$ is one of the $ADE$ stable involution
pairs which were completely classified in
\cite[Table~2]{alexeev17ade-surfaces}. The cbec of $(V,D)$ is the cbec
of the minimal resolution of singularities of $(\oV,\oD)$.

\cite{alexeev17ade-surfaces} devised an elaborate
system for denoting the $ADE$ surfaces using decorated Dynkin symbols.
For example the $D_5$ singularity of Fig.~\ref{fig:collide}
corresponds to $\pD_5^-$ of \cite[Table~$1$]{alexeev17ade-surfaces}.
Perhaps the easiest way to
understand them is to look at the $ADE$ subgraphs of the
Coxeter diagrams in Figs.~\ref{fig:coxeter1}, \ref{fig:coxeter2},
while remembering the conversion rules of Eq.~\eqref{eq:conversion} and
Fig.~\ref{fig:conversion}. Such a subgraph $G$ contains more information
that just its $ADE$ type. For example, an $A_n$ subgraph can sit in
$\Gamma_2$ in many different ways: it may lie entirely in the equatorial
$2\mk$-cycle, or one or two of its vertices may venture off-equator. It
may begin at a single-circled or a double-circled white vertex.
All of these possibilities correspond to different cbecs of anticanonical pairs
with an involution.
But the decorated $ADE$ symbols specify them uniquely.

The most useful way to think of an $ADE$ subgraph $G\subset\Gamma_2$ is
as a chain in the  
equatorial $2\mk$ cycle, plus some off-equator vertices each
attached to an even-numbered equatorial,
double-circled white vertex $\alpha_i$. For a given
$\alpha_i$, the off-equator vertices~$\beta_j$ connected to it are
disjoint from each other, cf. Fig.~\ref{fig:conversion}.

With the singularities of $B=\Gamma(X_0)$ 
modeled by the cbecs of the corresponding $ADE$ surfaces 
of \cite{alexeev17ade-surfaces}, the proof of Theorem \ref{smoothable}
goes through, essentially unchanged.

\subsection{Type II models}
\label{sec:type2-models}

Type II degenerations correspond to monodromy vectors
$\lambda\in\ch_2$ with $\lambda^2=0$. On the mirror side, the line bundle
$L\in\hS$ is nef but not big, $L^2=0$. Since the lattice area of
$\ias$ is $\lambda^2=L^2$, the sphere degenerates to an interval, the dual
complex of a Type II Kulikov model. The limit of an $\ias$
$B=P\cup P\uopp$ with an involution itself has an involution, so the
interval is either vertical, going from one pole of the sphere to
another, or horizontal.

The subgraph $\Gamma_2(\lambda)$ in this case is parabolic, a disjoint
union of finitely many $\wA_n$, $\wD_n$, $\wE_n$ extended Dynkin
diagrams. Just as in the previous section, the collisions are of types
(1) irrelevant, (2) partly relevant but not changing the stable model,
and (3) those that include some vertices on the $2\mk$-cycle, which do
affect the stable models. The stable models for
$(V_i,D_i,\epsilon C_g)$ in this case are the $\wA\wD\wE$ surfaces
classified in \cite{alexeev17ade-surfaces} and $(V_i,D_i)$ are their
resolutions to nonsingular pairs.

\section{Compact moduli}
\label{sec:compact-moduli}

We now restrict ourselves to the moduli spaces $F_S$ for the $50$
lattices $S$ of \fign\ for which $g\ge2$, excluding $(10,8,0)$. In
these cases the ramification divisor $R=X^\iota$ has a unique
component $C_g$ of genus $g\ge2$. If $X\to\oX$ is the contraction
given by a linear system $|mC_g|$ then the image $\oC_g$ of $C_g$ is
an ample Cartier divisor on $\oX$ and the pair $(\oX, \epsilon\oC_g)$
is a stable KSBA pair for $0<\epsilon\ll1$, so one has the KSBA
compactification as in
Section~\ref{sec:ksba-compactifications}. Recall that we denoted the
main component $\oF_S$ and its normalization was the semitoric compactification
$\oF^\fF=(\oF_S)^\nu$ by Proposition \ref{stratum-function}. 

The goal of this section is to prove the Main
Theorem~\ref{thm:compact-moduli}. The proof is given in
Section~\ref{sec:proof-main-thm}. We begin by defining the semifans
$\fF_I = \fram(\oT)$ forming $\fF$ as one ranges over the $0$-cusps
$I = e^\perp/e$. Then, we discuss stable models. 

\subsection{The ramification semifan}
\label{sec:semifan}

We use the same notations as in Section~\ref{sec:vinberg-theory}.  Let
$H=\oT=e^\perp_T/e$ be the hyperbolic lattice corresponding to a
$0$-cusp of $F_S$. Let $\Gamma$ be its Coxeter diagram.

Recall that we defined the Coxeter semifan in
Section~\ref{sec:coxeter-semifan}.  We now define a \emph{generalized Coxeter
semifan} by the ``Wythoff's construction'' of
\cite{coxeter1935wythoffs-construction}. We refer to
\cite[Sec.~10C]{alexeev2019stable-pair} for details. 

Divide the vertex set $V(\Gamma)$ of the Coxeter diagram into two
sets $\vrel\sqcup\virr$ of {\it relevant} and {\it irrelevant} roots.
(It is important to understand that the ``irrelevant" vertices are
irrelevant only when in isolation. They become relevant if they are
connected to the relevant ones.)

\begin{definition}\label{def:generalized-cox-semifan}
Modulo $W$, the fan $\fgen$ has a unique maximal dimensional
cone
\begin{displaymath}
  \ch\dgen = \cup_{w\in W\dirr} w.\ch, \quad \text{where }
  W\dirr = \la w_i \mid i\in \virr \ra.
\end{displaymath}
All the other cones in $\fgen$ are the faces of $\ch$ and their
$W$-translates.
\end{definition}

\begin{example}
  Consider the Coxeter fan for $W(A_2)=S_3$ in $\bR^2$, consisting of
  $6$ two-dimensional cones and their faces. If we take for $\virr$
  one of the vertices of the~$A_2$ diagram then $\fgen$ will be a fan
  with $3$ two-dimensional cones and it is \emph{not} a reflection
  fan. 
\end{example}

$\fgen$ has fewer cones than $\fF$ in the following sense:

\begin{definition}
  Let $V'\subset V(\Gamma)$ be a subset of vertices and
  $\Gamma'\subset\Gamma$ the induced subdiagram which is either
  elliptic or parabolic.  We divide $V'$ into the relevant and
  irrelevant parts $\relpart(V') \sqcup \irrpart(V')$ as follows: the
  vertices of the connected components of $\Gamma'$ that consist
  \emph{entirely} of irrelevant vertices form $\irrpart(V')$, and the
  rest is $\relpart(V')$. We say that $\relpart(V')$ is \emph{the
    relevant content of} $V'$.
\end{definition}

Then the cone $F'$ of $\fgen$ defined by $V'$ is the same as the
one defined by its relevant content, and $F\cap \ch$ is the cone
$F(\relpart(V'))$ of $\fF$. In particular, if
$\relpart(V')=\emptyset$ then $V'$ does not define any cone in
$\fgen$ at all.

\begin{remark}
  Let $W=\wref$. 
  If $\virr=\bB$ is a subset of $V_4$ as in
  Definition~\ref{def:refl-chambers} then by
  Lemma~\ref{lem:refl-chambers} $\fgen$ is simply the Coxeter semifan
  for the Weyl group $W\dnor(\bB^c)$. But in our case the set $\virr$
  is \emph{bigger} than $V_4$, and $\fgen$ is not a Coxeter
  semifan for any Weyl group.
\end{remark}

\begin{definition}\label{def:the-semifan}
  Let $\oT_1 = (10+\mk, 10-\mk, \mdelta)$, $1\le \mk\le 9$ be a
  lattice on the $\mg=1$ line and $\Gamma(\oT_1)$ be its Coxeter
  diagram as in Sections~\ref{sec:vinberg-g=1},
  \ref{sec:vinberg-exceptional}.  We declare the only relevant
  vertices $\vrel$ to be the $2\mk$ vertices on the $I_{2\mk}$ cycle,
  and all the other vertices to be irrelevant. In particular, all the
  black vertices are irrelevant.

  Any other lattice $\oT$ appearing as the target of a mirror move
  $S\leadsto\oT$ for $S$ with $g\ge2$ is of the form
  $\oT_{\mg} = (10+\mk-(\mg-1), 10-\mk-(\mg-1),\mdelta)$.  By
  Corollary~\ref{cor:reach-by-moves}, it can be reached from
  $\oT_1= (10+\mk, 10-\mk, 1)$ by a chain of $\mg-1$ standard
  operations described in Lemma~\ref{lem:heegner-diagrams}, and the
  set of vertices of $\Gamma(\oT_{\mg})$ is a subset of the vertices
  of $\Gamma(\oT_1)$ We declare the vertices of $\Gamma(\oT_{\mg})$ to
  be relevant if they were also relevant in $\Gamma(\oT_1)$.  For
  $\mg\ge2$ there are $2\mk-1$ of them.

  We define the semifan $\fram(\oT)$ to be the generalized Coxeter fan
  for this set $\virr$ of relevant vertices. It is a coarsening of the
  reflection fan $\fF_2$ for the $(-2)$-roots: Walls of $\fF_2$
  get erased exactly when they are the perpendicular to an entirely
  irrelevant diagram.
  
\end{definition}

\subsection{Stable models}
\label{sec:stable-models}

Let $X_0$ be one of the Type III surfaces with an involution~$\iota$
defined in Section~\ref{sec:main-construction}, $R$ the divisorial
part of the fixed locus $X^\iota$, and $C_g$ the connected component
of $R$ that has arithmetic genus $g\ge2$. In each case, we constructed
a family $X_0$ of such surfaces with an involution,
parameterized by a translate of a torus. After any necessary M1 modifications
resulting from a B-move, see Section \ref{sec:models-remaining},
the limit of $C_g$ is a big and nef divisor, and a large multiple
$|mC_g|$ defines a contraction $f\colon X_0\to \oX_0$. Denote by
$\oC_g$ the image of $C_g$, it is a Cartier divisor. The pair
$(\oX_0,\epsilon \oC_g)$ is a KSBA stable pair, see
Section~\ref{sec:ksba-compactifications}.

\begin{lemma}\label{lem:contractions-to-stable-model}
  Let $X_0=\cup V_i$ with $V_i$ corresponding to the lattice points of
  $B=P\cup P\uopp$. Assume $\mg=1$. Then under the morphism $f\colon
  X_0\to \oX_0$
  \begin{enumerate}
  \item if $i$ is in the interior of $P$ or $P\uopp$ then $V_i$ is
    contracted to a point,
  \item if $i$ is on the equator but not at a vertex of $P$ then $V_i$
    is contracted to $\bP^1$,
  \item if $i$ is a vertex of $P$ on the equator then $V_i\to\oV_i$ is
    birational. 
  \end{enumerate}
\end{lemma}
\begin{proof}
  Indeed, by construction, these are the irreducible components of
  $X_0$ on which $C_g$ has numerical dimension $0$, $1$, $2$,
  compare the description of the pairs $(V_i,D_i)$ of types (1), (2), (3) in 
  Section~\ref{sec:type3-nongeneric}.
\end{proof}

Recall that a lattice
$\oT_{\mg} = (10+\mk-(\mg-1), 10-\mk-(\mg-1),\mdelta)$ of genus
$\mg>1$ is reached from $\oT_1= (10+\mk, 10-\mk, 1)$ by a chain of
$\mg-1$ standard operations described in
Lemma~\ref{lem:heegner-diagrams}. By the construction of
Section~\ref{sec:models-remaining}, if the mirror move $S\to\oT$
is odd then the Kulikov surface for $\oT_{\mg}$ is obtained from that
of $\oT_1$ by a sequence of moves A{\dots}AA, and if $S\dto \oT$ is even
then it is by a sequence of moves A{\dots}AB. 
In either case, if $\mg\ge2$ then the $I_{2\mk}$-cycle on the equator
is broken and becomes a chain of length $2\mk-1$. 

\begin{definition}
  Let $(V,D=\sum D_j)$ be an anticanonical pair and $L$ be a big and
  nef divisor on $V$. We say that an irreducible component $D_j$ of
  $D$ is a \emph{short side} if $LD_j=1$, a \emph{long side} if
  $L\cdot D_j=2$, and a \emph{zero side} if $L\cdot D_j=0$. If $f\colon V\to\oV$
  is the contraction defined by $L$ and $L=f^*\oL$, then we say that
  the images $\oD_j$ of $D_j$ are short or long sides of $\oV$ if
  $\oL\cdot\oD_j=1$ or $2$ respectively.
\end{definition}

\begin{definition}\label{def:stable-models}
  We define two types of stable models $\oX_0=\cup_i (\oV_i,\oD_i,\epsilon \oC_{g,i})$,
  illustrated 
  in Fig.~\ref{fig:pumpkin}.

  (1) \emph{Pumpkin.}  Each surface $\oV_i$ has two sides
  $\oD_i = \oD_{i,}\dleft + \oD_{i,}\dright$, they are glued in a
  circle, all of $D_i$ meeting at the north and south poles.
  \begin{figure}[htbp]
    \centering

    \ifshort 
    \includegraphics[width=2.5 in]{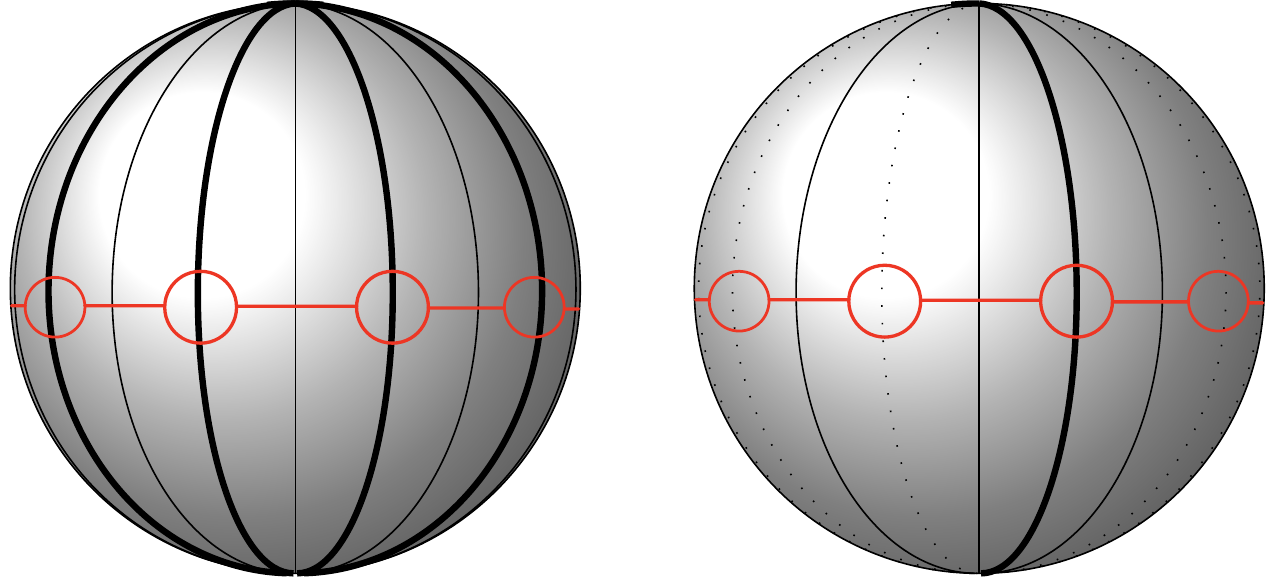}
    \hskip .8in
    \includegraphics[width=2.5 in]{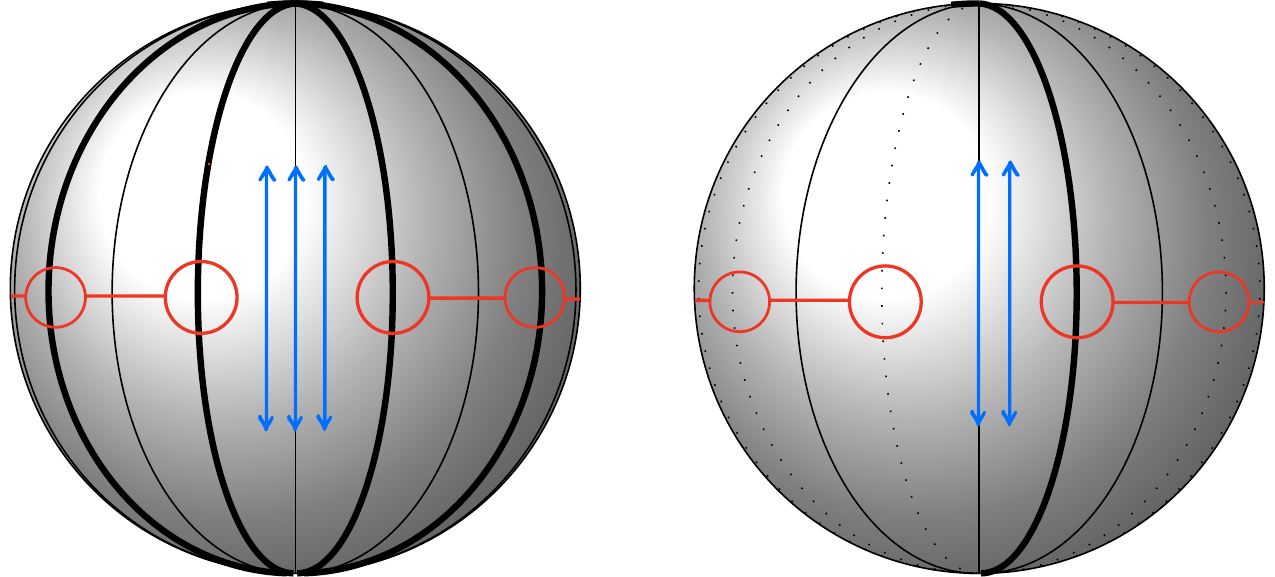}

    \else
    \includegraphics[width=.7\linewidth]{fig-pumpkin}
    \vskip 1cm
    \includegraphics[width=.7\linewidth]{fig-pumpkin2}
    \fi

    \caption{Pumpkin and smashed pumpkin type stable models}
    \label{fig:pumpkin}
  \end{figure}

  (2) \emph{Smashed pumpkin.} Starting with a surface of the pumpkin
  type, one short side is contracted to a point, so that the north and
  south poles are identified.

  If the surface $V_i$, say to the left, is
  $(\bF_1,D_1+D_2)$, where $D_1\sim f$ is the short side being
  contracted, $D_2\sim 2s_1+2f$ is the other side, and $C_{g,i}\sim f$ on $V$
  contract $V_i$ by the $\bP^1$-fibration $V_i\to\bP^1$. Then
  on the next surface $V_{i-1}$ to the left the long side will fold $2:1$ to
  itself, creating a nonnormal singularity along that side.

  If on $V_i$ the divisor $C_{g,i}$ has degree $C_{g,i}^2\ge2$ then only the
  short side is contracted and the resulting surface $\oV_i$ is normal
  in codimension~$1$, with only two points in the
  normalization glued together (the poles).
\end{definition}

\begin{theorem}\label{thm:stable-models}
  Let $(\oX_0 = \cup_i\oV_i, \epsilon \oC_g)$ be the stable model of a
  pair $(X_0=\cup_i V_i, \epsilon C_g)$, where $X_0$ is a Type III
  Kulikov surface $X_0 = \cup V_i$ and $C_g$ is the component of genus
  $g\ge2$ in the ramification divisor $R$. Then the normalization of
  each $\oV_i$ is an $ADE$ surface with an involution from
  \cite[Table~2]{alexeev17ade-surfaces}. Moreover, 
  \begin{enumerate}
  \item If $\oT$ is an odd $0$-cusp of $F_S$ then $\oX_0$ is of pumpkin
    type.  
  \item If $\oT$ is an even $0$-cusp of $F_S$ then $\oX_0$ is of 
    smashed pumpkin type. The surfaces $V_i$ of the last type in
    Definition~\ref{def:stable-models}, on which $V_i\to\oV_i$
    contracts one side are the surfaces of
    \cite[Table~2]{alexeev17ade-surfaces} for which one of the sides
    has length $0$, i.e. those with a double prime or a ``$+$''.
  \end{enumerate}
\end{theorem}
\begin{proof}
  By observation, these are the surfaces obtained by contracting the
  Kulikov models we constructed, defined by the big and nef divisor
  $C_{g,i}$ on $V_i$.  
\end{proof}

Figure~~\ref{fig:pumpkin}  
shows the 
maximally degenerate stable model $\oX_0$ and a less degenerate
one. The maximally degenerate surface corresponds to the empty
subdiagram $V_2(\lambda)=\emptyset$.

In the pumpkin type it is a union of $2\mk$
$(\bP^2,D_1+D_2)$ with a short side (a line) and a long side (a
conic). Long, (resp. short) sides correspond to the even-numbered,
doubly circled (resp. odd-numbered, singly circled) white vertices on
the $2\mk$ cycle in $\Gamma_2$.

In the smashed pumpkin type, $\oX_0$ is a union of $2\mk-2$
components, $2\mk-4$ of them are isomorphic to $(\bP^2,D_1+D_2)$ and
the remaining two to $(\bP^2,D_1+D_2)$ with a conic doubly folded on
itself. 

For nongeneric $\lambda$, for each vertex on the equator that is
included in $V_2(\lambda)$, the corresponding side is smoothed out and
removed from the picture.
For each off-equator vertex in $V_2(\lambda)$ which is
connected to some equatorial vertex in $V_2(\lambda)$, the combinatorial
type of the component $\oV_i$ changes, since the charge of its minimal
resolution increases.

\subsection{The Main Theorem}
\label{sec:proof-main-thm}

We recall that in $50$ of the $75$ lattices of \fign\ for a generic
surface $(X,\iota)$ in the moduli space $F_S$, the fixed locus
$R=X^\iota$ has a component $C_g$ of genus $g\ge2$.  These are the
lattices with $g\ge2$ excluding $(10,8,0)$.

Then $C_g$ is big, nef, and semiample, and defines a contraction
$X\to\oX$ to a K3 surface with $ADE$ singularities and an ample
Cartier divisor $\oC_g$. The KSBA theory
(Section~\ref{sec:ksba-compactifications}) then provides a geometric
compactification of $F_S$ by adding stable pairs $(\oX,\epsilon\oC_g)$
on the boundary. We denoted by $\oF_S$ this closure.

\begin{theorem}\label{thm:compact-moduli}
  In the $50$ cases of interest, the normalization of
  $\oF_S$ is semitoroidal, given by the collection of
  the semifans $\fF=\{\fram(\oT)\}$ defined in
  Section~\ref{sec:semifan}, one for each $0$-cusp $\oT$ of $F_S$.
  It is dominated by the semitoroidal compactifications
  $\ofsfref$ and $\oF_S^{\fF_2}$
  for the Coxeter semifans $\fref=\{\fref(\oT)\}$ and $\fF_2=\{\fF_2(\oT)\}$.
  
  For $S\ne(2,2,1)$ or $(3,3,1)$, $\ofsfref$ is toroidal. For $S$
  with $k\ge1$, both $\oF_S^{\fF_2}$ and $\ofsfref$ are toroidal.
\end{theorem}

\begin{proof}
The first part follows from Proposition \ref{stratum-function}
and the description of the combinatorial type of the KSBA-stable 
limit, for each $\lambda\in \ch_2$ described in Theorem~\ref{thm:stable-models}:
The combinatorial type changes whenever $\lambda$ degenerates
into a cone with larger relevant content, since the number of double
curves of the stable limit $\oX_0$ decreases. The semifan thus
obtained is the semifan defined in
Section~\ref{sec:semifan}.

By Theorem~\ref{thm:0-cusps}, the $0$-cusps of $F_S$ correspond to the
mirror moves $S\leadsto\oT$. For each $\oT$, the semifan $\fram(\oT)$ is
a coarsening of the Coxeter fan for the Weyl group $W_2$ generated by
the $(-2)$-vectors, which is in turn a coarsening of the Coxeter fan for the
full reflection group $\wref$.

For the lattices $S$ with $g\ge 2$, their mirrors $\oT$ are the
$2$-elementary lattices with $\mg\ge1$ and $k\ge1$, excluding
$(14,6,0)$. In all of these cases except for $\oT=(18,2,1)$ and
$(17,3,1)$, $\fref$ is a fan, and these $\oT$ appear as mirror
only for $S=(2,2,1)$ and $(3,3,1)$. So in these cases $\ofsfref$ is
toroidal.

Similarly, for the lattices $S$ with $g\ge1$, the mirrors $\oT$
satisfy $\mg\ge2$. Excluding the case $S=\oT = (10,8,0)$, the Weyl
group $W_2$ has finite covolume (see Section~\ref{sec:vinberg-g>1}),
so the Coxeter semifan for $W_2$ is a fan.
\end{proof}

\begin{remark}
From Definition~\ref{def:generalized-cox-semifan}, $\fram(\oT)$ is a
fan iff $\wref$ has finite covolume
and the Weyl group generated by the irrelevant roots is finite,
i.e. the complement of the relevant vertices forms an elliptic
diagram. (Indeed, this is equivalent to the polyhedron
$P\dram\subset\cH$ for the cone $\ch\dram$ having finite volume.)
One can go through the Coxeter diagrams in
\cite[Fig.~1]{alexeev2006del-pezzo} and Figs.~\ref{fig:coxeter1},
\ref{fig:coxeter2} to verify that this happens in rather few cases.
Heuristically, it is because there are many irrelevant nodes, as these are
the complement of the $2k$-cycle in Figs.~\ref{fig:coxeter1}, \ref{fig:coxeter2}.
\end{remark}

\section{Example. $S=(2,2,0)$: hyperelliptic K3 surfaces of degree~$4$}
\label{sec:2-0-1}

Consider the moduli space $F_S$ of K3 surfaces with the generic Picard
lattice $S=(2,2,0)_1= U(2)$. The K3 surfaces
in this family are double covers of $\bP^1\times\bP^1$ branched in a
divisor $B\in |\cO(4,4)|$. The pullback
$L=\pi^*\cO_{\bP^1\times\bP^1}(1,1)$ has degree $L^2=4$. These
surfaces are known as hyperelliptic K3 surfaces of degree~$4$. 

\begin{figure}[htbp]
  \includegraphics[width=5 in]{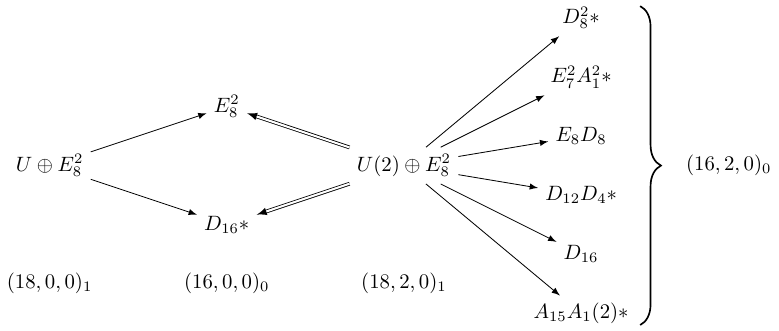}
  \smallskip
  \caption{Cusps for $T=(20,2,0)_2=U\oplus U(2)\oplus E_8^2  = U^2\oplus D_{16}$}
  \label{fig-18-allcusps}
\end{figure}

The generic transcendental lattice is $T=S^\perp=(20,2,0)_2$. Explicitly,
one has $T\simeq U\oplus U(2)\oplus E_8^2\simeq U^2\oplus  D_{16}$. The $0$-cusps of
$F_S$ are found by Theorem~\ref{thm:0-cusps}. In \fignb there
are two mirror moves: the odd
$S=(2,2,0)\to \oT=(18,2,0)\simeq U(2)\oplus E_8^2\simeq U\oplus D_{16}$ and the even
ordinary $S=(2,2,0)\dto\oT=(18,0,0) \simeq U\oplus E_8^2$. The $1$-cusps are given
by Theorem~\ref{thm:1-cusps}. They correspond to the negative definite
lattices $\ooT$ with the invariants $(16,0,0)_0$ (i.e. unimodular) and
$(16,2,0)_0$. These are listed in Table~\ref{tab:neg-def}. We give the
complete cusp diagram in Fig.~\ref{fig-18-allcusps}. (It can also be
found in \cite{laza2021git-versus}.) By 
Proposition~\ref{prop:1-cusps-modular}, the $1$-cusps for the
$\ooT(16,2,0)$ lattices are isomorphic to $\bH/\SL(2,\bZ)$, and those
for the $\ooT(16,0,0)$ lattices to $\bH/\Gamma_0(2)$. 

\begin{figure}[htbp]
  \includegraphics[width=5 in]{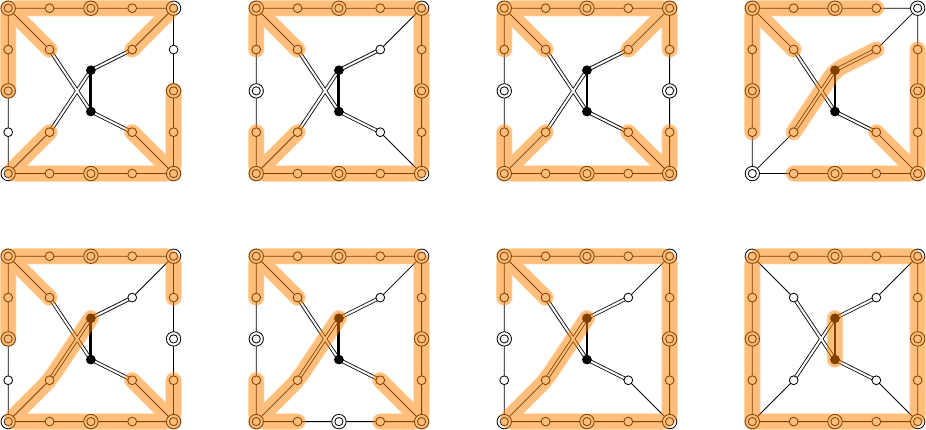}
  \caption{Maximal parabolic subdiagrams in $\Gamma_r$ for $\oT=(18,2,0)$}
  \label{fig-18-1cusps}
\end{figure}  

\begin{figure}[htbp]
  \includegraphics[width=3.5 in]{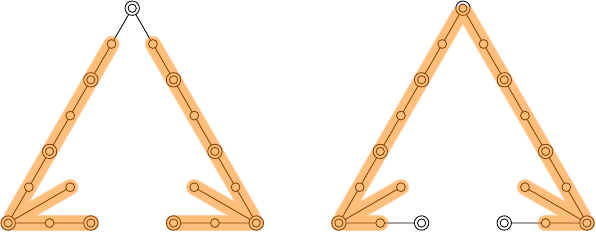}
  \smallskip
  \caption{Maximal parabolic subdiagrams in $\Gamma_r$ for $\oT=(18,0,0)$}
  \label{fig-19-1cusps}
\end{figure} 

Vinberg diagrams for $\oT=(18,2,0)$ and $\oT=(19,1,1)$ are given in 
Fig.~\ref{fig:coxeter2}. The diagram for $\oT=(18,0,0)$ is obtained from
that of $(19,1,1)$ by the procedure given in
Lemma~\ref{lem:heegner-diagrams}; the result is the diagram of
Fig.~\ref{fig-19-1cusps} (without the highlighting).
The $1$-cusps containing a given
$0$-cusp correspond to the maximal parabolic subdiagrams in the
Vinberg diagram. They are highlighted in Figs.~\ref{fig-18-1cusps} and
\ref{fig-19-1cusps}; the order of appearance is the same as in
Fig.~\ref{fig-18-allcusps}. 

\smallskip

The Kulikov models are as described in
Section~\ref{sec:kulikov-models-all}: for $\oT=(18,2,0)$ in
Section~\ref{sec:models-1} and for $\oT=(18,0,0)$ in
Section~\ref{sec:models-remaining}. The stable models are given by
Theorem~\ref{thm:stable-models}. 

At the $(18,2,0)$ odd $0$-cusp the Type III stable models are of
pumpkin type of Fig.~\ref{fig:pumpkin}. They correspond to the
elliptic subdiagrams, i.e. disjoint unions of $ADE$ graphs in
Fig.~\ref{fig-18-1cusps}. The maximal degeneration corresponds to the
empty subdiagram, or $(A^-_0 \mA_0)^8$.
It is a union of $16$
pairs $(\bP^2,D_1+D_2)$ with a line $D_1$ and a conic $D_2$, as in
Fig.~\ref{fig:pumpkin}. The list of irreducible components that
appear, in notations of \cite[Table~2]{alexeev17ade-surfaces}, is as
follows. We list them according to the number of the off-equator
vertices. For brevity, $m$ denotes an even index and $n$ an odd index.

\begin{enumerate} 
\item[($0$)] $A_n$ ($1\le n\le 15$) \quad $A_m^-$ ($0\le m\le 14$)
  \quad $\mA_n^-$ ($1\le n\le 15$)
\item[($1$)] $\pA_m^-$ ($2\le m\le 16$) \quad
  $\pA_n$ ($3\le n\le 15$)\quad $D_m$ ($4\le m\le 16$) \\
  $D^-_n$ ($5\le n\le 15$)\quad
  $\mE_6^-$\quad $\mE_7$\quad $\mE_8^-$
\item[($1+1$)] $\pA'_n$ ($n=7,11,15$)\quad $D'_n$ ($n=9,13,17$)
\item[($2$)] $\pA'_3$\quad $\pD_m$ ($4\le m\le 16$)\quad
  $\pD_n^-$ ($5\le n\le 15$)
\item[($2+1$)] $\pD'_m$ ($m=8,12,16$)
\item[($3$)] $\pD'_4$
\end{enumerate}

At the $(18,0,0)$ even ordinary $0$-cusp the Type III stable models
are of smashed pumpkin type of Fig~\ref{fig:pumpkin}. 
They
correspond to the elliptic subdiagrams in Fig.~\ref{fig-19-1cusps}. In
this case some of the irreducible components $\oV_i$ are non-normal,
with a long side folded $2:1$. As in \cite[Def.~6.9]{alexeev17ade-surfaces},
we denote these by adding
the letter $f$ to the Dynkin symbol.  The maximal
degeneration corresponds to the empty subdiagram, or
$\foA_0^- (\mA_0 A_0^-)^7 \mA_0^f$.  It is a union of $16$ pairs
$(\bP^2,D_1+D_2)$ but in the end two of these pairs the $\bP^2$ is
glued to itself along the conic $D_2$ which is folded $2:1$.  The list
of irreducible components is as follows:

\begin{enumerate} 
\item[($0$)] $\plA_n^-$ ($1\le n\le 15$) \quad $\plA_{17}^+$ \quad
  $\plA_m$ ($2\le m\le 14)$) \quad $\foA_{16}^+$ \\
  $\foA_n$ $(1\le n\le 13)$ \quad $\foA^f_{15}$ \quad
  $\foA_m^-\ (0\le m\le 14)$ \\
  $\mA_n^-\ (1\le n\le 13)$ \quad $\mA_m\ (0\le m\le 12)$
  \quad $A_n\ (1\le n\le 11)$.   
\item[($1$)] $\pA^+_{4}$ \quad $\pA^+_{16}$ \quad $\foA'_3$ \quad
  $\foA'_{15}$ \quad $\pA_m\ (2\le m\le 14)$ \quad
  $\pA_n\ (3\le n\le 13)$\\
  $D_5^+$\quad $\foD^+_{17}$ \quad $\foD_m\ (4\le m\le 14)$ \quad
  $\foD_n \ (5\le n\le 15)$ \quad $\plE_6^-$ \quad
  $\plE_7$ \quad $\plE_8^-$
\item[($1+1$)] $\pA'_{15}$
\end{enumerate}

\smallskip

The Type II stable models are described by the maximal parabolic
subdiagrams of Figs.~\ref{fig-18-1cusps}, \ref{fig-19-1cusps}. The
irreducible components correspond to the relevant connected
components.

\medskip

The normalization of $\oF_S$ is semitoroidal.
At the $(18,2,0)$ $0$-cusp, the semifan $\fF\dram(18,2,0)$ is a
coarsening of the Coxeter semifan $\fref(18,2,0)$ and neither of them
is a fan. At the $(18,0,0)$ $0$-cusp, a maximal-dimensional cone of the
semifan $\fF\dram(18,0,0)$ is a union of $4$ maximal-dimensional cones
of $\fref(18,0,0)=\fF_2(18,0,0)$, and both of them are fans,
so near this cusp the normalization of $\oF_S$ is toroidal.

\bibliographystyle{amsalpha}

\def\cprime{$'$}
\providecommand{\bysame}{\leavevmode\hbox to3em{\hrulefill}\thinspace}
\providecommand{\MR}{\relax\ifhmode\unskip\space\fi MR }
\providecommand{\MRhref}[2]{%
  \href{http://www.ams.org/mathscinet-getitem?mr=#1}{#2}
}
\providecommand{\href}[2]{#2}


\ifshort 
\else
\listoffigures 
\listoftables 
\fi

\end{document}